\documentclass[11pt]{amsart}

\usepackage{epsf,amssymb,overpic,amscd,psfrag,pstool, stmaryrd,times,bbold,mathrsfs}
\usepackage[all]{xy}

\usepackage{epsfig}
\usepackage{epstopdf}

\newtheorem{thm}{Theorem}[section]
\newtheorem{prop}[thm]{Proposition}
\newtheorem{lemma}[thm]{Lemma}

\newtheorem{claim}[thm]{Claim}

\numberwithin{equation}{subsection}
\numberwithin{thm}{subsection}

\theoremstyle{definition}
\newtheorem{defn}[thm]{Definition}

\newtheorem{notation}[thm]{Notation}

\theoremstyle{remark}
\newtheorem{q}[thm]{Question}
\newtheorem{rmk}[thm]{Remark}

\DeclareMathAlphabet{\mathpzc}{OT1}{pzc}{m}{it}

\renewcommand{\H}{\mathbb{H}}
\newcommand{\C}{\mathbb{C}}

\newcommand{\R}{\mathbb{R}}
\newcommand{\Z}{\mathbb{Z}}
\newcommand{\Q}{\mathbb{Q}}

\newcommand{\bdry}{\partial}
\newcommand{\s}{\vskip.1in}
\newcommand{\n}{\noindent}

\newcommand{\F}{\mathbb{F}}

\newcommand{\Kh}{Kh^\sharp}

\newcommand{\be}{\begin{enumerate}}
\newcommand{\ee}{\end{enumerate}}
\newcommand{\op}{\operatorname}
\newcommand{\bs}{\boldsymbol}

\newcommand{\wt}{\widetilde}

\newcommand{\ai}{A_{\infty}}

\begin{document}
	
\title[Higher-dimensional Heegaard Floer homology]
{Applications of higher-dimensional Heegaard Floer homology to contact topology}

\author{Vincent Colin}
\address{Universit\'e de Nantes, 44322 Nantes, France}
\email{vincent.colin@univ-nantes.fr}

\author{Ko Honda}
\address{University of California, Los Angeles, Los Angeles, CA 90095}
\email{honda@math.ucla.edu} \urladdr{http://www.math.ucla.edu/\char126 honda}

\author{Yin Tian}
\address{Yau Mathematical Sciences Center, Tsinghua University, Beijing 100084, China}
\email{yintian@mail.tsinghua.edu.cn}

\date{\today}



\thanks{VC supported by ERC Geodycon and ANR Quantact.  KH supported by NSF Grants DMS-1406564 and DMS-1549147. YT supported by NSFC 11601256 and 11971256.}

\begin{abstract}
The goal of this paper is to set up the general framework of higher-dimensional Heegaard Floer homology, define the contact class, and use it to give an obstruction to the Liouville fillability of a contact manifold and a sufficient condition for the Weinstein conjecture to hold.  We discuss several classes of examples including those coming from analyzing a close cousin of symplectic Khovanov homology and the analog of the Plamenevskaya invariant of transverse links.
\end{abstract}

\maketitle

\setcounter{tocdepth}{1}
\tableofcontents

\section{Introduction} \label{section: intro}

Let $(W,\beta,\phi)$ be a Weinstein domain of dimension $2n$, where $\beta$ is the Liouville $1$-form and $\phi$ is the compatible Morse function, and let $h\in \op{Symp}(W,\bdry W,d\beta)$ be a symplectomorphism on $W$ that restricts to the identity on $\bdry W$.

The goal of this paper is to define the hat version $\widehat{HF}(W,\beta,\phi;h)$ of the higher-dimensional Heegaard Floer homology groups and give some applications.  

\subsection{Higher-dimensional Heegaard Floer homology groups and the contact class}

The situation can be summarized in the following diagram:
$$\begin{CD}
\wt{W} @>{completion}>> \wt{W}^\wedge @>{cylinder}>> \widehat{X}=\widehat{X}_{\wt{W}}= \R\times[0,1]\times \wt{W}^\wedge\\
@A{capping}AA  \\
W @> {completion}>> \widehat{W} @> {cylinder} >> \widehat{X}_{W}= \R\times[0,1]\times \widehat W.
\end{CD}$$
Here $\wt{W}$ is obtained by capping off the pairwise disjoint Lagrangian disks $a_1,\dots, a_\kappa$ that are the unstable submanifolds of the critical handles of $W$ (i.e., a ``Lagrangian basis") and $\wt{W}^\wedge$ and $\widehat W$ are completions of $\wt{W}$ and $W$. Let $\wt a_1,\dots,\wt a_\kappa$ be Lagrangian spheres obtained by capping off $a_1,\dots,a_\kappa$. The top row can be viewed as working on the Heegaard surface and the bottom row can be viewed as working on the page of an open book decomposition.

The group $\widehat{HF}(W,\beta,\phi;h)$ is defined using $\widehat X$, an auxiliary almost complex structure $J^\Diamond$, and the collection $\{\wt a_1,\dots, \wt a_\kappa\}$, in a manner analogous to Lipshitz' cylindrical reformulation of Heegaard Floer homology \cite{Li}; this will be done in Section~\ref{section: defn of homology groups}.  In Section~\ref{section: reformulation of HF} we present an equivalent description in terms of $\widehat X_W$ which is analogous to the description of Heegaard Floer homology in terms of an open book decomposition due to Honda-Kazez-Mati\'c~\cite{HKM2} and uses $\{a_1,\dots, a_\kappa\}$.

The higher-dimensional Heegaard Floer homology groups satisfy the following:

\begin{thm} \label{thm: invariance or lack thereof}
$\widehat{HF}(W,\beta,\phi;h)$ is invariant under:
\be
\item[(I1)] trivial Weinstein homotopies;
\item[(I2)] changes of almost complex structure $J^\Diamond$; and
\item[(I3)] isotopies of $h$ in $\op{Symp}(W,\bdry W,d\beta)$.
\ee
\end{thm}

This is proven in Section~\ref{section: partial invariance}.  

\begin{rmk}
The higher-dimensional Heegaard Floer groups are not expected to be invariant under handleslides (I4), although they are in some cases (cf.\ Theorem~\ref{thm: Kh is invariant of link} below).
\end{rmk}

In Section~\ref{section: defn of contact class} we define the {\em $p$-twisted contact class}
$$c^p(W,\beta,\phi;h)\in \widehat{HF} (W,\beta,\phi ;h)$$ for $p\in \Z$ corresponding to the open book decomposition $(W,\beta,\phi;h)$.  Although handleslide invariance does not hold, the contact class can nevertheless shed light on the corresponding contact structure $(M,\xi)$ as follows:

\begin{thm}\label{thm: non-vanishing}
If the contact class 
$$c^p (W,\beta,\phi ;h) \in \widehat{HF} (W,\beta,\phi ;h)$$ 
vanishes for an open book $(W,\beta,\phi ;h)$ supporting $(M,\xi)$, then 
\be
\item $(M,\xi)$ is not Liouville fillable; and
\item $(M,\xi)$ satisfies the Weinstein conjecture.
\ee
\end{thm}

In particular the theorem provides a convenient method for verifying that a contact structure is {\em not} Liouville fillable. In Section~\ref{subsection: defn of right-veering} we define higher-dimensional analogs of right-veering surface diffeomorphisms (called {\em strongly right-veering} and {\em weakly right-veering} symplectomorphisms, where strongly right-veering implies weakly right-veering) and prove the following:

\begin{thm}\label{thm: not right-veering}
	A contact structure $(M,\xi)$ supported by an open book decomposition whose pages are Liouville and whose monodromy $h$ is not strongly right-veering is not Liouville fillable and satisfies the Weinstein conjecture.
\end{thm}

Let $(W,\beta,\phi; h)$ be an open book decomposition for $(M,\xi)$.  It is unknown whether $(M,\xi)$ is overtwisted if $h$ is not strongly right-veering.  It is even not known whether $(M,\xi)$ is overtwisted whenever $h$ is a product of negative symplectic Dehn twists. For example, it is currently not known how to use the techniques of Casals-Murphy-Presas~\cite{CMP} or the overtwisted oranges of \cite{HH} to prove such contact structures are overtwisted.

\subsection{Variant of symplectic Khovanov homology}

Another class of non-Liouville fillable examples comes from analyzing a close cousin of symplectic Khovanov homology and the analog of the Plamenevskaya invariant of transverse links~\cite{Pl}. 

Let $W=W_{\kappa-1}$ be the $4$-dimensional Milnor fiber of the $A_{\kappa-1}$ singularity and let $\wt W=W_{2\kappa-1}$ be the Milnor fiber of the $A_{2\kappa-1}$ singularity obtained by capping off the $\kappa$ Lagrangian thimbles $a_1,\dots,a_\kappa$ emanating from the $\kappa$ critical points. (There are $\kappa$ thimbles as opposed to a Lagrangian basis with $\kappa-1$ elements, but technically there is no difference.) Let $h_\sigma$ be the monodromy on $W$ which descends to a braid $\sigma$ on the $2$-disk $D^2$ with $\kappa$ singular points via the standard Lefschetz fibration $\pi: W\to D^2$. Also let $\widehat\sigma$ be the braid closure of $\sigma$.  

The variant of symplectic Khovanov homology of interest is the $\F[\mathcal{A}]\llbracket\hbar,\hbar^{-1}]$-module 
$$\Kh (\widehat\sigma):=\widehat{HF}(W;h_\sigma),$$ 
where $\F$ is a field (e.g., $\F=\Z/2$ or $\Q$), $\F[\mathcal{A}]$ is a group ring over the abelian group $\mathcal{A}$ described in Section~\ref{subsection: symplectic Kh definitions}, and we are using the $\kappa$ thimbles $a_1,\dots,a_\kappa$.

\begin{thm} \label{thm: Kh is invariant of link}
	$\Kh (\widehat \sigma)$ is a link invariant, i.e., is independent of the choice of thimbles $\{a_1,\dots,a_\kappa\}$ and invariant under positive and negative Markov stabilizations.
\end{thm}

We will briefly discuss the relationship between $\Kh(\widehat \sigma)$ and $Kh_{\op{symp}}(\widehat\sigma)$ in Section~\ref{subsection: spectral sequence}, where $Kh_{\op{symp}}(\widehat\sigma)$ is the cylindrical formulation of symplectic Khovanov homology which was shown by Mak-Smith~\cite{MS} to be the same as the original Seidel-Smith definition~\cite{SS}.
In a forthcoming paper, we will investigate a generalization of $\Kh(\widehat \sigma)$ and $Kh_{\op{symp}}(\widehat\sigma)$ to braids $\sigma$ in surface bundles over $S^1$.

Next we consider the contact class $c^0(W;h_\sigma)\in  \Kh (\widehat \sigma)=\widehat{HF}(W;h_\sigma)$.

\begin{thm} \label{thm: invariance of contact class in Kh}
The contact class $c^0(W;h_\sigma)$ is an invariant of $\sigma$ up to positive Markov stabilizations.
\end{thm} 

The contact class is denoted $\psi^\sharp(\widehat\sigma)$ by analogy with the Plamenevskaya invariant $\psi(\widehat\sigma)$ of a transverse link \cite{Pl}.

There is a large literature devoted to the Plamenevskaya transverse link invariant. In particular, Plamenevskaya~\cite{Pl} showed that:

\begin{thm}[Plamenevskaya~\cite{Pl}]
If the braid closure $\widehat \sigma$, viewed as a transverse link in the standard contact $\R^3$, admits a braid representative which is not right-veering, then $\psi(\widehat \sigma)=0$.
\end{thm}

There are also many examples of braids $\sigma$ satisfying the following:
\be
\item[(B1)] $\sigma$ is right-veering;
\item[(B2)] the invariant $\psi(\widehat \sigma)$ vanishes; and
\item[(B3)] the open book corresponding to $\sigma$ lifted to the branched double cover carries a tight contact structure.
\ee
The first examples are due to Baldwin-Plamenevskaya~\cite[Example~7.8]{BP} and there are others e.g., due to Hubbard-Lee~\cite{HL}.

In this paper we analyze the Baldwin-Planenevskaya examples $\sigma_{BP,\ell}$, which are $3$-braids of the form $\sigma_1^{-\ell}\sigma_2\sigma_1^2\sigma_2$, where $\ell\geq 3$ and $\sigma_i$ is a half twist about strands $i$ and $i+1$.  Let $h_{BP,\ell}:=h_{\sigma_{BP,\ell}}\in \op{Symp}(W,\bdry W)$ be the corresponding symplectomorphism.  In Section~\ref{section: BP} we prove the following:

\begin{thm}  \label{thm: BP}
	$\psi^\sharp(\widehat\sigma_{BP,\ell})=0$ for $\ell\ge 3$.  Hence the contact manifolds supported by $(W;h_{BP,\ell})$ are not Liouville fillable and satisfy the Weinstein conjecture. 
\end{thm}

Note that by Acu~\cite{Acu}, since $W=W_{\kappa-1}$ is planar, all contact structures supported by an open book decomposition with page $W_{\kappa-1}$ satisfy the Weinstein conjecture, so the latter result is not new. It should be noted that we have not determined whether the contact manifolds supported by $(W;h_{BP,\ell})$ are tight.

\begin{q}
Is $(W;h_\sigma)$ tight if $\sigma$ satisfies (B1) and (B3)?
\end{q}

\s\n
{\em Acknowledgments.} VC thanks Baptiste Chantraine and Paolo Ghiggini for their interest and support. KH is grateful to Yi Ni and the Caltech Mathematics Department for their hospitality during his sabbatical in 2018. KH also thanks John Baldwin and Otto van Koert for helpful correspondence and Tianyu Yuan for pointing out some errors.

\section{Preliminaries} \label{section: preliminaries}

\subsection{Liouville and Weinstein cobordisms and homotopies} \label{subsection: Weinstein cobordisms}

In this subsection we review Liouville and Weinstein cobordisms and homotopies. The reader is referred to \cite{CE} for more details.

\begin{defn}
A {\em Liouville cobordism} $(W,\beta,Y)$ is a triple consisting of:
\begin{itemize}
\item a compact $2n$-dimensional manifold $W$ with boundary
$$\bdry W=\bdry_+W \sqcup \bdry_-W;$$
\item a $1$-form $\beta$ on $W$ such that $d\beta$ is symplectic; and
\item a vector field $Y$ which satisfies $i_Y d\beta=\beta$, is positively transverse to $\bdry_+W$, and is negatively transverse to $\bdry_-W$.
\end{itemize}
The $1$-form $\beta$ is called the {\em Liouville $1$-form} and the vector field $Y$ the {\em Liouville vector field}.
A {\em Liouville domain} $(W,\beta,Y)$ is a Liouville cobordism with $\bdry_- W=\varnothing$.
\end{defn}

We will sometimes omit $Y$ or write $(W,d\beta)$.

\begin{defn} $\mbox{}$
\begin{enumerate}
\item A {\em Weinstein cobordism} $(W,\beta,Y,\phi)$ is quadruple consisting of Liouville cobordism $(W,\beta,Y)$ and a Morse function $\phi:W\to\R$ (called a {\em compatible Morse function}) such that $Y$ is gradient-like for $\phi$ and $\partial W$ is a union of regular level sets of $\phi$.
\item A {\em Weinstein domain} $(W,\beta,Y,\phi)$ is a Weinstein cobordism with $\bdry_-W=\varnothing$.
\item If the critical points of $\phi:W\to\R$ are Morse or birth-death type, then $(W,\beta,Y,\phi)$ is a {\em generalized Weinstein cobordism/domain}, as appropriate.
\end{enumerate}
\end{defn}

{\em In this paper we assume that Weinstein domains are connected, unless stated otherwise.}

\begin{defn}  A {\em trivial Weinstein cobordism} is
$$(W=[a,b]\times Z^{2n-1}, \beta=e^\sigma \beta_0,Y=\bdry_\sigma,\phi(\sigma,x)=\sigma),$$
where $(\sigma,x)$ are coordinates on $[a,b]\times Z$ and $\beta_0$ is a contact form on $Z$.
\end{defn}

Given a Weinstein domain $(W,\beta,Y,\phi)$ with connected boundary, there exists a collar neighborhood $[-\varepsilon+\phi(\bdry W),\phi(\bdry W)]\times \bdry W$ which is a trivial Weinstein cobordism, i.e., $\beta=e^\sigma\beta|_{\bdry W}$, $Y=\bdry_\sigma$, and $\phi(\sigma,x)=\sigma$.

\begin{defn}
Let $W$ be a compact, oriented $2n$-dimensional manifold.
\begin{enumerate}
\item A {\em Liouville homotopy} on $W$ is a $1$-parameter family of Liouville domain structures $(\beta_t,Y_t)$, $t\in[0,1]$, on $W$.
\item A {\em Weinstein homotopy} on $W$ is a $1$-parameter family of generalized Weinstein domain structures $(\beta_t,Y_t,\phi_t)$, $t\in[0,1]$, on $W$.
\end{enumerate}
\end{defn}

\begin{defn} $\mbox{}$
\begin{enumerate}
\item The {\em completion $(\widehat{W} ,\widehat{\beta})$} of a Liouville domain $(W, \beta )$ is obtained by gluing $([0,\infty) \times \partial W, e^\sigma \beta \vert_{\partial W} )$ to $(W,\beta)$ along $\partial W$. Here $\sigma$ denotes the $[0,\infty)$-coordinate on $[0,\infty) \times \partial W$.
\item The {\em completion $(\widehat{W} ,\widehat{\beta},\widehat\phi)$} of a Weinstein domain $(W,\beta,\phi)$ additionally requires that the extension $\widehat{\phi}$ of $\phi$ be proper and have strictly positive derivative with respect to the $\sigma$-coordinate.
\end{enumerate}
\end{defn}

Homotopies of Liouville or Weinstein domain can easily be extended to Liouville or Weinstein homotopies
of their completions.

\begin{defn}
Let $[(W,\beta,\phi)]$ be the equivalence class of all Weinstein structures $(W,\beta',\phi')$ that are Weinstein homotopic to a Weinstein domain $(W,\beta,\phi)$. (In particular we are assuming that $\phi'$ is a Morse function.)
The {\em genus} of $[(W,\beta,\phi)]$, denoted by $g(W,\beta,\phi)$ or $g([(W,\beta,\phi)])$, is $1/2$ of the minimum number of critical (= $n$-dimensional) handles of $\phi'$ over all $(W,\beta',\phi')\in [(W,\beta,\phi)]$.
\end{defn}

\begin{defn} $\mbox{}$
\begin{enumerate}
\item A Weinstein structure $(W,\beta,Y,\phi)$ is {\em generic} if, for all pairs $(p,q)$ of critical points of $\phi$, the ascending submanifold of $p$ and the descending submanifold of $q$, both with respect to $Y$, intersect transversely.
\item A Weinstein homotopy $(W,\beta_t,Y_t,\phi_t)$, $t\in[0,1]$, is {\em trivial} if $\phi_t$ is Morse and $(W,\beta_t,Y_t,\phi_t)$ is generic for all $t\in[0,1]$.
\end{enumerate}
\end{defn}

In particular, if $(W,\beta,\phi)$ is generic, then there are no trajectories of $Y$ between two critical points of index $n$.

Let $(W,\beta,\phi)\in [(W,\beta,\phi)]$ and $g=g([W,\beta,\phi])$. If $(W,\beta,\phi)$ is generic, then $\phi$ has $\kappa\geq 2g$ critical handles and there is a collection $\{a_1,\dots,a_\kappa\}$ of properly embedded, pairwise disjoint Lagrangian disks with Legendrian boundary on $\bdry W$ that represent the cocores of the $\kappa$ critical handles and that are the $n$-dimensional unstable manifolds of the Liouville vector field $Y$. We will refer to $\{a_1,\dots,a_\kappa\}$ as the {\em basis of Lagrangian disks for $(W,\beta,\phi)$.}

We also state the following useful lemma (cf.\  \cite[Proposition 11.8]{CE}), which is an easy consequence of the Moser technique:

\begin{lemma} \label{lemma: Liouville homotopy}
Let $(W,\beta_t)$, $t\in[0,1]$, be a Liouville homotopy.
\begin{enumerate}
\item If $\beta_t=\beta_0$, $t\in[0,1]$, on a neighborhood of $\bdry W$, then there exists a $1$-parameter family of diffeomorphisms $g_t:W\stackrel\sim\to W$, $t\in[0,1]$, with $g_0=\op{id}$ and $g_t|_{\bdry W}=\op{id}$ such that $g_t^* \beta_t-\beta_0$ is exact for all $t\in[0,1]$.
\item If we do not assume that $\beta_t$ is fixed along $\bdry W$, then there exists a $1$-parameter family of diffeomorphisms $\widehat{g}_t$ between the completions $(\widehat{W} , \widehat{\beta_0} )$ and $(\widehat{W},\widehat{\beta_t})$ such that $\widehat{g}_t^* \widehat\beta_t-\widehat\beta_0$ is exact for all $t\in[0,1]$ and is equal to zero outside a compact set.
\end{enumerate}
\end{lemma}

Lemma~\ref{lemma: Liouville homotopy}(2) says that homotopic Liouville domains have isomorphic completions.

Finally, if $(W,\beta)$ is a Liouville domain, then we denote the group of symplectomorphisms of $(W,d\beta)$ that restrict to the identity on $\bdry W$ by $\op{Symp}(W,\bdry W,d\beta)$.

\subsection{Capping of $(W,\beta,\phi;h)$}\label{subsection: extension}

In this subsection we explain how to extend a Weinstein domain $(W,\beta,\phi)$ to a Weinstein domain $(\wt W=H\cup W,\wt \beta,\wt \phi)$, called the {\em capping of $(W,\beta,\phi)$}, such that $\bdry_-H=\bdry_+ W$.

Let $N'(\bdry_+W)$ be a collar neighborhood $[-\varepsilon,0]\times \bdry_+ W\subset W$ of $\bdry_+ W=\{0\}\times \bdry_+ W$ with coordinates $(\sigma,x)$ such that
$$\phi:[-\varepsilon,0]\times \bdry W\to \R$$
is given by $(\sigma,x)\mapsto \sigma$ and $\beta=e^\sigma\beta|_{\bdry W}$, i.e., the Weinstein cobordism is trivial. The manifold $H$, called a {\em cap for $(W,\beta,\phi)$}, is obtained from the trivial Weinstein cobordism $N(\bdry_+W):=[0,\varepsilon]\times \bdry_+ W$ by attaching $\kappa$ critical handles $H_1,\dots,H_\kappa$ along $\{\varepsilon\}\times\bdry a_1,\dots, \{\varepsilon\}\times \bdry a_\kappa$. The pair $(\wt\beta|_H,\wt\phi|_H)$ is chosen so that the $Y$-descending submanifolds of the critical points of $\wt\phi|_H$ have boundary $\bdry a_1,\dots,\bdry a_\kappa$, where $Y$ is the Liouville vector field of $\wt\beta$. We write $\wt a_1,\dots,\wt a_\kappa\subset \wt W$ (called the {\em capping of $a_1,\dots, a_\kappa$}) for the pairwise disjoint embedded Lagrangian spheres obtained by capping off $a_1,\dots,a_\kappa$ and write $\wt {\bs a}=\wt a_1\cup\dots\cup \wt a_\kappa$.

Let $f:\wt W\to \R$ be a smooth function which satisfies the following:
\begin{enumerate}
\item $f$ has small $C^1$-norm;
\item $f$ is supported on a small neighborhood of $\wt {\bs a}$ which is identified with a neighborhood $U$ of the $0$-section of $\pi:T^*\wt{\bs a}\to \wt{\bs a} $; let $U_i$ be the component of $U$ which is a neighborhood of $\wt a_i$;
\item $f=\pi^* f_i$ on a smaller neighborhood $V_i\subset U_i$ of $\wt a_i$, where $f_i:\wt a_i\to\R$ is a Morse function with two critical points: a maximum on $\op{int}(a_i)$ and a minimum $x_i$ on $\wt a_i -a_i$;
\item $f=\sigma$ when restricted to $V_i\cap([-\varepsilon/2,\varepsilon/2]\times \bdry W)$.
\end{enumerate}
Let $X_f$ be the Hamiltonian vector field of $f$ given by $df=i_{X_f} d\wt \beta$.  We then modify $h\cup \op{id}|_H\in \op{Symp}(\wt W,\bdry \wt W,d\wt \beta)$ by composing with the Hamiltonian diffeomorphism which is the time-$1$ flow along $X_f$.  The result of the composition is written as $\wt h\in \op{Symp}(\wt W,\bdry \wt W,d\wt \beta)$.  Note that (4) implies that $\wt h(a_i)\subset W$ and that $\wt h(\bdry a_i)$ is a positive pushoff of $\bdry a_i$ with respect to the Reeb vector field for $\beta|_{\bdry W}$.  Also view the minimum $x_i$ of $f_i$ as an intersection point of $\wt a_i$ and $\wt h(\wt a_i)$.

See Figure~\ref{fig: handle-new} for the handle $H_i$, the Lagrangian submanifolds $\wt a_i$ and $\wt{h}(\wt a_i)$, and the intersection point $x_i$.

\begin{figure}[ht]
	\begin{overpic}[scale=0.6]{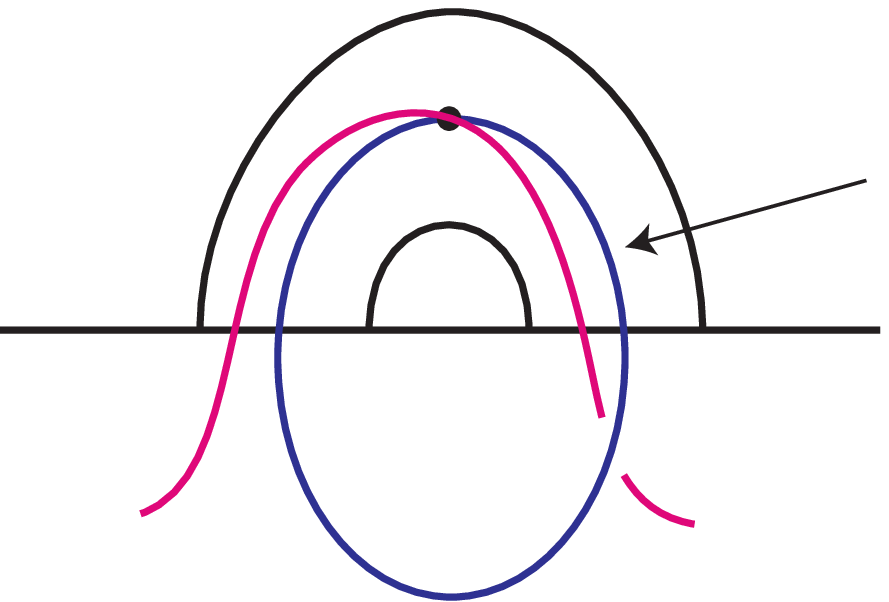}
		\put(49,3){\tiny $a_i$}
		\put(100,48){\tiny $\wt a_i-a_i$}
		\put(8,15){\tiny $\wt{h}(\wt a_i)$}
		\put(49.2,57.3){\tiny $x_i$}
		\put(70,59){\tiny $H_i$}
		\put(84,18){\tiny $W$}
	\end{overpic}
	\caption{}
	\label{fig: handle-new}
\end{figure}

\begin{rmk} \label{rmk: spheres}
The Lagrangian submanifolds $\wt a_i$ and $\wt{h}(\wt a_i)$ are automatically exact, since they are spheres of dimension $\geq 2$.
\end{rmk}

\begin{rmk}
We make a local calculation to clarify a potential sign confusion. Let $q_1,\dots,q_n,p_1,\dots,p_n$ be local coordinates on $\wt W$ near
$$\wt a_i=\{p_1=\dots=p_n=0\}$$
such that $\beta=-\sum_j p_jdq_j$ and let $f_i: \wt a_i\to \R$ be given by $f_i(q)=\tfrac{\varepsilon}{2} \sum_j q_j^2$ near the minimum $x_i=\{q_1=\dots=q_n=0\}\in \wt a_i$. Then $df_i=\varepsilon\sum_j q_j dq_j$. Since $d\wt \beta=\sum_j dq_j dp_j$, we have $X_{f_i}= -\varepsilon\sum_i q_i \bdry_{p_i}$.  This means that the graph of $-df_i$ (not $df_i$) is obtained from $\wt a_i$ by flowing in the direction of $X_{f_i}$. Hence the point $x_i$, viewed as an intersection point of the Lagrangian Floer cohomology group $\widehat{CF}(\wt{h}(\wt a_i),\wt a_i)$, is a ``top generator" that we call a component of the {\em contact class}; see Section~\ref{section: defn of contact class} below.
\end{rmk}

\subsection{Open book decompositions and contact structures} \label{subsection: open books}

In this subsection we review the foundational work of Giroux~\cite{Gi1} and Giroux-Mohsen~\cite{GM,Mo} relating contact structures and open book decompositions.

\subsubsection{Open book decompositions}

An {\it open book decomposition} of a closed manifold $M$ consists of a pair $(K,\theta )$ where:
\begin{itemize}
\item $K$ is a codimension two submanifold of $M$ with trivial normal bundle;
\item $\theta : M\setminus K \to S^1$ is a fibration equal to the normal angular coordinate in a trivialized neighborhood
$N(K) =K\times D^2$ of $K$.
\end{itemize}
The submanifold $K$ is the {\it binding} of the open book decomposition. The compactification of any fiber of $\theta$ by $K$ is called a {\it page}.

An open book decomposition can equivalently be described by a pair $(W,h)$, where $W$ is a page ($\partial W\neq \varnothing$) and $h\in \op{Diff}(W,\bdry W)$, the set of diffeomorphisms of $W$ that restrict to the identity on $\bdry W$. The manifold $M$ is then obtained as the {\it relative mapping torus} of $(W,h)$, i.e.,
$$M\simeq W\times [0,1]/\sim_h ,$$
where $\sim_h$ identifies $(h(x),0) \sim_h (x,1)$ for all $x\in W$ and $(y,t)\sim_h (y,t')$ for all $y\in \partial W$, $t,t' \in [0,1]$.

\subsubsection{General case}

Given $(W,\beta;h)$ with $h\in \op{Symp}(W,\bdry W,d\beta)$, we construct a relative mapping torus $M= M_{(W,\beta;h)}$ (which is slightly different from the one above) and a contact form $\alpha_{(W,\beta;h)}$ on $M$ as follows:
First construct the mapping torus
$$N_{(W,\beta;h)}= ([0,1]\times W)/\sim_0,$$
where $(t,x)$ are coordinates on $[0,1]\times W$ and $(1,x)\sim_0 (0,h(x))$ for all $x\in W$. 
We then let
$$M=M_{(W,\beta;h)}= (\bdry W\times D^2)\cup N_{(W,\beta;h)}/\sim_1,$$
where $\sim_1$ identifies $\bdry W\times S^1$ with $\bdry N_{(W,\beta;h)}$ via $(x,\theta)\sim_1 (\theta/ 2\pi,x)$.  Here $(r,\theta)$ are polar coordinates on $D^2$ and $S^1=\bdry D^2$.  By Lemma~\ref{lemma: Liouville homotopy} there exists a contact form $\alpha_{(W,\beta;h)}=dt+\beta_t$ on $N_{(W,\beta;h)}$ such that $\beta_0=\beta$, $\beta_t$ are Liouville, and $\beta_t=\beta_0$ on $\bdry W$.  Moreover there exists an extension of the form
$$\alpha_{(W,\beta;h)}=\beta_0|_{\bdry W}+ f(r)d\theta$$ to $\bdry W\times D^2$ such that $f(r)={1\over 2}r^2$ near $r=0$ and $f'(r)>0$ for $r>0$.

\begin{notation}
We write $N_{(W,\beta;h)}$ for the mapping torus and $M_{(W,\beta;h)}$ for the relative mapping torus.
\end{notation}

We now make the following slightly ad hoc definition:

\begin{defn}
A contact manifold $(M,\xi)$ {\em admits or is supported by} an open book decomposition $(W,\beta;h)$ if $(M,\xi)$ is contactomorphic to  $(M_{(W,\beta;h)},\ker\alpha_{(W,\beta;h)})$.
\end{defn}

Then $\bdry W\times \{0\}$ is the binding and $(\{t\}\times W)\cup (\bdry W\times \{r=t/ 2\pi\})$ are the pages of the open book.

\begin{thm}[Giroux~\cite{Gi1}]\label{thm: existence of open book}
Let $M$ be closed oriented $(2n+1)$-dimensional manifold and $\xi$ a cooriented contact structure on $M$. Then $(M,\xi)$ admits an open book decomposition $(W,\beta,\phi;h)$, where $(W,\beta,\phi)$ is a connected $2n$-dimensional Weinstein domain and $h\in \op{Symp}(W,\bdry W,d\beta)$. Moreover, $(W,\beta,\phi;h)$ can be taken to be of {\em Donaldson type}.
\end{thm}

Let $\alpha$ be a contact form on $M$ and $J$ an almost complex structure on $\xi=\ker\alpha$ which is $d\alpha|_\xi$-compatible. Let $g$ be a Riemannian metric such that the Reeb vector field $R_\alpha$ is of unit length and orthogonal to $\xi$ and such that $g|_\xi$ is compatible with $J$ and $d\alpha|_\xi$. Then $(M,\alpha)$ admits an open book decomposition of {\em Donaldson type} if there exist constants $C,\eta>0$ and approximate holomorphic functions
$$s_k: M\to \C, \quad k\geq 1,\quad k\in \Z,$$
such that:
\begin{itemize}
\item for all $p\in M$, $|s_k(p)|\leq C$, $|ds_k-ik s_k\alpha|\leq C\sqrt{k}$, and $|\overline\bdry_\xi s_k|\leq C$;
\item if $|s_k(p)|\leq \eta$, then $|\bdry_\xi s_k|\geq \eta \sqrt k$;
\end{itemize}
for $k$ sufficiently large.\footnote{One problem with this definition is that we are not specifying how large $k$ needs to be.  We want $k\gg 0$ such that Theorem~\ref{thm:stabilization} holds for any Donaldson open books corresponding to $k$ and $k'\geq k$.} Here $\bdry_\xi s_k$ and $\overline\bdry_\xi s_k$ are the $J$-linear and $J$-antilinear components of $ds_k|_\xi$. The binding is $s_k^{-1}(0)$ and the pages are $(\arg\circ s_k)^{-1}(\theta)$.

Let $(W,\beta,\phi)$ be a generic Weinstein manifold, $h\in \op{Symp}(W,\bdry W,d\beta)$, $\bs a=a_1\cup\dots\cup a_\kappa$ be the basis of Lagrangian disks, and $c$ be a properly embedded Lagrangian disk in $W$ whose Legendrian boundary is disjoint from $\bdry {\bs a}$.  We then define the {\em positive stabilization
$$S^+_{c}(W,\beta,\phi;h):=(W',\beta',\phi';h')$$
of $(W,\beta,\phi;h)$ along $c$} as follows: Let $(W',\beta',\phi')$ be the Weinstein domain obtained from $(W,\beta,\phi)$ by attaching a symplectic $n$-handle $W_0$ to $W$ along the Legendrian sphere $\partial c$. Let $\gamma$ be the Lagrangian sphere obtained by gluing $c$ to the Lagrangian core of $W_0$ and let $\tau_{\gamma}$ be the positive symplectic Dehn twist along $\gamma$. We then set
\begin{equation} \label{defn: def of h prime}
h':= (h\cup \op{id}|_{W_0})\circ \tau_\gamma.
\end{equation}
The {\em negative stabilization $S^-_{c}(W,\beta,\phi;h)$} is defined similarly, where $\tau_\gamma$ is replaced by $\tau_\gamma^{-1}$ in the definition of $h'$.

Let $a_0$ be the cocore of $W_0$. Then we write $\bs a'=a_0\cup \bs a$ and $\bs\alpha'=\alpha_0\cup \bs\alpha$, where $\alpha_0$ is obtained by capping off $a_0$.

\begin{thm}[Giroux-Mohsen~\cite{GM}] \label{thm:stabilization} 
Any two Donaldson open book decompositions $(W_i,\beta_i,\phi_i; h_i)$, $i=0,1$, of $(M,\xi)$ can be taken to one another by a sequence of positive stabilizations, conjugations, and Weinstein homotopies.
\end{thm}

\subsubsection{Dimension three}

When $\op{dim} M=3$, the binding $K$ is a link in $M$, the page $W$ is a (symplectic) surface with boundary, and the Liouville form $\beta$ plays no essential role.

In \cite{Gi1}, Giroux showed that Theorem~\ref{thm:stabilization} can be improved to the following:

\begin{thm}
Any two open book decompositions $(W_0; h_0)$ and $(W_1;h_1)$ of $(M,\xi)$ can be taken to one another by a sequence of positive stabilizations, conjugations, and homotopies.
\end{thm}

\begin{rmk}
It is still unknown whether in higher dimensions any two open book decompositions $(W_i,\beta_i,\phi_i; h_i)$, $i=0,1$, of $(M,\xi)$ can be taken to one another by a sequence of positive stabilizations, conjugations, and homotopies.  It is plausible in light of the recent advances in convex hypersurface theory \cite{HH} that this is true if $(W_i,\beta_i)$ are Weinstein.
\end{rmk}

If we extend the equivalence classes of open book decomposition with negative stabilizations, they only remember a part of the
homotopy class of the supported contact structure as plane field: the associated $spin^c$-structure.

\begin{thm}[\cite{GG}]
The equivalence classes of open book decomposition in $M$ modulo positive and negative stabilizations, conjugations, and isotopies
are in one-to-one correspondence with $spin^c$-structures on $M$.
\end{thm}

By the Alexander trick, given an open book decomposition $(K,\theta )$ of $M$, any link $L$ in $M$ can be isotoped to
a link which is disjoint from the binding and transverse to the fibers of $\theta$. We say that such a link is in {\it braid position} with respect to $(K,\theta )$.

\section{Definition of higher-dimensional Heegaard Floer (co-)homology groups} \label{section: defn of homology groups}

Let $(W,\beta,\phi)$ be a generic Weinstein domain, $h\in \op{Symp}(W,\bdry W,d\beta)$, and $(\wt W,\wt \beta,\wt \phi)$ be a capping of $(W,\beta,\phi)$. Also let $\bs a=\sqcup_{i=1}^\kappa a_i$ be the basis of Lagrangian disks for $(W,\beta,\phi)$.

The goal of this section is to define the higher-dimensional Heegaard Floer hat (co-)homology groups $\widehat{HF}(W,\beta,\phi;h; J^\Diamond)$, where $J^\Diamond$ is an auxiliary almost complex structure on $\widehat{X}=\R\times[0,1]\times \wt W^\wedge$ and $\wt W^\wedge$ is the completion of $\wt W$.  This will be done in a manner analogous to the combination of works of Lipshitz~\cite{Li} and Honda-Kazez-Mati\'c~\cite{HKM2} as explained in \cite{CGH1}.

\subsection{Symplectic fibration} \label{subsection: symplectic fibration}

In Sections~\ref{subsection: symplectic fibration}--\ref{subsection: HF for alpha alpha prime} as well as in Section~\ref{section: A infty operations} we work in the following slightly more general setting: Let $(\Sigma,\beta_\Sigma)$ be a Liouville domain and $(\widehat\Sigma,\widehat\beta_\Sigma)$ its completion.
Consider the fibration
$$\pi=\pi_{\R\times[0,1]}: \widehat X= \R\times[0,1]\times \widehat\Sigma\to \R\times[0,1],$$
given by the projection onto the first two factors. 
Also let $\pi_{\widehat\Sigma}$ be the projection to $\widehat\Sigma$ and $\pi_{[0,1]\times \widehat\Sigma}$ be the projection to $[0,1]\times\widehat\Sigma$.

Let $(s,t)$ be coordinates for $\R\times[0,1]$.  Then $\pi$ is a symplectic fibration with respect to the symplectic form $\widehat\Omega=ds\wedge dt+ d\widehat\beta_\Sigma$ on $\widehat X$ and the symplectic form $ds\wedge dt$ on $\R\times[0,1]$.  Let
$$\bs\alpha=\sqcup_{i=1}^\kappa \alpha_i,\quad \bs\alpha'=\sqcup_{i=1}^\kappa \alpha_i'$$
be $\kappa$-component exact Lagrangian submanifolds of $\Sigma$.   Then the submanifolds
$$L_{1i}=\R\times\{1\}\times\alpha_i,\quad L_{0i}=\R\times\{0\}\times \alpha'_i$$
are Lagrangian submanifolds of $(\widehat X,\widehat\Omega)$. We also write
$$L_1= L_{1,\bs\alpha}=\sqcup_i L_{1i},\quad L_0=L_{0,\bs\alpha'}=\sqcup_i L_{0i}.$$

\subsection{Almost complex structures} \label{subsection: almost complex structures}

Consider the symplectization end $\widehat\Sigma-\Sigma$ on which $\widehat\beta_\Sigma= e^\sigma \beta_0$, where $\sigma$ is the $(0,\infty)$-coordinate and $\beta_0:=\beta_\Sigma|_{\bdry\Sigma}$ is a contact form on $\bdry\Sigma$. Let $\xi=\ker \beta_0$ and let $R=R_{\beta_0}$ be the Reeb vector field for $\beta_0$.

Let $J_{\widehat\Sigma}$ be an almost complex structure on $\widehat\Sigma$ which is tamed by $d\widehat\beta_\Sigma$ and is {\em adapted to $\beta_0$} on the symplectization end $\widehat\Sigma-\Sigma$, i.e.,
\begin{itemize}
\item $J_{\widehat\Sigma}(\bdry_\sigma)= R$, $J_{\widehat\Sigma}(R)=-\bdry_\sigma$, and $J_{\widehat\Sigma}(\xi)=\xi$.
\item $d\beta_0 (v,J_{\widehat\Sigma}(v))>0$ for all nonzero $v\in\xi$.
\end{itemize}
The space of such almost complex structures $J_{\widehat\Sigma}$ will be denoted by $\mathcal{J}_{\widehat\Sigma}=\mathcal{J}_{\widehat\Sigma,\widehat\beta_\Sigma}$.

Let $J=J_{\R\times[0,1]}\times J_{\widehat\Sigma}$ be a product almost complex structure on ${\widehat X}$ where $J_{\R\times[0,1]}(\bdry_s)= \bdry_t$ and  $J_{\widehat\Sigma}\in \mathcal{J}_{\widehat\Sigma}$.  Let $J^\Diamond$ be a $C^\infty$-small perturbation of $J$ which satisfies the following:
\begin{itemize}
\item[(J1)] $J^\Diamond$ is $s$-invariant;
\item[(J2)] $J=J^\Diamond$ on a neighborhood of $\R\times\{0,1\}\times{\widehat\Sigma}$ and on $\R\times[0,1]\times (\widehat\Sigma-\Sigma)$; and
\item[(J3)] $\widehat\Omega(v,J^\Diamond v)>0$ for all nonzero tangent vectors $v$, i.e., $\widehat\Omega$ is $J$-positive.
\end{itemize}
We will refer to $J$ as a {\em $(\widehat \Sigma,\widehat\beta_\Sigma)$-compatible almost complex structure on ${\widehat X}$} and to $J^\Diamond$ as a {\em perturbed $(\widehat \Sigma,\widehat\beta_\Sigma)$-compatible almost complex structure on ${\widehat X}$}. The space of $(\widehat \Sigma,\widehat \beta_\Sigma)$-compatible almost complex structures on ${\widehat X}$ of class $C^\infty$ will be denoted by $\mathcal{J}$ and the space of perturbed $(\widehat \Sigma,\widehat\beta_\Sigma)$-compatible almost complex structures on ${\widehat X}$ of class $C^\infty$ by $\mathcal{J}^\Diamond$.

\subsection{Moduli spaces}

\begin{defn}
A {\em $\kappa$-tuple of intersection points of $\bs\alpha$ and $\bs \alpha'$} is a $\kappa$-tuple $\mathbf{y}=\{y_1,\dots,y_\kappa\}$, where $y_i\in \alpha_i\cap \alpha_{\sigma(i)}'$ for some permutation $\sigma$ of $\{1,\dots,\kappa\}$. We denote the set of $\kappa$-tuples of intersection points of $\bs\alpha$ and $\bs \alpha'$ by $\mathcal{S}=\mathcal{S}_{\bs\alpha,\bs\alpha'}$.
\end{defn}

Let $\mathcal{M}_{J^\Diamond}(\mathbf{y},\mathbf{y}')$ be the moduli space of holomorphic maps
$$u: (\dot F,j)\to ({\widehat X},J^\Diamond),$$
where we range over all $(\dot F,j)$ such that  $(F,j)$ is a compact Riemann surface with boundary, $\bf{p}_+$ and $\bf{p}_-$ are disjoint sets of boundary punctures of $F$, and $\dot F=F-\bf{p}_+-\bf{p}_-$, and:
\begin{enumerate}
\item each component of $\bdry F- \bf{p}_+-{\bf p}_-$ maps to some $L_{1i}$ or $L_{0i}$ and each $L_{1i}$ and $L_{0i}$, $i=1,\dots,\kappa$, is used exactly once;
\item $u$ maps the neighborhoods of the punctures of $\bf{p}_+$ (resp. $\bf {p}_-$) asymptotically to strips over the Reeb chords of $\mathbf{y}$ (resp.\ $\mathbf{y}'$) at the positive end (resp.\ negative end).
\end{enumerate}
Here we refer to the $s\to +\infty$ (resp. $s\to -\infty$) end as the {\em positive end} (resp.\ {\em negative end}). An element $u\in \mathcal{M}_{J^\Diamond}(\mathbf{y},\mathbf{y}')$ will be referred to as a {\em curve in ${\widehat X}$} or a {\em multisection of $\pi: \widehat X\to \R\times[0,1]$ from ${\bf y}$ to ${\bf y}'$}.\footnote{Strictly speaking, $u$ is guaranteed to be a multisection of $\pi: \widehat X\to\R\times[0,1]$ only when $J^\Diamond=J$.} 

\begin{rmk} \label{rmk: curve contained in Sigma}
By the definition of $J^\Diamond$, the function $\sigma\circ \pi_{\widehat\Sigma}\circ u$ is subharmonic for all $u\in \mathcal{M}_{J^\Diamond}(\mathbf{y},\mathbf{y}')$, where $\sigma$ is the $(0,\infty)$-coordinate for the end $\widehat\Sigma-\Sigma$.  This means that $\op{Im}(u)\subset X:=\R\times[0,1]\times \Sigma$.
\end{rmk}

\subsection{Fredholm index} \label{subsection: Fredholm index}

We now compute the Fredholm index of $u: \dot F\to {\widehat X}$ in $\mathcal{M}_{J^\Diamond}({\bf y},{\bf y}')$.  What we call the {\em Fredholm index} in this paper takes into consideration the variations of complex structures on the domain $\dot F$. The setup is similar to that of \cite[Section 4.4.2]{CGH1}.

Let $\check X=[-1,1]\times[0,1]\times \Sigma$ be the compactification of $X=\R\times[0,1]\times \Sigma$, obtained by attaching $[0,1]\times \Sigma$ at the positive and negative ends, and let
$$\check L_1=[-1,1]\times\{1\}\times\bs\alpha,\quad \check L_0=[-1,1]\times\{0\}\times \bs \alpha'$$
be the compactifications of $L_1$ and $L_0$.  We then define
$$Z_{\bs\alpha,\bs\alpha'}=\check L_1 \cup\check L_0\cup (\{-1,1\}\times[0,1]\times(\bs\alpha\cap \bs\alpha')).$$
A holomorphic map $u$ from $\bf{y}$ to $\bf{y}'$ can be compactified to a continuous map
$$\check u: (\check F,\bdry \check F)\to (\check X,Z_{\bs\alpha,\bs\alpha'}),$$
where $\check F$ is obtained from $\dot F$ by performing a real blow-up at its boundary punctures.

Given $u$ from $\bf{y}$ to $\bf{y}'$, we define its {\em Maslov index} $\mu(u)$ as follows: We construct a (not necessarily oriented) real rank $n$ subbundle $\mathcal{L}$ of $\check u^* T\Sigma$ on $\bdry \check F$.  The bundle $\mathcal{L}$ is given by $\check u^* T\bs\alpha$ and $\check u^* T\bs\alpha'$ along $\bdry \dot F$ and we extend $\mathcal{L}$ to $\bdry \check F-\bdry \dot F$ by rotating from $T_{y} \bs\alpha'$ to $T_{y}\bs\alpha$ via $e^{J_{\widehat\Sigma} t}$, $t\in[0,{\pi\over 2}]$, where $y\in \bf{y}\cup \bf{y}'$, each component of $\bdry\check F-\bdry \dot F$ is given an oriented parametrization by $[0,\frac{\pi}{2}]$, and we assume without loss of generality that $J_{\widehat\Sigma} ( T_{y}\bs\alpha')= T_{y}\bs\alpha$. Then $\mu(u)$ is the Maslov index of $\mathcal{L}$ with respect to any trivialization of $\check u^* T\Sigma$ on $\check F$.

The following is a generalization of \cite[Equation~(4.4.4)]{CGH1} and its proof will be omitted.  We often write $\chi(u)=\chi(F)$.

\begin{lemma} \label{lemma: Fredholm index}
The Fredholm index of $u$ is
\begin{equation} \label{eqn: Fredholm index formula}
\op{ind}(u)= (n-2) \chi(u) +\mu(u) + (2-n) \kappa,
\end{equation}
where $\dim(W)=2n$ and $\dim (\widehat{X})=2n+2$.
\end{lemma}

\subsection{Splittings of generators and moduli spaces} \label{subsection: splittings}

In this subsection we discuss the splittings of generators and moduli spaces for $\bs\alpha,\bs\alpha'$ in general.

\subsubsection{Generators} \label{subsubsection: generators}

Let
\begin{equation} \label{eqn: h}
h_{\bs\alpha,\bs\alpha'}\subset H_1([0,1]\times\Sigma,(\{1\}\times\bs\alpha)\cup (\{0\}\times \bs\alpha');\Z)
\end{equation}
be the subset of classes $[\bs\delta]$ which admit representatives $\bs\delta$ such that
$$\bdry \bs \delta=\textstyle \sum_{i=1}^\kappa \delta_{i1} -\textstyle\sum_{i=1}^\kappa \delta_{i0},$$
where $\delta_{i1}$ is a point on $\{1\}\times \alpha_i$ and $\delta_{i0}$ is a point on $\{0\}\times \alpha'_i$.  Let
$$\mathcal{H}_{\bs\alpha,\bs\alpha'}:\mathcal{S}_{\bs\alpha,\bs\alpha'}\to h_{\bs\alpha,\bs\alpha'}$$
be the map which sends ${\bf y}=\{y_1,\dots,y_\kappa\}$ to the homology class $[[0,1]\times {\bf y}]$. Then $\mathcal{H}_{\bs\alpha,\bs\alpha'}$ gives a splitting
$$\mathcal{S}_{\bs\alpha,\bs\alpha'}=\textstyle\coprod_{[\bs\delta]\in h_{\bs\alpha,\bs\alpha'}} \mathcal{S}_{\bs\alpha,\bs\alpha',[\bs\delta]}.$$
Note that $\mathcal{M}_{J^\Diamond}({\bf y},{\bf y}')\not=\varnothing$ only if $\mathcal{H}_{\bs\alpha,\bs\alpha'}({\bf y})=\mathcal{H}_{\bs\alpha,\bs\alpha'}({\bf y}')$.

We will suppress ``$\bs\alpha,\bs\alpha'$'' from the notation when it is understood.

\subsubsection{Moduli spaces} \label{subsubsection: moduli spaces}

Suppose ${\bf y},{\bf y}'\in \mathcal{S}_{\bs\alpha,\bs\alpha',[\bs\delta]}$ for some $[\bs\delta]$.  Let
$$\mathcal{A}^{{\bf y},{\bf y}'}_{\bs\alpha,\bs\alpha'}\subset  H_2([0,1]\times \Sigma, (\{1\}\times \bs\alpha) \cup (\{0\}\times \bs\alpha')\cup ([0,1]\times {\bf y})\cup ([0,1]\times {\bf y}');\Z)$$
be the subspace spanned by classes $[T]$ which admit representatives $T$ such that
\begin{equation}\label{eqn: from y to y prime}
\bdry T=\textstyle\sum_{i=1}^\kappa ([0,1]\times y_i)-\textstyle\sum_{i=1}^\kappa([0,1]\times y'_i) -\textstyle\sum_{i=1}^\kappa w_{0i}+\textstyle\sum_{i=1}^\kappa w_{1i},
\end{equation}
where $w_{1i}$ is an arc in $\{1\}\times \alpha_{i}$ from $y_{i}$ to $y'_{i}$ and $w_{0i}$ is an arc in $\{0\}\times \alpha_i'$. A class $[T]$ satisfying Equation~\eqref{eqn: from y to y prime} is said to be {\em from ${\bf y}$ to ${\bf y}'$}.

There is a splitting
$$\mathcal{M}_{J^\Diamond}({\bf y},{\bf y}')=\textstyle\coprod_{B\in \mathcal{A}^{{\bf y},{\bf y}'}_{\bs\alpha,\bs\alpha'}} \mathcal{M}^B_{J^\Diamond}({\bf y},{\bf y}'),$$
where the superscript $B$ in $\mathcal{M}^B_{J^\Diamond}({\bf y},{\bf y}')$ is the modifier ``$u\in \mathcal{M}_{J^\Diamond}({\bf y},{\bf y}')$ is in the class $B\in \mathcal{A}^{{\bf y},{\bf y}'}_{\bs\alpha,\bs\alpha'}$.'' {\em  More generally, if $*$ is a modifier, we write $\mathcal{M}_{J^\Diamond}^*(\mathbf{y},\mathbf{y}')$ to indicate the subset of $\mathcal{M}_{J^\Diamond}(\mathbf{y},\mathbf{y}')$ satisfying property $*$.}

Similarly, if $\chi$ is the modifier ``$\chi(u)=\chi$'', then there is a further decomposition
$$\mathcal{M}_{J^\Diamond}^B({\bf y},{\bf y}')=\textstyle \coprod_{\chi\in \Z,\chi\leq k} \mathcal{M}_{J^\Diamond}^{B,\chi}({\bf y},{\bf y}').$$
Since $\pi\circ u$ is a $\kappa$-fold branched cover of $\R\times[0,1]$ when $J=J^\Diamond$, it follows that $\chi(u)\leq \kappa$ for $J^\Diamond$ arbitrarily close to $J$ by Gromov compactness.

\subsubsection{Complete set of capping surfaces} \label{subsubsection: complete set}

\begin{defn} \label{defn: complete set}
A {\em complete set of capping surfaces} consists of the following data:
\begin{itemize}
\item For each $[\bs\delta]\in h_{\bs\alpha,\bs\alpha'}$, a representative $\bs\delta=\sqcup_{i=1}^\kappa \delta_i$ consisting of $\kappa$ disjoint embedded arcs.
\item For each pair $([\bs\delta],{\bf y})$ consisting of $[\bs\delta]\in h_{\bs\alpha,\bs\alpha'}$ and ${\bf y}\in \mathcal{S}_{\bs\alpha,\bs\alpha',[\bs\delta]}$, a representative $T_{\bs\delta,{\bf y}}$ satisfying Equation~\eqref{eqn: from y to y prime} with $[0,1]\times y_i$ and $[0,1]\times y_i'$ replaced by $\delta_i$ and $[0,1]\times y_i$.
\end{itemize}
We will often denote a complete set of capping surfaces by $\{T_{\bs\delta,{\bf y}}\}$.
\end{defn}

Any two classes in $\mathcal{A}^{{\bf y},{\bf y}'}_{\bs\alpha,\bs\alpha'}$ from ${\bf y}$ to ${\bf y}'$ differ by an element of
\begin{equation} \label{eqn: mathcal A}
\mathcal{A}_{\bs\alpha,\bs\alpha'}:= H_2([0,1]\times \Sigma, (\{1\}\times \bs\alpha) \cup (\{0\}\times \bs\alpha');\Z).
\end{equation}
If we choose a complete set $\{T_{\bs\delta,{\bf y}}\}$ of capping surfaces, then for each pair ${\bf y},{\bf y}'\in  \mathcal{S}_{\bs\alpha,\bs\alpha',[\bs\delta]}$ there is a unique class $A\in  \mathcal{A}^{{\bf y},{\bf y}'}_{\bs\alpha,\bs\alpha'}$ given by $T_{\bs\delta,{\bf y}}+A=T_{\bs\delta,{\bf y}'}$.  We then write $\mathcal{M}^B_{J^\Diamond}({\bf y},{\bf y}')$ instead as $\mathcal{M}^{B-A}_{J^\Diamond}({\bf y},{\bf y}')$ so that $B-A\in \mathcal{A}_{\bs\alpha,\bs\alpha'}$.

\subsection{Regularity}

\begin{lemma} \label{lemma: simply-covered}
Suppose $J\in \mathcal{J}$.
\begin{enumerate}
\item If $u\in \mathcal{M}_{J}^{A,\chi}({\bf y},{\bf y}')$, then no irreducible component of $u$ lies on a fiber of $\pi$.
\item Every irreducible component of $u$ is simply-covered.
\end{enumerate}
\end{lemma}

\begin{proof}
(1) Arguing by contradiction, suppose $\wt u$ is a component of $u$ such that $\pi\circ \wt u$ maps to a point $x\in \R\times[0,1]$.  If $x$ is an interior point, then $\pi_{\widehat\Sigma}\circ \wt u$ is a closed curve, which is not possible since $({\widehat\Sigma},\widehat\beta)$ is an exact symplectic manifold. Similarly, if $x$ is a boundary point, then $\pi_{\widehat\Sigma}\circ \wt u$ is compact with Lagrangian boundary on either $\bs\alpha$ or $\bs\alpha'$, which contradicts the exactness of $\bs\alpha$ and $\bs\alpha'$.

(2) By (1), for every component $\wt u$ of $u$, $\pi\circ \wt u$ is a $\mbox{deg} \geq 1$ branched cover of $\R\times[0,1]$.  The positive ends are asymptotic to a subset of ${\bf y}$ and the negative ends are asymptotic to a subset of ${\bf y}'$.  Since each element of ${\bf y}$ and ${\bf y}'$ is used exactly once, (2) follows.
\end{proof}

\begin{lemma} \label{lemma: transversality}
The moduli space $\mathcal{M}_{J^\Diamond}^{A,\chi}({\bf y},{\bf y}')$ is transversely cut out if $J^\Diamond\in\mathcal{J}^\Diamond$ is generic and sufficiently close to some $J\in \mathcal{J}$.
\end{lemma}

\begin{proof}
By Lemma~\ref{lemma: simply-covered} and Gromov compactness, $u\in \mathcal{M}_{J^\Diamond}^{A,\chi}({\bf y},{\bf y}')$ is simply-covered, provided $J^\Diamond$ is sufficiently close to $J$. This implies that $\pi_{[0,1]\times \widehat\Sigma}\circ u$ is somewhere injective.  We then apply the usual argument; see for example \cite[Proposition 3.8]{Li}.
\end{proof}

Let $(\mathcal{J}^{\Diamond})^{reg}\subset \mathcal{J}^\Diamond$ be the dense subset of regular almost complex structures, i.e., almost complex structures $J^\Diamond$ for which the moduli spaces $\mathcal{M}_{J^\Diamond}^{A,\chi}({\bf y},{\bf y}')$ are transversely cut out for all ${\bf y}$, ${\bf y}'$, $A$, and $\chi$.

\subsection{Orientations} \label{subsection: orientations}

We briefly discuss how to give a coherent (= compatible with gluing) system of orientations for $\mathcal{M}_{J}({\bf y}, {\bf y'})$, following \cite{FOOO}. The key point is to pick (stable) trivializations of $TL_1$ and $TL_0$ and to extend them over the chords $[0,1]\times {\bf y}$ and $[0,1]\times {\bf y'}$.

The orientation data can be decoupled into the base direction $\R\times[0,1]$ and the fiber direction $\widehat\Sigma$.  

We first discuss the fiber direction and orient the curves $\pi_{\widehat\Sigma}\circ u$. The orientation data for $(\widehat\Sigma,\bs\alpha,\bs\alpha')$ consists of:
\be
\item orientations on $\bs\alpha$ and $\bs \alpha'$;
\item a {\em relative spin structure} for the pair $(\bs\alpha,\bs\alpha')$; 
\item and for each intersection point $p\in \bs\alpha\cap \bs\alpha'$ a {\em capping Lagrangian path}, a {\em capping orientation}, and a {\em stable capping trivialization}.
\ee

Since our Lagrangians are (unions of) spheres, a relative spin structure is given by:
\begin{itemize}
\item a trivial $\R$-bundle over $\widehat \Sigma$ such that $(T\bs\alpha \oplus \R)|_{\bs \alpha}$ and $(T\bs\alpha'\oplus \R)|_{\bs \alpha'}$ are trivial; and 
\item trivializations ${\frak t}$ and ${\frak t}'$ for $(T\bs\alpha \oplus \R)|_{\bs \alpha}$ and $(T\bs\alpha'\oplus \R)|_{\bs \alpha'}$.
\end{itemize}
A {\em capping Lagrangian path} is a path $\{\mathcal L_{p,t}\}_{0\leq t \leq 1} $ in the oriented Lagrangian Grassmannian $\op{Lag}(T_p \widehat \Sigma,d\beta_\Sigma(p))$ such that $\mathcal L_{p,0} = T_p \bs\alpha'$ and $\mathcal L_{p,1}= T_p\bs \alpha$ as oriented vector spaces. A {\em stable capping trivialization} is a trivialization $\wt {\frak t}_p$ of $\{\mathcal L_{p,t}\oplus \R\}_{0\leq t\leq 1}$ that agrees with ${\frak t}_p$ and ${\frak t}'_p$ that we have already chosen.

For each $p\in \bs\alpha\cap \bs\alpha'$, let $\pi_p: \H \to \widehat \Sigma$ be the constant map to $p$, where $\H =\{z ~|~ \op{Im} z \geq 0\}$ is the upper half plane, and let $\xi =\pi_p^*(T_pM)$ be a trivial(ized) vector bundle over $\H$. We define a Cauchy-Riemann tuple $(\xi, \eta^{p+}, D^{p+})$ as follows: 
\begin{itemize}
	\item the real subbundle $\eta^{p+}\subset \xi$ is given by $\eta^{p+}_z= {\mathcal L}_{p,0}$ for $z\in (-\infty, 0)$, $\eta^{p+}_z={\mathcal L}_{p,z}$ for $z \in [0,1]$, and $\eta^{p+}_z = {\mathcal L}_{p,1}$ for $z \in (1,+\infty)$; and
	\item $D^{p+}$ is a fixed real linear Cauchy-Riemann operator $W^{k+1,p}(\H,\xi) \to W^{k,p}(\H, \wedge^{0,1}\H \otimes_{\mathbb C}\xi)$ with boundary condition on $\eta^{p+}$; see \cite[Section 3.2]{BH} for a more complete discussion.
\end{itemize}
Then a {\em capping orientation} is a choice of orientation $\mathfrak{o}(D^{p+})$.  We can similarly choose the Cauchy-Riemann tuple $(\xi, \eta^{p-}, D^{p-})$ by swapping the roles of $\bs \alpha$ and $\bs \alpha'$.

The $\R\times[0,1]$-direction is easy: The Lagrangians are $\ell_i=\R\times \{i\}$, $i=0,1$, the ``capping Lagrangian paths'' at the end $s=\pm \infty$ we take to be limits of $\R\langle \bdry_s\rangle$, and the trivializations are given by $\bdry_s$. 

\subsection{Definition of $\widehat{CF}(\Sigma,\bs\alpha',\bs\alpha;J^\Diamond)$}  \label{subsection: HF for alpha alpha prime}

In this subsection we define the cochain groups $\widehat{CF}(\Sigma,\bs\alpha',\bs\alpha;J^\Diamond)$, where $J^\Diamond\in (\mathcal{J}^\Diamond)^{reg}$. We assume that $n>2$.

\begin{rmk}
The astute reader might have noticed that we have switched the orders of $\bs\alpha$ and $\bs\alpha'$ (i.e., we switched to cohomology).  The order was switched for consistency with the $\ai$-conventions from \cite{Se3}.
\end{rmk}

\n
{\em Coefficient ring.} We first describe the coefficient ring $\Lambda \llbracket\hslash\rrbracket$. Since we have already chosen auxiliary orientation data in Section~\ref{subsection: orientations}, the ground ring is $\Z$.   The ``Planck constant'' $\hslash$ is a variable, $\Lambda\llbracket\hslash\rrbracket$ is the formal power series ring in $\hslash$ over $\Lambda$, and $\Lambda:=\Lambda_{\bs\alpha,\bs\alpha'}$ is the Novikov ring with $\Z$-coefficients over
$$\mathcal{A}'_{\bs\alpha,\bs\alpha'}:=\mathcal{A}_{\bs\alpha,\bs\alpha'}/\rm{torsion}$$
with respect to the symplectic form $d\beta$. More explicitly, elements $\lambda\in\Lambda$ are formal sums
$$\textstyle\sum_{A\in \mathcal{A}'_{\bs\alpha,\bs\alpha'}} \lambda_A e^A,$$
where $\lambda_A\in \Z$ and for each $C>0$ there are only finitely many $A\in \mathcal{A}'_{\bs\alpha,\bs\alpha'}$ with $d\beta((\pi_{{\widehat\Sigma}})_*(A))\leq C$.

\s\n
{\em Cochain group.} The cochain group $\widehat{CF}(\Sigma,\bs\alpha',\bs\alpha;J^\Diamond)$ is the free 
$\Lambda\llbracket\hslash\rrbracket$-module generated by $\mathcal{S}=\mathcal{S}_{\bs\alpha,\bs\alpha'}$.  The differential is given by:
\begin{equation} \label{eqn: differential}
d {\bf y}= \textstyle\sum_{{\bf y}'\in\mathcal{S},\chi\leq \kappa,A\in\mathcal{A}'_{\bs\alpha,\bs\alpha'}}\langle d {\bf y}, \hslash^{\kappa-\chi} e^A{\bf y}' \rangle \cdot \hslash^{\kappa-\chi} e^A{\bf y}',
\end{equation}
where $\langle d {\bf y}, \hslash^{\kappa-\chi} e^A{\bf y}'\rangle$ is the count of $\mathcal{M}_{J^\Diamond}^{\op{ind}=1,A,\chi}({\bf y},{\bf y}')/\R$.

\begin{rmk}
When $n=2$, we additionally impose an embeddedness condition on the holomorphic curves (or, alternatively, count ECH index $I=1$ curves).  This, together with the adjunction formula, gives restrictions on $\chi$.  On the other hand, when $n\geq 3$, the embeddedness condition is generically satisfied for $\op{ind}=1$ curves.  Also when $n=3$, the Fredholm index formula (Equation~\eqref{eqn: Fredholm index formula}) gives no restrictions on $\chi$.  These differences account for the slightly different form of the differential for $n>2$.
\end{rmk}

\n
{\em Grading.} In view of Equation~\eqref{eqn: Fredholm index formula} and the fact that the differential is degree-increasing, we set $|\hbar|=n-2$. Also $|e^A|=-2c_1(A)$.  The elements in $\mathcal{S}$ are {\em relatively graded} such that if $u\in \mathcal{M}_{J^\Diamond}^{\op{ind}=\ell,A,\chi}({\bf y},{\bf y}')$, then
$$(|{\bf y'}| + (n-2)(\kappa-\chi) -2c_1(A))  - |{\bf y}| =\ell.$$

\s\n
{\em Splitting.} The splitting $\mathcal{S}=\coprod_{[\bs\delta]\in h_{\bs\alpha,\bs\alpha'}} \mathcal{S}_{[\bs\delta]}$ gives rise to the splitting
$$\widehat{CF}(\Sigma,\bs\alpha',\bs\alpha;J^\Diamond)=\textstyle\oplus_{[\bs\delta]\in h_{\bs\alpha,\bs\alpha'}} \widehat{CF}(\Sigma,\bs\alpha',\bs\alpha;J^\Diamond;[\bs\delta]).$$

It remains to show that $\mathcal{M}_{J^\Diamond}^{\op{ind}=1,A,\chi}({\bf y},{\bf y}')/\R$ is compact and that $d^2=0$.

\begin{lemma} \label{lemma: bdry squared equals zero}
$d^2=0$.
\end{lemma}

\begin{proof}
Let $\mathcal{M}':=\mathcal{M}^{\op{ind}=2,A,\chi}_{J^\Diamond}({\bf y},{\bf y}')/\R$. We claim that
$$\bdry \mathcal{M}'=\textstyle\coprod (\mathcal{M}^{\op{ind}=1,A_1,\chi_1}_{J^\Diamond}({\bf y},{\bf y}'')/\R) \times (\mathcal{M}^{\op{ind}=1,A_2,\chi_2}_{J^\Diamond}({\bf y}'',{\bf y}')/\R),$$
where the union is over all ${\bf y}''\in\mathcal{S}$, $\chi_1+\chi_2-\kappa=\chi$, and $A_1,A_2\in \mathcal{A}'_{\bs\alpha,\bs\alpha'}$ such that $A_1+A_2=A$. Here $\bdry \mathcal{M}'$ is constructed using the usual SFT compactness theorem \cite{BEHWZ}.

First suppose that $J=J^\Diamond$ and $J\in (\mathcal{J}^\Diamond)^{reg}$. By Lemma~\ref{lemma: simply-covered}(1), $\mathcal{M}'$ and $\bdry \mathcal{M}'$ have no components that lie in fibers $\pi^{-1}(pt)$. Let $u_\infty\in\bdry \mathcal{M}'$ be the SFT limit of a sequence $u_i: \dot F_i\to {\widehat X}$, $i=1,2,\dots$, in $\mathcal{M}'$; without loss of generality we may assume that the topological types of all the $\dot F_i$ are the same. Observe that $u_i$ has image in $X=\R\times[0,1]\times\Sigma$, since the maximal principle holds for $\pi_{\widehat\Sigma}\circ u_i$.

We claim that $u_\infty$ cannot have any interior nodes that are obtained in the limit by pinching a closed curve in $\dot F_i$. Suppose for simplicity that the domain of $u_\infty$ is obtained by pinching a single closed curve in $\dot F_i$. Then
$$\op{ind}(u_\infty)=\op{ind}(u_i)+2(n-2)$$
by Equation~\eqref{eqn: Fredholm index formula}.
On the other hand, the matching/incidence condition at the node forces $\op{ind}(u_\infty)> 2n-2$:
Assuming that two components of $u_\infty$ are matched, in order for $\pi_{[0,1]\times\Sigma}\circ u_\infty$ to have an intersection point in $[0,1]\times \Sigma$, we require
\begin{equation}\label{eqn: ineq for ind u infty}
\op{ind}(u_\infty)-2\geq \op{dim}([0,1]\times \Sigma)-4= 2n-3.
\end{equation}
Here the $-2$ on the left-hand side comes from quotienting out the $\R$-translations.
If there is only one component of $u_\infty$, then a slightly different argument also yields Inequality~\eqref{eqn: ineq for ind u infty}. Hence $\op{ind}(u_\infty)> 2n-2$ and $\op{ind}(u_i)>2$, a contradiction.

It remains to consider the limit of pinching an arc $c_i$ in $\dot F_i$. It is not possible for $c_i$ to connect $\alpha_q$ to $\alpha_r$, where $q\not=r$, since $\alpha_q$ and $\alpha_r$ are disjoint (and similarly $\alpha'_q$ to $\alpha'_r$, where $q\not=r$).  On the other hand, if $c_i$ connects $\alpha_q$ to itself, then $\pi\circ u_\infty$ maps a boundary component of the domain $\dot F_\infty$ of $u_\infty$ identically to a point $x$ on $\bdry (\R\times[0,1])$. This implies that $\pi\circ u_\infty$ maps a component of $\dot F_\infty$ to $x$, contradicting Lemma~\ref{lemma: simply-covered}(1). Finally, if an arc $c_i$ connects $\alpha_q$ to $\alpha'_r$, then this pinching corresponds to stretching in the $s$-direction.

It follows that $u_\infty$ is an $l$-level building $v_1\cup\dots \cup v_l$, where each $v_j$ is a degree $\kappa$ multisection of $\pi$.  By Lemma~\ref{lemma: transversality} and the assumption of the regularity of $J$, $l=2$ and $\op{ind}(v_1)=\op{ind}(v_2)=1$.

Next suppose that $J^\Diamond\in (\mathcal{J}^\Diamond)^{reg}$ is sufficiently close to $J\in \mathcal{J}$. Then $\mathcal{M}'$ and $\bdry \mathcal{M}'$ have no components that are close to lying on a fiber $\pi^{-1}(pt)$. The rest of the argument that shows that $u_\infty$ is an $l$-level building $v_1\cup\dots \cup v_l$, where each $v_j$ is close to being a degree $\kappa$ multisection of $\pi$, is identical.
\end{proof}

\begin{lemma} \label{lemma: compactness}
For all $A$, $\chi$, ${\bf y}$, and ${\bf y}'$,  $\mathcal{M}_{J^\Diamond}^{\op{ind}=1,A,\chi}({\bf y},{\bf y}')/\R$ is compact.
\end{lemma}

\begin{proof}
Similar to that of Lemma~\ref{lemma: bdry squared equals zero} and is left to the reader.
\end{proof}

We write $\widehat{HF}(\Sigma,\bs\alpha',\bs\alpha;J^\Diamond)$ for the cohomology of $\widehat{CF}(\Sigma,\bs\alpha',\bs\alpha; J^\Diamond)$.

\begin{rmk}
It is not hard to see that two cochain groups $\widehat{CF}(\Sigma,\bs\alpha',\bs\alpha; J^\Diamond)$ corresponding to distinct complete sets of capping surfaces are isomorphic cochain complexes.
\end{rmk}

\subsection{Definition of $\widehat{CF}(W,\beta,\phi; h; J^\Diamond)$} \label{subsection: specialization}

The {\em higher-dimensional Heegaard Floer cochain group} $\widehat{CF}(W,\beta,\phi;h; J^\Diamond)$ is a special case of $\widehat{CF}(\Sigma,\bs\alpha',\bs\alpha;J^\Diamond)$, where:
\begin{itemize}
\item $(\Sigma,\beta_\Sigma,\phi_\Sigma)$ is the capping $(\wt{W},\wt\beta,\wt\phi)$ of $(W,\beta,\phi)$,
\item $\bs\alpha$ is the capping $\wt {\bs a}$ of the basis $\bs a$ of Lagrangian disks for $(W,\beta,\phi)$,
\item $\bs\alpha'=\wt{h}(\bs\alpha)$, where $\wt h$ is obtained from $h\cup \op{id}_H$ and $H$ is the cap for $W$, by composing with a small Hamiltonian diffeomorphism as in Section~\ref{subsection: extension}, and
\item the coefficient ring has been further specialized.
\end{itemize}
We write $\widehat{HF}(W,\beta,\phi;h; J^\Diamond)$ for the cohomology of $\widehat{CF}(W,\beta,\phi;h; J^\Diamond)$.

We often write $\widehat{CF}(h(\bs a),\bs a)$ for $\widehat{CF}(W,\beta,\phi;h; J^\Diamond)$.

We first discuss the splittings of generators and moduli spaces in the special case when $\bs\alpha=\wt{\bs a}$ and $\bs\alpha'=\wt{h}(\wt{\bs a})$.

\subsubsection{Generators}\label{subsubsection: generators II}

Let $\mathcal{S}=\mathcal{S}_{\wt{\bs a},\wt{h}(\wt{\bs a})}$.
Recall there is a unique point $x_i\in \wt a_i\cap \wt{h}(\wt a_i)$ which is contained in the handle $H_i$.  We define two maps $\mathcal{H}'$ and $\iota$.  The map
$$\mathcal{H}':\mathcal{S}\to H_1([0,1]\times W, (\{1\}\times {\bs a}) \cup (\{0\}\times h({\bs a}));\Z),$$
sends ${\bf y}=\{y_1,\dots,y_\kappa\}$ to the homology class $[[0,1]\times \eta({\bf y})]$, where $\eta$ is defined as follows: $\eta$ maps $y_i\in \wt a_i\cap \wt{h}(\wt a_{\sigma(i)})$ to the nearby point $\eta(y_i)\in a_i\cap h(a_{\sigma(i)})$ if $y_i\not=x_i$ (recall that $\wt h$ is close to but not exactly the same as $h\cup\op{id}|_H$) and to some point on $\bdry a_i$ if $y_i=x_i$. Let $M=M_{(W,\beta;h)}$ be the relative mapping torus of $(W,\beta;h)$.  The map
$$\iota: \op{Im}(\mathcal{H}')\to H_1(M;\Z),$$
is defined as follows: Let $\bs \delta$ be a representative of $[\bs\delta]\in\op{Im}(\mathcal{H}')$ with ``initial points'' $\delta_{i0}$ on $\{0\}\times h(a_i)$ for $i=1,\dots,\kappa$ and ``terminal points'' $\delta_{i1}$ on $\{1\}\times a_i$ for $i=1,\dots,\kappa$. Viewing $\bs\delta$ as a chain in $M$, we define $\iota([\bs\delta])$ to be the homology class of the ``closure'', obtained from $\bs\delta$ by adjoining arcs from $\delta_{i0}$ to $\delta_{i1}$ (which we call {\em augmenting arcs}) along $\{1\}\times a_i$ for each $i$.  Note that the resulting homology class is independent of the choice of augmenting arcs, since the $a_i$ are simply-connected.  We then set $\mathcal{H}''=\iota\circ\mathcal{H}'$.

The map $\mathcal{H}''$ gives a splitting
$$\mathcal{S}=\textstyle\coprod_{\Gamma\in H_1(M;\Z)} \mathcal{S}_\Gamma,$$
which is analogous to the splitting into Spin$^c$-structures in dimension $3$.\footnote{In general, Spin$^c$-structures on $M$ are classified by $2H^2(M;\Z)\oplus H^1(M;\Z/2\Z)$. When $\dim M=2n+1=3$, this can be identified with $H^2(M;\Z)$ and hence Spin$^c$-structures are in one-to-one correspondence with $H_1(M;\Z)$.  For $n>1$, there is no such identification.}

\subsubsection{Moduli spaces} \label{subsubsection: moduli spaces II}

Let
\begin{gather*}
\mathcal{A}^0_{{\bf y},{\bf y}'}\subset  H_2([0,1]\times W, (\{1\}\times {\bs a}) \cup (\{0\}\times h({\bs a})))\cup ([0,1]\times\eta({\bf y}))\cup ([0,1]\times \eta({\bf y}'));\Z)
\end{gather*}
be the subspace spanned by classes $[T]$ that admit representatives $T$ such that
$$\bdry T=\textstyle\sum_{i_0=1}^\kappa ([0,1]\times \eta(y_{i_0}))-\textstyle\sum_{i_1=1}^\kappa([0,1]\times\eta(y'_{i_1})) -\textstyle\sum_{i_0=1}^\kappa w_{0i_0}+\textstyle\sum_{i_1=1}^\kappa w_{1i_1},$$
where $w_{1i_1}$ is an arc in $\{1\}\times a_{i_1}$ from $y_{i_1}$ to $y'_{i_1}$ and $w_{0i_0}$ is similarly defined.

Next let
\begin{align*}
\mathcal{A}_{{\bf y},{\bf y}'}:=& H_2(M,\wt{\mathcal{H}}({\bf y})\cup\wt{\mathcal{H}}({\bf y}');\Z),
\end{align*}
where $\wt{\mathcal{H}}({\bf y})$ refers to the ``closure'' of $[0,1]\times\eta({\bf y})$, as described in Section~\ref{subsubsection: generators II}, representing the homology class $\mathcal{H}''({\bf y})$.

There is a map
$$\kappa=\kappa_{\wt{\mathcal{H}}({\bf y})}: \mathcal{A}^0_{{\bf y},{\bf y}'}\to \mathcal{A}_{{\bf y},{\bf y}'},$$
defined as follows: View the representative $T$ as a surface in $M$.  Then the arcs $w_{0i_0}$ and $w_{1i_1}$, together with the augmenting arcs for $\eta({\bf y})$ and $\eta({\bf y}')$, form closed curves in $\cup_i(\{1\}\times a_i)$ which bound disks $D_i\subset \{1\}\times a_i$ since the $a_i$ are contractible. Then $\kappa([T])$ is represented by $T\cup (\cup_i D_i)$.  By the contractibility of $a_i$, $\kappa([T])$ does not depend on the choice of $D_i$ and $\kappa_{\wt{\mathcal{H}}_1({\bf y})}$ and $\kappa_{\wt{\mathcal{H}}_2({\bf y})}$ coming from different choices of augmenting arcs can be naturally identified.

The map $\kappa$ gives rise to the splitting
$$\mathcal{M}_{J^\Diamond}({\bf y},{\bf y}')=\textstyle\coprod_{B\in \mathcal{A}_{{\bf y},{\bf y}'}} \mathcal{M}^B_{J^\Diamond}({\bf y},{\bf y}'),$$
where the superscript $B$ is the modifier ``$u\in \mathcal{M}_{J^\Diamond}({\bf y},{\bf y}')$ is in the class $B\in \mathcal{A}_{{\bf y},{\bf y}'}$.''

\subsubsection{Definition of $\widehat{CF}(W,\beta,\phi; h; J^\Diamond)$} \label{subsubsection: definition of specialization}

We now choose a complete set of capping surfaces $\{T_{\delta,{\bf y}}\}$ in $[0,1]\times W$, defined in a manner analogous to that of Section~\ref{subsubsection: complete set}, and take their ``closures'' $\{\wt T_{\delta,{\bf y}}\}$ in $M$.

Using this, we can define $\widehat{CF}(W,\beta,\phi; h; J^\Diamond)$ as in Section~\ref{subsection: HF for alpha alpha prime} with $\Lambda\llbracket\hslash\rrbracket$-coefficients, where $\Lambda$ is the Novikov ring with $\Z$-coefficients over $H_2(M;\Z)$.

\section{$\ai$-operations} \label{section: A infty operations}

In this section we explain how to modify the discussion from Section~\ref{section: defn of homology groups} when the base of the symplectic fibration is a disk with $m+1$ boundary punctures instead of $\R\times[0,1]$ and obtain the $\ai$-operations.

Let $D=\{|z|\leq 1\}\subset \C$ and let $D_m=D-\{p_0,\dots,p_m\}$, where $p_0,\dots,p_m\in \bdry D$, arranged in counterclockwise order around $\bdry D$.  Let $\bdry_i D_m$, $i=0,\dots,m$, be the component of $\bdry D_m$ which is the counterclockwise open arc from $p_i$ to $p_{i+1}$, where the subscripts are taken modulo $m+1$. Also let $e_i\subset D_m$ be a ``neighborhood of $p_i$'' which we view as a strip-like end $[0,\infty)\times [0,1]$ with coordinates $(s_i,t_i)$.

\begin{figure}[ht]
	\begin{overpic}[scale=1]{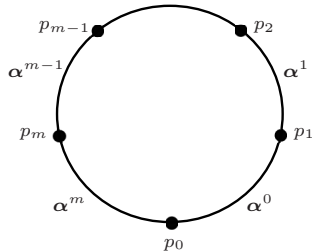}
		\put(47,-7){\tiny $p_0$}
		\put(102,40){\tiny $p_1$}
		\put(85,85){\tiny $p_2$}
		\put(-6,85){\tiny $p_{m-1}$}
		\put(-14,40){\tiny $p_m$}
		\put(82,8){\tiny $\bs\alpha^0$}
		\put(97.5,65){\tiny $\bs \alpha^1$}
		\put(-19,65){\tiny $\bs\alpha^{m-1}$}
		\put(0,8){\tiny $\bs\alpha^m$}
	\end{overpic}
	\caption{The base $D_m$ of the symplectic fibration $\widehat X_m= D_m\times \widehat \Sigma$.  The labels $\bs\alpha^i$ indicate the boundary condition $L_i= \bdry_i D_m\times \bs \alpha^i$.}
	\label{fig: A-infty}
\end{figure}

Consider the symplectic fibration
$$\pi_m: (\widehat{X}_m=D_m\times \widehat\Sigma,\widehat\Omega_m = \omega + d\widehat\beta_\Sigma)\to (D_m,\omega),$$
where $\omega$ is an area form on $D$ which restricts to $ds_i\wedge dt_i$ on each strip-like end $e_i$. For $i=0,\dots, m$, let
$$\mathcal{E}_i= \pi_m^{-1}(e_i)= [0,\infty)\times[0,1]\times \widehat\Sigma$$
be the ``ends in the horizontal direction'' of $\widehat{X}_m$. Also let $\bs\alpha^i=\sqcup_{j=1}^\kappa\alpha^i_j$ be a compact exact $\kappa$-component Lagrangian submanifold of $(\widehat\Sigma,\widehat\beta)$ such that $\bs\alpha^i\pitchfork \bs\alpha^{i+1}$ for all $i=0,\dots,m$. We then set
$$L_i= \bdry_i D_m\times \bs\alpha^i,\quad L_{ij}=\bdry_i D_m\times \alpha^i_j.$$

Next let $j_m$ be the standard complex structure on $D_m$; we may assume that $j_m(\bdry_{s_i})=\bdry_{t_i}$ on each $e_i$.  Also let $J_m=j_m\times J_{\widehat\Sigma}$ be a product almost complex structure on $\widehat{X}_m$, and $J^\Diamond_m$ be a $C^\infty$-small perturbation of $J_m$ such that (J1)--(J3) from Section~\ref{subsection: almost complex structures} hold on each end $\mathcal{E}_i$, with $\R$ replaced by $[0,\infty)$ and $(s,t)$ replaced by $(s_i,t_i)$.

For each $i=0,\dots,m$, let $\mathbf{y}^i=\{y^i_1,\dots,y^i_\kappa \}$ be a $\kappa$-tuple of intersection points $y^i_j\in \alpha^{i-1}_j\cap \alpha^{i}_{\sigma_i(j)}$, $j=1,\dots,\kappa$, for some permutation $\sigma_i$.

Let $\mathcal{M}_{J^\Diamond_m}(\mathbf{y}^0,\dots,\mathbf{y}^m)$ be the moduli space of holomorphic maps
$$u: (\dot F,j)\to (\widehat{X}_m, J^\Diamond_m),$$
where we range over all $(\dot F,j)$ such that $(F,j)$ is a compact Riemann surface with boundary, ${\bf p}_0,\dots,{\bf p}_m$ are disjoint sets of boundary punctures of $F$, and $\dot F=F-\cup_i {\bf p}_i$, and:
\begin{enumerate}
\item each component of $\bdry F- \cup_i {\bf p}_i$ maps to some $L_{ij}$ and each $L_{ij}$ is used exactly once;
\item $u$ maps the neighborhoods of the punctures of ${\bf p}_i$ asymptotically to strips over the Reeb chords of  $\mathbf{y}^i$ on $\mathcal{E}_i$.
\end{enumerate}
Note that $\dot F$ has $(m+1)\kappa$ punctures.

Let $\check X_m$ be the compactification of $D_m\times \Sigma$, obtained by attaching $[0,1]\times \Sigma$ to each end $[0,1]\times[0,\infty)\times \Sigma$, let $\check L_i$ be the compactification of $L_i$, and let $\check u:\check F\to \check X_m$ be the compactification of $u\in \mathcal{M}_{J^\Diamond_m}(\mathbf{y}^0,\dots,\mathbf{y}^m)$, defined as in Section~\ref{subsection: Fredholm index}. The Maslov index $\mu(u)$ is defined as in Section~\ref{subsection: Fredholm index} via the bundle $\mathcal{L}$: for each component of $\bdry \dot F$, $\mathcal{L}$ is given by some $\check u^* T\alpha^i_j$ and, on the components of $\bdry\check F-\bdry\dot F$,  $\mathcal{L}$ is given by rotating from $T_{y^i_j} \alpha^{i-1}_{j}$ to $T_{y^i_j}\alpha^i_{\sigma_i(j)}$ via $e^{J_{\widehat\Sigma}t}$, $t\in[0,\tfrac{\pi}2]$, where $i=1,\dots,m$ and $j=1,\dots,\kappa$, and by rotating from $T_{y^m_j}\alpha^m_j$ to $T_{y^m_j}\alpha^0_{\sigma_m(j)}$ via $e^{-J_{\widehat\Sigma}t}$, $t\in[0,\tfrac{\pi}2]$, where $j=1,\dots,\kappa$. Without loss of generality we are assuming that $J_{\widehat\Sigma}(T_{y^i_j} \alpha^{i-1}_j)=T_{y^i_j}\alpha^i_{\sigma_i(j)}$ for all $i=0,\dots,m$, $j=1,\dots,\kappa$.

The following lemma generalizes Lemma~\ref{lemma: Fredholm index} and its proof is left to the reader.

\begin{lemma} \label{lemma: Fredholm index 2}
The Fredholm index of $u\in \mathcal{M}_{J^\Diamond_m}(\mathbf{y}^0,\dots,\mathbf{y}^m)$ is
\begin{equation}
\op{ind}(u)= (n-2) \chi(u) +\mu(u) + 2\kappa-m \kappa n,
\end{equation}
where $\dim(W)=2n$ and $\dim (\widehat{X})=2n+2$.
\end{lemma}




Choose complete sets of capping surfaces
$$\{T^{\bs\alpha^0,\bs\alpha^1}_{\bs\delta^0,{\bf y}^0}\},\dots, \{T^{\bs\alpha^{m-1},\bs\alpha^m}_{\bs\delta^{m-1},{\bf y}^{m-1}}\}, \{T^{\bs\alpha^0,\bs\alpha^m}_{\bs\delta^m,{\bf y}^m}\}$$
corresponding to the pairs
$$(\bs\alpha^0,\bs\alpha^1),\dots,(\bs\alpha^{m-1},\bs\alpha^m),(\bs\alpha^0,\bs\alpha^m).$$
For each $(m+1)$-tuple $\bs\delta^0,\dots,\bs\delta^m$, we choose
$$[T^{\bs\alpha^0,\dots,\bs\alpha^m}_{\bs\delta^0,\dots,\bs\delta^m}]\in H_2(\check X_m, (\sqcup_{i=0}^m \check L_i) \cup (\sqcup_{i=0}^m \bs\delta^i)),$$
such that $\bdry T^{\bs\alpha^0,\dots,\bs\alpha^m}_{\bs\delta^0,\dots,\bs\delta^m}$ is equal to  $\bs\delta^0+\dots+\bs\delta^{m-1}-\bs\delta^m$ plus some arcs on $\sqcup_{i=0}^m \check L_i$, if such a homology class exists.  Such a collection $\{T^{\bs\alpha^1,\dots,\bs\alpha^m}_{\bs\delta^1,\dots,\bs\delta^m}\}$ is a {\em complete set of capping surfaces with base $D_m$}.

The coefficient ring for $\widehat{CF}(\Sigma,\bs\alpha^i,\bs\alpha^{i+1})$ is $\Lambda_{\bs\alpha^i,\bs\alpha^{i+1}}\llbracket \hbar \rrbracket$, where $\Lambda_{\bs\alpha^i,\bs\alpha^{i+1}}$ is the Novikov ring with $\Z$-coefficients over $$\mathcal{A}'_{\bs\alpha^i,\bs\alpha^{i+1}}:=\mathcal{A}_{\bs\alpha^i,\bs\alpha^{i+1}}/\rm{torsion}$$ 
with respect to the symplectic form $d\beta$ and the coefficient ring for $\widehat{CF}(\Sigma,\bs\alpha^0,\bs\alpha^m)$ is $\Lambda_{\bs\alpha^0,\bs\alpha^m}$. Let $\Lambda'$ be the Novikov ring with $\Z$-coefficients over 
$$\mathcal{A}':=H_2(\check X_m, \sqcup_{i=0}^m \check L_i)/\rm{torsion},$$  
which can be viewed as a module over $\Lambda_{\bs\alpha^i,\bs\alpha^{i+1}}$ and over $\Lambda_{\bs\alpha^0,\bs\alpha^m}$.  By tensoring with $\Lambda'\llbracket \hbar \rrbracket$, we obtain $\widehat{CF}(\Sigma,\bs\alpha^i,\bs\alpha^{i+1};\Lambda'\llbracket \hbar \rrbracket)$ and $\widehat{CF}(\Sigma,\bs\alpha^0,\bs\alpha^m;\Lambda'\llbracket \hbar \rrbracket)$.

Using the complete sets of capping surfaces, we define the map
$$\Psi_m: \widehat{CF}(\Sigma,\bs\alpha^{m-1},\bs\alpha^m;\Lambda'\llbracket \hbar \rrbracket)\otimes\dots \otimes \widehat{CF}(\Sigma,\bs\alpha^0,\bs\alpha^1;\Lambda'\llbracket \hbar \rrbracket)\to \widehat{CF}(\Sigma,\bs\alpha^0,\bs\alpha^m;\Lambda'\llbracket \hbar \rrbracket),$$
where the coefficient of $e^{\kappa-\chi}e^A\mathbf{y}^m$ in $\Psi_{m}(\mathbf{y}^0\otimes\dots\otimes \mathbf{y}^{m-1})$, with $\chi\leq \kappa$ and $A\in \mathcal{A}'$,  is the count of $\mathcal{M}_{J^\Diamond_m}^{\op{ind}=0,A,\chi}(\mathbf{y}^0,\dots,\mathbf{y}^m)$.

\begin{lemma}
$\Psi_{m}$ is a cochain map.
\end{lemma}

\begin{proof}
The proof is similar to that of Lemma~\ref{lemma: bdry squared equals zero}.
\end{proof}

We can also make a similar construction when the points $p_0,\dots,p_m\in \bdry D$ are allowed to move, i.e., when the conformal structure on $D_m$ is parametrized by the interior of the Stasheff associahedron $\mathcal{A}_m$. Summing over all rigid maps $u: (\dot F,j)\to (\widehat X_m,J^\Diamond_m)$ when viewed in the set of maps in a varying target parametrized by $\op{int}\mathcal{A}_m$ gives the $\ai$-map
\begin{equation} \label{eqn: A infty}
\mu_m:\widehat{CF}(\Sigma,\bs\alpha^{m-1},\bs\alpha^m;\Lambda'\llbracket \hbar \rrbracket)\otimes\dots \otimes \widehat{CF}(\Sigma,\bs\alpha^0,\bs\alpha^1;\Lambda'\llbracket \hbar \rrbracket)\to \widehat{CF}(\Sigma,\bs\alpha^0,\bs\alpha^m;\Lambda'\llbracket \hbar \rrbracket).
\end{equation}

\begin{prop}
The operations $\mu_m$, $m=1,2,\dots$, satisfy the $\ai$-relations.
\end{prop}

\begin{proof}
The proof is also similar to that of Lemma~\ref{lemma: bdry squared equals zero}.
\end{proof}

Let $\mathcal{R}_\kappa(\Sigma)$ be the $A_\infty$-category whose objects are compact $\kappa$-component Lagrangian submanifolds $\bs\alpha=\sqcup_j \alpha_j$ of $\Sigma$, whose hom sets are
$$\op{Hom}(\bs\alpha,\bs\alpha'):=\widehat{CF}(\Sigma,\bs\alpha,\bs\alpha';\Lambda'),$$ 
and whose $A_\infty$ maps $\mu_m$ are given by Equation~\eqref{eqn: A infty}.

\section{Reformulation of $\widehat{HF}(W,\beta,\phi;h;J^\Diamond)$} \label{section: reformulation of HF}

In this section we give a reformulation of $\widehat{HF}(W,\beta,\phi;h;J^\Diamond)$ where we count holomorphic curves in $\widehat X_W: =\R\times[0,1]\times \widehat{W}$ instead of $\widehat X= \R\times[0,1]\times \widehat\Sigma$.  Here $\widehat{W}$ is the completion of $W$. We can think of $\widehat \Sigma$ as the higher-dimensional analog of a Heegaard surface, whereas $\widehat{W}$ is the page of an open book decomposition.

\subsection{The completion $\widehat{W}$}

Recall that $(W,\beta,\phi)$ has a collar neighborhood $N'(\bdry W)=[-\varepsilon,0]\times \bdry W$ with coordinates $(\sigma,x)$ on which $\beta=e^\sigma\beta|_{\bdry W}$ and $\phi(\sigma,x)=\sigma$.
Let $(\widehat W,\widehat{\beta},\widehat\phi)$ be the completion of $(W,\beta,\phi)$, obtained by gluing
$$([0,\infty)\times \bdry W,e^\sigma\beta|_{\bdry W}),$$
with coordinates $(\sigma,x)$, to $(W,\beta)$ along $\bdry W$ and extending $\phi$ to $\widehat\phi(\sigma,x)=\sigma$. Let $\widehat h$ be the extension of $\wt h|_W$ from Section~\ref{subsection: extension} such that $\widehat{h}(\sigma,x)=(\sigma,\wt h|_{\bdry W}(x))$ on $[0,\infty)\times \bdry W$.  Next we set
$$\widehat {\bs a} = {\bs a}\cup ([0,\infty)\times \bdry{\bs a}).$$
Note that $\bdry {\bs a}$ is a Legendrian submanifold of $(\bdry W,\beta|_{\bdry W})$ and that, by Condition (4) in the definition of $f$ in Section~\ref{subsection: extension}, $[-\varepsilon,\infty)\times \bdry {\bs a}$ is a cylinder over the Legendrian $\bdry {\bs a}$.

We are initially in a Morse-Bott situation where there is an $S^1$-family of short Reeb chords from $\bdry a_i$ to $\wt h(\bdry a_i)$ for each $i$, since $\wt h(\bdry a_i)$ is a positive Reeb pushoff of $\bdry a_i$ on $\bdry W$.  Here we are measuring the length of a Reeb chord $c$ using the {\em action} $\mathfrak{A}_{\beta_{\bdry W}}(c)= \int_c \beta_{\bdry W}$.  We then perturb $\wt h(\bdry {\bs a})$ (without changing its name) so that for each $i=1,\dots,\kappa$ there are two short Reeb chords from $\bdry a_i$ to $\wt h(\bdry a_i)$.  The longer (resp.\ shorter) of the two will be denoted by $\hat{x}_i$ (resp.\ $\check{x}_i$); see Figure~\ref{fig: legendrian-new}.

\begin{figure}[ht]
	\begin{overpic}[scale=.6]{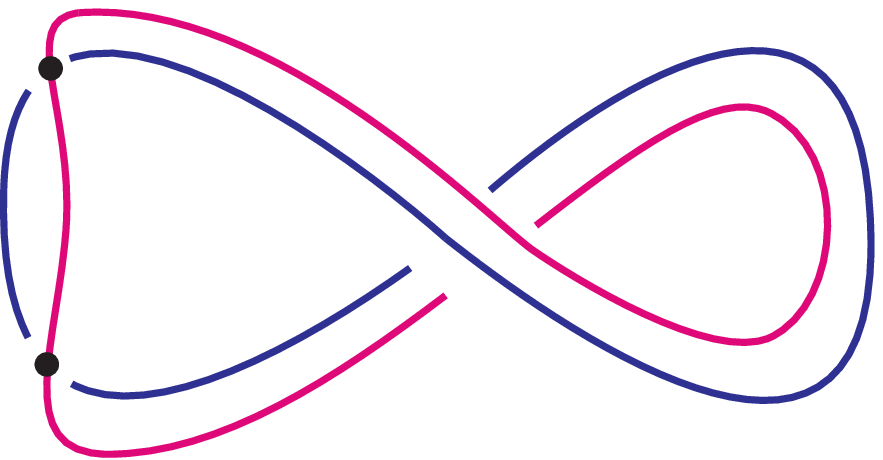}
		\put(-1,7){\tiny $\hat{x}_i$}
		\put(-1,45){\tiny $\check{x}_i$}
		\put(101.2,25){\tiny $\bdry a_i$}
		\put(24.7,50){\tiny $\wt h(\bdry a_i)$}
	\end{overpic}
	\caption{Legendrian $\bdry a_i$ and its pushoff $\wt h(\bdry a_i)$ in the Lagrangian projection.}
	\label{fig: legendrian-new}
\end{figure}

\subsection{Generators}

We modify the set of generators $\mathcal{S}$ as follows:  Given $\mathbf{y}=\{y_1,\dots,y_\kappa\}\in \mathcal{S}$, let $\widehat{\mathbf{y}}$ be the result of replacing all the contact class intersection points $x_i$ by $\hat{x}_i$ in $\mathbf{y}$. We then set $\widehat{\mathcal{S}}=\{\widehat{\mathbf{y}}~|~\mathbf{y}\in \mathcal{S}\}$.  Let us write $\widehat{\mathbf{y}}=\widehat{\mathbf{y}}_0\sqcup\widehat{\mathbf{y}}_1$, where none of the terms of $\widehat{\mathbf{y}}_0$ are of the form $\hat{x}_i$ and all of the terms of $\widehat{\mathbf{y}}_1$ are of the form $\hat{x}_i$.  As in Section~\ref{subsubsection: generators}, $\widehat{\mathcal{S}}$ admits a splitting
$$\widehat{\mathcal{S}}=\textstyle\coprod_{\Gamma\in H_1(M;\Z)} \widehat{\mathcal{S}}.$$

\subsection{Symplectic fibrations}

Let
$$\pi_{\R\times[0,1]}: \widehat X_W=\R\times[0,1]\times \widehat W \to \R\times [0,1]$$
be a symplectic fibration with fiber $\widehat W$, where $(s,t)$ are coordinates on $\R\times[0,1]$, and let $\pi_{\widehat W}$ be the projection $\widehat X_W\to\widehat W$.  We define the Lagrangian submanifolds
$$L_{0,\widehat {\bs a}}= \R\times \{1\}\times \widehat{\bs a},\quad L_{1,\widehat{h}(\widehat{\bs a})}= \R\times\{0\}\times \widehat{h}(\widehat{\bs a}).$$

Similarly let
$$\widehat X_{\R\times \bdry W}:=\R\times[0,1]\times (\R\times \bdry W)\to \R\times[0,1]$$
be a symplectic fibration whose fiber $\R\times \bdry W$ is a symplectization of $(\bdry W,\beta|_{\bdry W})$.  When we stretch a holomorphic curve in $\widehat X$ along $\R\times[0,1]\times \bdry W$, i.e., do a ``sideways stretching'', the sideways levels that we obtain are $\widehat X_W$, $\widehat X_H:=\R\times [0,1]\times \widehat H$, and several levels of $\widehat X_{\R\times\bdry W}$ in between. Here $\widehat H$ refers to the completion of $H$ obtained by attaching positive and negative ends. Refer to Figure~\ref{fig: handlestretching}.

\begin{figure}[ht]
	\begin{overpic}[scale=.6]{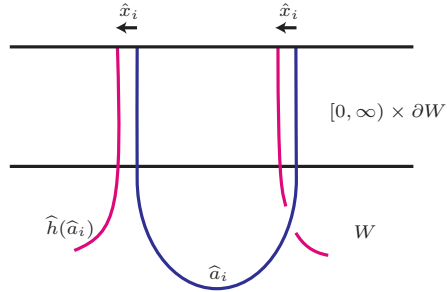}
		\put(49.3,3){\tiny $\widehat{a}_i$}
		\put(9,14){\tiny $\widehat{h}(\widehat{a}_i)$}
		\put(66.6,68){\tiny $\hat{x}_i$}
		\put(27,68){\tiny $\hat{x}_i$}
		\put(79,43){\tiny $[0,\infty)\times \bdry W$}
		\put(85,13){\tiny $W$}

	\end{overpic}
	\caption{Schematic picture indicating the result of stretching $\wt W$ along $\bdry W$. $\hat{x}_i$ indicates the possible locations of $\hat{x}_i$.}
	\label{fig: handlestretching}
\end{figure}

\subsection{Almost complex structures}

Let $J_{\widehat W}$ be an almost complex structure on $\widehat W$ which is tamed by $d\widehat\beta$ and is adapted to $\beta_{\bdry W}$ on the symplectization end $\widehat W- W$. The space of such almost complex structures $J_{\widehat W}$ will be denoted by $\mathcal{J}_{\widehat W}$.  We define $J_{\widehat{X}_W}=J_{\R\times[0,1]}\times J_{\widehat W}$ where $J_{\widehat W}\in \mathcal{J}_{\widehat W}$ and a $C^\infty$-small perturbation $J^\Diamond_{\widehat X_W}$ of $J_{\widehat X_W}$ as in Section~\ref{subsection: almost complex structures} so that (J1)--(J3) are satisfied with $\widehat\Sigma$ replaced by $\widehat W$. Let $\mathcal{J}_{\widehat X_W}$ and $\mathcal{J}^\Diamond_{\widehat X_W}$ be the space of such $J_{\widehat X_W}$ and $J^\Diamond_{\widehat X_W}$. We can similarly define $J_{\R\times \bdry W}$, $J_{\widehat{X}_{\R\times \bdry W}}=J_{\R\times[0,1]}\times J_{\R\times \bdry W}$, and spaces $\mathcal{J}_{\widehat{X}_{\R\times \bdry W}}$ and $\mathcal{J}^\Diamond_{\widehat X_{\R\times \bdry W}}$.

\subsection{Moduli spaces} \label{subsection: moduli spaces for variant}

When a curve in $\mathcal{M}_{J^\Diamond}({\bf y}, {\bf y}')$ is stretched along $\widehat X_{\R\times \bdry W}$, almost all of the essential information is contained in the $\widehat X_W$ part.

Let $(s_1,t_1)$ be coordinates on the infinite strip $\R\times[0,1]$ with the standard complex structure and let $J_{\widehat X_{\R\times \bdry W}}=J_{\R\times[0,1]}\times J_{\R\times \bdry W}\in \mathcal{J}_{\widehat X_{\R\times \bdry W}}$.

\begin{defn}
	Let $c$ be a Reeb chord of $(\bdry W, \beta|_{\bdry W})$. A holomorphic map
	$$v:\R\times[0,1]\to \widehat X_{\R\times \bdry W}$$
	is a {\em diagonal strip over a Reeb chord $c$} if
	$$\pi_{\R\times[0,1]}\circ v(s_1,t_1)=(s_1+C,t_1),\mbox{ or } (-s_1+C,1-t_1)$$
	for some $C\in \R$ and $\pi_{\R\times \bdry W}\circ v$ is a trivial strip over $c$.
\end{defn}

\n
{\em Definition of $\mathcal{M}(\widehat{\bf y},\widehat{\bf y}')$.}
Let $J^\Diamond_{\widehat X_W}\in (\mathcal{J}^\Diamond_{\widehat X_W})^{reg}$. We define $\mathcal{M}(\widehat{\bf y},\widehat{\bf y}')$, where $$\widehat{\bf y}=\widehat{\mathbf{y}}_0\sqcup\widehat{\mathbf{y}}_1 \quad \mbox{ and } \quad \widehat{\bf y}'=\widehat{\mathbf{y}}_0'\sqcup \widehat{\mathbf{y}}_1'.$$
If $\widehat{\mathbf{y}}_1\not \subset \widehat{\mathbf{y}}_1'$, then we set $\mathcal{M}(\widehat{\bf y},\widehat{\bf y}')=\varnothing$.    On the other hand, if $\widehat{\mathbf{y}}_1 \subset \widehat{\mathbf{y}}_1'$, then
$$\mathcal{M}(\widehat{\bf y},\widehat{\bf y}'):= \mathcal{M}_{J^\Diamond_{\widehat X_W}}(\widehat{\bf y}_0,\widehat{\bf y}'-\widehat{\bf y}_1)\times \mathcal{M}_{J^\Diamond_{\widehat X_H}}(\widehat{\bf y}_1,\widehat{\bf y}_1),$$
where $\mathcal{M}_{J^\Diamond_{\widehat X_W}}(\widehat{\bf y}_0,\widehat{\bf y}'-\widehat{\bf y}_1)$ is the set of holomorphic maps
$$\widehat u_0: (\dot F_0,j_0)\to (\widehat X_W, J^\Diamond_{\widehat X_W}),$$
where we range over all $(\dot F_0,j_0)$ such that $(F_0,j_0)$ is a compact Riemann surface with boundary, ${\bf p}_+$ and ${\bf p}_-$ are disjoint sets of boundary punctures of $F_0$, and $\dot F_0=F_0-{\bf p}_+-{\bf p}_-$, and:
\begin{enumerate}
	\item each component of $\bdry F_0- {\bf p}_+-{\bf p}_-$ maps to some component of $L_{1,\widehat {\bs a}}$ or $L_{0,\widehat{h}(\widehat{\bs a})}$ and each component of $L_{1,\widehat {\bs a}}$ and $L_{0,\widehat{h}(\widehat{\bs a})}$ is used at most once;
	\item at the positive end $s\gg 0$, $\widehat u_0$ maps the neighborhoods of the punctures of ${\bf p}_+$ asymptotically to strips over Reeb chords $[0,1]\times \widehat{\mathbf{y}}_0$;
	\item at the negative end $s\ll 0$, $\widehat u_0$ maps the neighborhoods of the punctures of ${\bf p}_-$ asymptotically to strips over Reeb chords $[0,1]\times \widehat{\mathbf{y}}'_0$ or to diagonal strips over $\widehat{\mathbf{y}}'_1-\widehat{\mathbf{y}}_1$;
\end{enumerate}
and $\mathcal{M}_{J^\Diamond_{\widehat X_H}}(\widehat{\bf y}_1,\widehat{\bf y}_1)$ is a one-element set $\{\widehat{u}_2\}$ where $\widehat{u}_2$ is the union of trivial strips in $\widehat X_H$ from ${\bf y}_1$ to ${\bf y}_1$.
Note that $\widehat u_0$ has $\# \widehat{\mathbf{y}}_0$ positive ends and $\# \widehat{\mathbf{y}}_0$ negative ends and $\widehat u_2$ has $\#\widehat{\bf y}_1$ positive ends and $\#\widehat {\bf y}_1$ negative ends.

As in Section~\ref{subsubsection: moduli spaces},
$\mathcal{M}(\widehat{\bf y},\widehat{\bf y}')$ admits a splitting
$$\mathcal{M}(\widehat{\bf y},\widehat{\bf y}')= \textstyle\coprod_{\chi\in \Z, A\in \mathcal{A}_{{\bf y},{\bf y}'}} \mathcal{M}^{A,\chi} (\widehat{\bf y},\widehat{\bf y}'),$$
where $\mathcal{A}_{{\bf y},{\bf y}'}$ is defined analogously.

\subsection{Definition of the variant}

We now define the variant
$$(C',d'):=\widehat{CF}'(W,\beta,\phi;h; J^\Diamond_{\widehat X_W})$$
of $\widehat{CF}(W,\beta,\phi;h;J^\Diamond)$, where $C'$ is the free $\Lambda\llbracket\hslash\rrbracket$-module generated by $\widehat{\mathcal{S}}$ and $\Lambda$ is the Novikov ring with $\Z$-coefficients over $H_2(M;\Z)$. The differential $d' \widehat{\mathbf{y}}$ is as given in Equation~\eqref{eqn: differential}, with ${\bf y}$ and ${\bf y}'$ replaced by $\widehat{\bf y}$ and $\widehat{\bf y}'$ and the moduli space $\mathcal{M}_{J^\Diamond}^{\op{ind}=1,A,\chi}({\bf y},{\bf y}')/\R$ replaced by $\mathcal{M}^{\op{ind}=1,A,\chi}(\widehat{\bf y},\widehat{\bf y}')/\R$.  In particular, when $\widehat{\bf y}=\widehat{\bf y}_1$, i.e., $\widehat{\bf y}_0=\varnothing$, then $\mathcal{M}^{\op{ind}=1}(\widehat{\bf y},\widehat{\bf y}')=\varnothing$ since $\mathcal{M}(\widehat{\bf y},\widehat{\bf y}')$ is a one-element set which is the union of trivial strips.

\begin{notation}
We often write $\widehat{CF}(\widehat{h}(\widehat{\bs a}),\widehat{\bs a})$ as shorthand for $\widehat{CF}'(W,\beta,\phi;h; J^\Diamond_{\widehat X_W})$.
\end{notation}

\begin{lemma}
	$(d')^2=0$.
\end{lemma}

\begin{proof}
	This follows from analyzing the degenerations of $\mathcal{M}^{\op{ind}=2,A,\chi}(\widehat{\bf y},\widehat{\bf y}')/\R$ and properly keeping track of components which are diagonal strips over $\hat{x}_i$ in $\widehat X_1$.
\end{proof}

\subsection{Equivalence with original definition}

The goal of this subsection is to prove the following equivalence result:

\begin{thm} \label{thm: alternative version}
Suppose the function $f$ in the definition of $\wt{h}$ is sufficiently $C^r$-small for $r\gg 0$. Then there is an isomorphism
$$\widehat{CF}'(W,\beta,\phi;h; J^\Diamond_{\widehat X_W})\simeq \widehat{CF}(W,\beta,\phi;h;J^\Diamond)$$
of cochain complexes for some $J^\Diamond\in (\mathcal{J}^\Diamond)^{reg}$.
\end{thm}

\begin{proof}
	This follows from stretching $\wt W$ along $\bdry W$.
	
	Let $[-\varepsilon,\varepsilon]\times \bdry W$ be a fixed collar neighborhood of $\bdry W$ in $\wt W$ with coordinates $(\sigma,x)$.  Let $J_{\wt W^\wedge}^i\in \mathcal{J}_{\wt W^\wedge}^{reg}$, $i=1,2,\dots$, be a sequence of almost complex structures on ${\wt W^\wedge}$ that satisfy the following:
	\begin{itemize}
		\item $J_{\wt W^\wedge}^i$ agrees with a fixed $J_{\wt W^\wedge}\in \mathcal{J}_{\wt W^\wedge}^{reg}$ on $[-\varepsilon,\varepsilon]\times \bdry W$;
		\item on $[-{\varepsilon\over 2},{\varepsilon\over 2}]\times \bdry W$, $J_{\wt W^\wedge}^i$ is adapted to $\beta|_{\bdry W}$ with one slight modification: $J_{\wt W^\wedge}^i(\tfrac{1}{i}\bdry_\sigma)=R$ where $R$ is the Reeb vector field for $\beta|_{\bdry W}$;
	\end{itemize}
	We then set $J^i=J_{\R\times[0,1]}\times J^i_{\wt W^\wedge}$.  For questions about compactness, it suffices to use the split almost complex structures $J^i$ instead of the perturbed ones.

	Let $u_i\in \mathcal{M}_{J^i}^{\op{ind}=1,A,\chi}({\bf y},{\bf y}')$, $i=1,2,\dots$, be a sequence of holomorphic maps.
	We have the following compactness result:
	
	\begin{lemma} \label{lemma: A}
		Suppose the function $f$ in the definition of $\wt{h}$ is sufficiently $C^r$-small for $r\gg 0$.
		Then $u_i$ limits to $u_\infty=u_{\infty 1}\cup u_{\infty 2}$, where
		$$u_{\infty 1}\in \mathcal{M}_{J^\Diamond_{\widehat X_W}}^{\op{ind}=1,A,\chi}(\widehat{\bf y}_0,\widehat{\bf y}'-\widehat{\bf y}_1)$$
		and $u_{\infty 2}$ is the union of strips in $\widehat X_H$ from $\hat{x}_i$ to $x_i$ when $x_i\in \widehat{\bf y}'_1-\widehat{\bf y}_1$ and trivial strips in $\widehat X_H$ from $x_i$ to $x_i$ when $x_i\in \widehat{\bf y}_1$.
	\end{lemma}
	
	\begin{proof}
		Consider the sequence $v_i=\pi_{\wt W^\wedge}\circ u_i$.  Then, by the usual SFT compactness theorem~\cite{EGH}, $v_i$ limits to a multi-level building $v_\infty=v_{\infty 1}\cup \dots\cup v_{\infty l}$, where $v_{\infty 1}$ maps to $\widehat W$, $v_{\infty 2},\dots,v_{\infty, l-1}$ map to $\R\times \bdry W$, and $v_{\infty l}$ maps to $\widehat H$.  The ends of $v_{\infty,j}$ a priori limit to Reeb chords of $R$ from $\bdry {\bs a}$ to $\wt{h}(\bdry {\bs a})$ and closed orbits of $R$.
		
		We claim that the negative ends of $v_{\infty,l}$ can only be short Reeb chords for $f$ with sufficiently small $C^r$-norm.  If the claim fails to hold, then, since there are no positive punctures and $f$ has arbitrarily small $C^r$-norm, the presence of long negative Reeb chords and closed Reeb chords contradicts the usual energy bounds involving $\beta|_H$.
		
		The claim, together with action considerations, then implies that all the ends of $v_{\infty j}$, $j=1,\dots,l$, are short Reeb chords. By area and Fredholm index considerations, all the components of $v_{\infty j}$, $j=2,\dots,l-1$, are trivial strips over short Reeb chords, which is equivalent to saying that they can be discarded. A similar argument implies that $v_{\infty l}$ is a union of strips from $x_i$ to $\hat x_i$ (here $x_i$ is an intersection point and $\hat x_i$ is a negative puncture with respect to $\widehat H$).
		
		The above description of $v_\infty$ then implies the lemma.
	\end{proof}
	
	\begin{lemma} \label{lemma: B}
		$\op{ind}(u_i)=1$ if and only if $\op{ind}(u_{\infty 1})=1$.
	\end{lemma}
	
	\begin{proof}
		Each component of $u_{\infty 2}$ has index $1$.  Since the gluing/incidence condition for gluing a component of $u_{\infty 2}$ to $u_{\infty 1}$ is a codimension $1$ condition, $\op{ind}(u_i)=1$ is equivalent to $\op{ind}(u_{\infty 1})=1$.
	\end{proof}
	
	\begin{lemma} \label{lemma: C}
		The mod $2$ count of strips from $\hat{x}_i$ to $x_i$ in $\widehat{X}_H$ modulo translation is one.
	\end{lemma}
	
	\begin{proof}
		This is equivalent to the assertion that the mod $2$ count of strips from $x_i$ to $\hat{x}_i$ in $\widehat{H}$ is one.  Let $(\bs{\wt a}\cap H)^\wedge$ be the completion of $\bs{\wt a}\cap H$  in $\widehat{H}$, obtained by attaching cylindrical ends. By Gromov compactness, for sufficiently small $f$, all the strips are contained in a neighborhood of $(\bs{\wt a} \cap H)^\wedge$.   Hence it suffices to do a model calculation on $T^* S^n$, where $\bs{\wt a}$ is the $0$-section, $\wt{h}(\bs{\wt a})$ is the graph of $df$, $f$ is sufficiently small, $df=0$ at the north and south poles $p_n,p_s$, and we are stretching along the equator.  By the usual argument from \cite[Proposition 1]{Fl}, for any (generic) point $q\in \bs{\wt a}$, there is a unique holomorphic strip from $p_n$ to $p_s$ that passes through $q$.  When we stretch along the neck and view the equator as an $S^1$-family of Reeb orbits, the count of strips from $p_n$ to a generic point on the $S^1$-family is $1$,  implying the lemma.
	\end{proof}
	
	In view of Lemma~\ref{lemma: C} and the definition of $\mathcal{M}(\widehat{\bf y}, \widehat{\bf y}')$ from Section~\ref{subsection: moduli spaces for variant}, there is a bijection
	$$\mathcal{M}_{J^\Diamond}^{\op{ind}=1,A,\chi}({\bf y},{\bf y}')/\R\simeq\mathcal{M}^{\op{ind}=1,A,\chi}(\widehat{\bf y},\widehat{\bf y}')/\R,$$
	which implies the theorem.
\end{proof}

\section{Partial invariance} \label{section: partial invariance}

In this section we prove Theorem~\ref{thm: invariance or lack thereof}, i.e., invariance of the higher-dimensional Heegaard Floer homology groups with respect to:
\begin{enumerate}
	\item[(I1)] trivial Weinstein homotopies;
	\item[(I2)] changes of almost complex structure; 
	\item[(I3)] changes of symplectomorphism within its symplectic isotopy class.
\end{enumerate}
We will work out (I1) and (I2) in Section~\ref{subsection: I1 I2}; the invariance under (I3) is similar and will be omitted.  We will leave the non-invariance with respect to handleslides (I4) to the reader.

In Section~\ref{subsection: I5} we also discuss invariance under elementary stabilizations (cf.\ Definition~\ref{defn: elementary}).

\subsection{Invariance under (I1) and (I2)} \label{subsection: I1 I2}

Let $(W,\beta_s,\phi_s)$, $s\in[0,1]$, be a trivial Weinstein homotopy with compatible Liouville vector fields $Y_s$. Let ${\bs a}^s =a^s_1\cup \dots\cup a^s_\kappa$ be the Lagrangian basis for $(W,\beta_s,\phi_s)$.  We attach a cap $H_s$ to $W$ along the Legendrian spheres $\{ \partial a^s_1 ,\dots,\partial a^s_\kappa \}$ in a smooth manner with respect to $s$. The manifold $W\cup H_s$ is diffeomorphic to a fixed compact manifold $\wt W$, which inherits a family $(\wt\beta_s,\wt \phi_s)$ of Weinstein structures and cappings ${\bs \alpha}^s=\wt{\bs a}^s$ of ${\bs a}^s$.

By Lemma~\ref{lemma: Liouville homotopy}(2), there exists a $1$-parameter family
$$g_s: (\wt W^\wedge,d\wt \beta_0^\wedge) \stackrel\sim\to (\wt W^\wedge ,d \wt\beta_s^\wedge)$$
of symplectomorphisms such that $g_s^* \wt{\beta}_s^\wedge=\wt\beta_0^\wedge$ outside of a compact set. Hence, after pulling back using $g_s$ and renaming, we may assume that:
\begin{itemize}
	\item[(H1)] each ${\bs \alpha}^s$ is a submanifold of a fixed symplectic manifold $(\wt W^\wedge,\omega=d\wt \beta_0^\wedge)$ with a cylindrical end; and
	\item[(H2)] the Weinstein structure $(\wt{\beta}_s^\wedge ,\wt{\phi}_s^\wedge)$ depends on $s$ only on a compact subset of $\wt W^\wedge$.
\end{itemize}
In particular, the Liouville vector field $Y_s$ is constant outside of the same compact set.

Consider the submanifolds
\begin{align*}
L'_{1,\bs\alpha} & =((-\infty ,0] \times \{ 1\}  \times \bs\alpha^0) \cup (\cup_{s\in [0,1]} (\{ s\} \times \{1\} \times \bs\alpha^s ))\\
& \qquad \qquad \cup([1, \infty ) \times \{ 1\} \times \bs\alpha^1),\\
L'_{0, \wt{h} (\bs\alpha)} & =((-\infty ,0] \times \{ 0\} \times \wt{h}(\bs\alpha^0)) \cup (\cup_{s\in [0,1]} (\{ s\} \times \{0\} \times \wt{h} (\bs\alpha^s ))) \\
& \qquad\qquad \cup ([1, \infty ) \times \{0\} \times \wt{h} (\bs\alpha^1)),
\end{align*}
of $\widehat X= \R \times [0,1]\times\wt W^\wedge$ which agree with $L_{1,\bs\alpha^1}$ and $L_{0,\wt{h}(\bs\alpha^1)}$ at the positive end and with $L_{1,\bs\alpha^0}$ and $L_{0,\wt{h}(\bs\alpha^0)}$ at the negative end.

\begin{lemma} \label{lemma: existence of widehat omega prime}
	There exists a (closed) symplectic connection $\widehat\Omega'$ on $\widehat X$ which satisfies the following:
	\begin{enumerate}
		\item on each fiber $\widehat\Omega'$ restricts to $\omega$;
		\item $\widehat\Omega'$ agrees with $\widehat\Omega$ at the positive and negative ends of $\widehat X$; and
		\item $L'_{1,\bs\alpha}$ and $L'_{0,\wt{h} (\bs\alpha)}$ are Lagrangian.
	\end{enumerate}
\end{lemma}

\begin{proof}
	Let $H_{s,t}$, $s,t\in[0,1]$, be a $2$-parameter family of Hamiltonian functions on $\wt W^\wedge$ with $H_{s,t}=0$ outside of a compact region of $[0,1]_s\times [0,1]_t\times \wt W^\wedge$.  We also view the family $H_{s,t}$ as a function $H$ on $\widehat X$.  Let $\psi_{s_0,t_0}$ be the corresponding family of Hamiltonian diffeomorphisms, obtained by integrating the vector field corresponding to $d_{\wt W^\wedge} H_{s,t}$ along the segment $\{t=t_0,0\leq s\leq s_0\}$. (We are taking $\psi_{s_0,t_0}=\op{id}$ for $s_0\leq 0$.) Here $d_{\wt W^\wedge}$ is the differential in the $\wt W^\wedge$-direction.  Recall that we are using the convention that the Hamiltonian vector field $X_g$ of a function $g$ with respect to a symplectic form $\omega$ satisfies $dg=i_{X_g} \omega$.
	
	We may choose $H_{s,t}$ such that $\psi_{s,t}$ satisfies the following:
	\be
	\item[(A)] $\psi_{s,t}= \op{id}$ for $s\leq 0$ and $\psi_{s,t}$ is independent of $(s,t)$ for $s\geq 1$;
	\item[(B)] $\psi_{s,1}(\bs\alpha^0)=\bs\alpha^s$ and $\psi_{s,0}(\wt{h}(\bs\alpha^0))=\wt{h}(\bs\alpha^s)$ for $s\in[0,1]$.
	\ee
	This is possible since the $\bs\alpha^s$ are exact.
	
	We then set
	\begin{equation} \label{eqn: defn of Omega prime}
	\widehat\Omega'=ds\wedge dt +\omega+ dH \wedge ds =ds\wedge dt +\omega+d_{\wt W^\wedge}H_{s,t}\wedge ds-\tfrac{\bdry H_{s,t}}{\bdry t}ds\wedge dt.
	\end{equation}
	$\widehat\Omega'$ is closed. (1) and (2) are immediate.  (3) uses (A) and (B) and is left to the reader.
\end{proof}

Let $J_i=J_{\R\times[0,1]}\times J_{\wt W^\wedge,i}$, $i=0,1$, be $(\wt{W}^\wedge,\wt{\beta}_i^\wedge)$-compatible almost complex structures on $\widehat X$ and let $J_i^\Diamond$ be their perturbations which we assume to be regular. Also let
$$\mathcal{E}=[1,\infty)\times\bdry \wt W \subset \wt W^\wedge.$$
(Here we are assuming that the boundary of $\wt W$ is $\{1\}\times \bdry \wt W$. Hence $\wt W^\wedge - \wt W=[0,\infty)\times \bdry \wt W$ is slightly larger than $\mathcal{E}$.)
By (H2) we may assume without loss of generality that $\wt \beta_s^\wedge$, $s\in[0,1]$, is independent of $s$ on $\mathcal{E}$.

\s\n
{\em Case 1.} First suppose that  $J_0^\Diamond=J_1^\Diamond$ on $\R_s\times[0,1]_t\times \mathcal{E}$.   Let $J'$ be an $\widehat\Omega'$-tame almost complex structure on $\widehat X$ which satisfies the following:
\begin{enumerate}
	\item[(K1)] $J'$ agrees with $J_1$ at the positive end ($s\gg 0$) and with $J_0$ at the negative end ($s\ll 0$);
	\item[(K2)] the map $\pi: \widehat X\to \R\times[0,1]$ is $(J',J_{\R\times[0,1]})$-holomorphic; and
	\item[(K3)] $J'$ is $\R\times[0,1]$-invariant on $\R\times[0,1]\times\mathcal{E}$;
\end{enumerate}
and let $(J')^\Diamond$ be its perturbation which agrees with $J_0^\Diamond=J_1^\Diamond$ at the positive and negative ends and which we assume to be regular.

The count of index $0$ holomorphic curves in $(\widehat X,\widehat\Omega',(J')^\Diamond)$ with boundary on $L'_{1,\bs\alpha}$ and $L'_{0,\wt{h}(\bs\alpha)}$ yields a map:
\begin{equation} \label{equation: map 1}
\mathcal{I}_{10} : \widehat{CF} (W,\beta_1,\phi_1;h;J_1^\Diamond)\to \widehat{CF} (W,\beta_0 ,\phi_0;h;J_0^\Diamond).
\end{equation}
The compactness of the relevant moduli space is proved as in Lemma~\ref{lemma: compactness}. In particular, curves cannot enter the region $\R\times[0,1]\times \mathcal{E}$ by the maximum principle. Reversing the roles of $(\beta_0 ,\phi_0 )$ and $(\beta_1 , \phi_1 )$, we obtain the inverse map:
\begin{equation} \label{equation: map 2}
\mathcal{I}_{01} : \widehat{CF} (W,\beta_0 ,\phi_0;h;J_0^\Diamond)\to \widehat{CF} (W,\beta_1,\phi_1;h;J_1^\Diamond).
\end{equation}

\begin{lemma} \label{lemma: chain map 1}
	The maps $\mathcal{I}_{10}$ and $\mathcal{I}_{01}$ are cochain maps.
\end{lemma}

\begin{proof}
	Similar to the proof of Lemma~\ref{lemma: bdry squared equals zero}.
\end{proof}

\begin{lemma} \label{lemma: homotopy inverse 1}
	The maps $\mathcal{I}_{10}$ and $\mathcal{I}_{01}$ are homotopy inverses.
\end{lemma}

\begin{proof}
	There exists a family $(\widehat\Omega''_\eta,L''_{1,\bs\alpha,\eta}, L''_{0,\widehat{h}(\bs\alpha),\eta},(J'')^\Diamond_\eta)$ of cobordisms parametrized by $\eta\in[0,1)$ such that:
	\begin{itemize}
		\item $(\widehat\Omega''_0,L''_{1,\bs\alpha,0}, L''_{0,\widehat{h}(\bs\alpha),0},(J'')^\Diamond_0)$ is a trivial cobordism;
		\item $(\widehat\Omega''_\eta,L''_{1,\bs\alpha,\eta}, L''_{0,\widehat{h}(\bs\alpha),\eta},(J'')^\Diamond_\eta)$ coincides with $(\widehat\Omega,L_{1,\bs\alpha^1},L_{0,\wt{h}(\bs\alpha^1)},J_1^\Diamond)$ at the positive and negative ends; and
		\item as $\eta\to 1$, $(\widehat\Omega''_\eta,L''_{1,\bs\alpha,\eta}, L''_{0,\widehat{h}(\bs\alpha),\eta},(J'')^\Diamond_\eta)$ limits to a $2$-level building where the top level is the cobordism for $\mathcal{I}_{10}$ and the bottom level is the cobordism for $\mathcal{I}_{01}$.
	\end{itemize}
	By the usual chain homotopy argument, $\mathcal{I}_{01}\circ \mathcal{I}_{10}$ is chain homotopic to the identity.  The case of $\mathcal{I}_{10}\circ \mathcal{I}_{01}$ is similar.
\end{proof}

\s\n
{\em Case 2.}  Next suppose that $\wt \beta_s^\wedge$, $s\in[0,1]$, is independent of $s$ on all of $\wt W^\wedge$. Let $J_{\wt W^\wedge,s}$, $s\in[0,1]$, be a family of almost complex structures on $\wt W^\wedge$ compatible with $\wt\beta_s^\wedge$ (for any $s$) from $J_{\wt W^\wedge,0}$ to $J_{\wt W^\wedge,1}$ such that:
\begin{itemize}
	\item $J_{\wt W^\wedge,s}=J_{\wt W^\wedge,0}$ on $\mathcal{E}$ for all $s\in[0,\tfrac{1}{2}]$; and
	\item $J_{\wt W^\wedge,s}=J_{\wt W^\wedge,1}$ on $\wt W$ for all $s\in[\tfrac{1}{2},1]$.
\end{itemize}
Then let $J_s= J_{\R\times[0,1]}\times J_{\wt W^\wedge,s}$, $s\in[0,1]$, and let $J_s^\Diamond$ be its regular perturbation.

Using $J_s$, $s\in[\tfrac{1}{2},1]$, we can construct a cobordism which restricts to the trivial cobordism on $\R\times[0,1]\times \wt W$ and hence induces a quasi-isomorphism
$$\mathcal{I}_{1,1/2}:\widehat{CF} (W,\beta_1,\phi_1;h;J_1^\Diamond)\to \widehat{CF} (W,\beta_{1/2},\phi_{1/2};h;J_{1/2}^\Diamond),$$
by the maximal principle. In other words, if $u$ is a curve that is counted in $\mathcal{I}_{1,1/2}$, then $\pi_{\wt W}\circ u$ has image inside $\wt W$. Since we can also define a quasi-isomorphism
$$\mathcal{I}_{1/2,0}: \widehat{CF} (W,\beta_{1/2},\phi_{1/2};h;J_{1/2}^\Diamond)\to \widehat{CF} (W,\beta_0,\phi_0;h;J_0^\Diamond)$$
by Case 1, the invariance with respect to (I1) and (I2) follows.

\subsection{Invariance under elementary stabilizations} \label{subsection: I5}

Let
$$S^\pm_{c}(W,\beta,\phi;h)=(W',\beta',\phi';h')$$
be the positive/negative stabilization of $(W,\beta,\phi;h)$ along a Lagrangian disk $c\subset W$.

\begin{defn} \label{defn: elementary}
	An {\em elementary stabilization} is a stabilization where $c$ is disjoint from the Lagrangian basis ${\bs a} =a_1\cup\dots\cup a_\kappa$ of $W$ associated with the Weinstein structure $(\beta,\phi)$.
\end{defn}

We may assume that the handle $W_0=W'-W$ is attached to $W$ along the Legendrian sphere $\partial c$ and that $W_0$ is disjoint from $H$.  Hence we can attach the handles of $H$ and $W_0$ separately in any order. Let $H_0$ be an $n$-handle attached to the boundary of the cocore of $W_0$ and let $H'=H_0\cup H$.  We then obtain
$$\wt{W}':= W'\cup H'.$$
The cocore of $W_0$ together with the core of $H_0$ form a Lagrangian sphere $\wt a_0$. As before, we first extend $h'$ from $W'$ to $\wt W'$ by the identity and then perturb the extension by the Hamiltonian flow of a function $f_i$ near $\wt a_i$, for $i=0,\dots,\kappa$.  The result is a symplectomorphism $\wt{h}' \in \op{Symp}(\wt W',\bdry \wt W',d\beta')$. There is a unique point $x_0$ (resp.\ $z_0$) of $\wt a_0 \cap \wt{h}'(\wt a_0)$ in $H'$ corresponding to the minimum (resp.\ maximum) of $f_0$ in the case of a positive (resp.\ negative) stabilization.

Let $(J')^\Diamond$ and $J^\Diamond$ be perturbed $(\widehat W',\widehat\beta')$- and $(\widehat W,\widehat \beta)$-compatible almost complex structures on $\widehat X_{\wt W'}=\R\times[0,1]\times{(\wt W')^\wedge}$ and ${\widehat X}=\R\times[0,1]\times {(\wt W)^\wedge}$. We then have the following:

\begin{thm}[Invariance under elementary stabilizations]\label{thm: invariance under stabilization} $\mbox{}$
	\begin{enumerate}
		\item If $S^+_{c}(W,\beta,\phi;h)=(W',\beta',\phi';h')$ is an elementary stabilization, then the map
		$$\Theta^+ : \widehat{CF} (W,\beta,\phi;h;J^\Diamond)\to \widehat{CF} (W',\beta',\phi';h';(J')^\Diamond)$$
		$${\bf y}\mapsto \{x_0\}\cup {\bf y},$$
		is a quasi-isomorphism. It takes the contact class $[\mathbf{x}]$ to the contact class $[\{x_0\}\cup \mathbf{x}]$.
		\item If $S^-_{c}(W,\beta,\phi;h)=(W',\beta',\phi';h')$ is an elementary stabilization, then the map
		$$\Theta^- : \widehat{CF} (W,\beta,\phi;h;J^\Diamond)\to \widehat{CF} (W',\beta',\phi';h';(J')^\Diamond)$$
		$${\bf y}\mapsto \{z_0\}\cup {\bf y},$$
		is a quasi-isomorphism.
	\end{enumerate}
\end{thm}

\begin{proof}
	We prove (1); (2) is similar. Note that $\wt a_0$ does not intersect any $\wt{h}'(\wt a_i)$ for $i=1,\dots,\kappa$ since $h'$ was defined to be $(h\cup \op{id}|_{W_0})\circ \tau_\gamma$ (not the other way around).  Hence the only intersection between $\wt a_0$ and $\cup_{i=0}^\kappa \wt{h}'(\wt a_i)$ is $x_0$.
	
	We will see in Lemma~\ref{lemma: contact cycle} that, for a good choice of almost complex structure $(J')^\Diamond$, the only curve with $x_0$ at the positive end is a trivial strip. Hence any curve $u$ from $ \{x_0\}\cup {\bf y}$ to  $\{x_0\}\cup {\bf y}'$ consists of a trivial strip from $x_0$ to itself and other curves from ${\bf y}$ to ${\bf y}'$ that do not involve $x_0$. 
\end{proof}

\section{The contact class} \label{section: defn of contact class}

The goal of this section is to define the {\em $p$-twisted contact class
$$c^p (W,\beta,\phi;h)\in \widehat{HF} (W,\beta ,\phi;h;J^\Diamond)$$
of a Weinstein open book decomposition $(W,\beta,\phi;h)$} for $p\in \Z$ and study its properties.  Although $\widehat{HF} (W,\beta ,\phi;h;J^\Diamond)$ is not invariant under handleslides, there are nevertheless surprising applications of the contact class.

One of the main properties of the contact class is Theorem~\ref{thm: non-vanishing}, which gives a convenient method for certifying that certain contact manifolds are {\em not} Liouville fillable.  Theorem~\ref{thm: non-vanishing} provides large classes of contact manifolds that are not Liouville fillable.

\subsection{Definition of the contact class}

Recall the functions $f:\wt W\to \R$ and $f_i:\wt a_i\to \R$ from Section~\ref{subsection: extension}.  Let $x_i\in \wt a_i -a_i$ be the minimum of the Morse function $f_i$, viewed as a point of $\wt a_i\cap \wt{h} (\wt a_i)$ and satisfying $\wt{h}(x_i)=x_i$. 

\begin{defn}\label{defn: contact class}
The {\em contact class of the open book decomposition $(W,\beta,\phi;h)$} is
\begin{equation}\label{eqn: contact class}
c^0 (W,\beta,\phi;h) ={\bf x}= \{x_1,\dots,x_\kappa\} \in \widehat{CF} (W,\beta ,\phi;h;J^\Diamond).
\end{equation}
More generally, for $p\in \Z$, the {\it $p$-twisted} contact class of $(W,\beta,\phi;h)$ is
$$c^p (W,\beta,\phi;h) =\hbar^p {\bf x}.$$
\end{defn}

\begin{lemma}\label{lemma: contact cycle}
For all $p\geq 0$, the  $p$-twisted contact class $c^p (W,\beta,\phi;h)$ is a cycle for a certain choice of almost complex structure $J^\Diamond\in (\mathcal{J}^\Diamond)^{reg}$ and sufficiently $C^1$-small $f:\wt W\to \R$.
\end{lemma}

Hence the contact class can be viewed as an element of $\widehat{HF} (W,\beta ,\phi;h;J^\Diamond)$.  We often abuse notation and refer to both the cycle and the homology class as the ``contact class''.

Before proving Lemma~\ref{lemma: contact cycle} we review some notions from \cite{CE}.  Recall that two functions $\phi_0,\phi_1:\Sigma\to \R$ on a manifold $\Sigma$ are {\em target equivalent} if there is an increasing diffeomorphism $g:\R\stackrel\sim\to \R$ such that $\phi_1=g\circ \phi_0$.

According to \cite[Theorem 1.1(a)]{CE}, given a Weinstein domain $(\Sigma,\beta_0,\phi_0)$, there exists a Morse function $\phi_1$ which is target equivalent to $\phi_0$ and a complex structure $J_{\widehat\Sigma}$ on $\widehat\Sigma$ such that:
\begin{enumerate}
\item[(P1)] $\phi_1$ is plurisubharmonic with respect to $J_{\widehat\Sigma}$ on $\Sigma-N(\bdry\Sigma)$, where $N(\bdry\Sigma)$ is a small collared neighborhood of $\bdry\Sigma$;
\item[(P2)] $(\Sigma,\beta_0,\phi_1)$ is Weinstein homotopic to $(\Sigma,\beta_1,\phi_1)$ through a homotopy which fixes $\phi_1$, such that $\beta_1:= - d\phi_1\circ J_{\widehat\Sigma}$ on $\Sigma-N(\bdry\Sigma)$;
\item[(P3)] $J_{\widehat\Sigma}$ is adapted to $\beta_1|_{\bdry\Sigma}$ on $\widehat\Sigma-\Sigma$;
\item[(P4)] $x_i$, $i=1,\dots,\kappa$, are critical points of $\phi_1$ with $\phi_1$-value $1$ and all the other critical points of $\phi_1$ have $\phi_1$-value $<1$.
\end{enumerate}

Let $\wt{h}\in \op{Symp}(\wt W,\bdry\wt W,d\wt\beta_1)$ be the symplectomorphism from Section~\ref{subsection: extension} with $\wt\beta=\wt\beta_1$ and sufficiently $C^1$-small $f:\wt W\to \R$.  Let $\mathcal{J}$ be the space of $(\wt W^\wedge, \wt \beta_1^\wedge)$-compatible almost complex structures on $\widehat X=\widehat X_{\wt W}$.

\begin{proof}[Proof of Lemma~\ref{lemma: contact cycle}]
We show that there exist $J^\Diamond\in (\mathcal{J}^\Diamond)^{reg}$ and $f:\wt W\to\R$ such that no nonconstant $J^\Diamond$-holomorphic map is asymptotic to ${\bf x}$ at the positive end.

We may assume that $\wt{\bs a}\subset\{ \phi_1 \leq 1\}$.  If the $C^1$-norm of $f$ is sufficiently small, then we may also assume that $\wt{h} (\wt{\bs a})\subset \{ \phi_1 \leq 1\}$.

First consider the case $J^\Diamond =J=j_{\R\times[0,1]}\times J_{\wt W^\wedge}$, where the pair $(\phi_1, J_{\wt W^\wedge})$ satisfies (P1)--(P4) with $\Sigma=\wt W$.

We claim that each $x_i$ there exists a neighborhood $N(x_i)$ and a coordinate chart
$$\varphi_i:  N(x_i)\to \R^{2n}_{q_1,\dots,q_n,p_1,\dots,p_n},$$
such that $\varphi_i(x_i)=0$, the symplectic form is $\sum_j dq_j dp_j$, $J_{\wt W^\wedge}(\bdry_{q_j})=\bdry_{p_j}$, and $\varphi_i(\wt a_i)\subset \{p_1=\dots=p_n=0\}$. Indeed the proof of \cite[Theorem 1.1(a)]{CE} --- and in particular Step 1 of \cite[Proposition 13.12]{CE} --- gives such a neighborhood $N(x_i)$. We also take
$$f(q_1,\dots,q_n,p_1,\dots,p_n)= -\tfrac{\varepsilon}{2} \textstyle\sum_{j=1}^n q_j^2$$
on $N(x_i)$ for sufficiently small $\varepsilon>0$. Then $\varphi_i(\wt{h}(\wt a_i))$ is contained in the linear subspace spanned by $\bdry_{q_j}-\varepsilon \bdry_{p_j}$, $j=1,\dots,n$.

Arguing by contradiction, suppose there exists a nonconstant map
$$u : (\dot{F} ,j) \to (\widehat X,J)$$
which is asymptotic to ${\bf x}$ at the positive end. Since $J$ is a product, the map $u_{\wt W^\wedge}:=\pi_{\wt W^\wedge} \circ u$ is holomorphic. By Remark~\ref{rmk: curve contained in Sigma}, $\op{Im}(u_{\wt W^\wedge})\subset \Sigma$.  We may assume that $u_{\wt W^\wedge}$ is nonconstant near $x_i$. Consider the restriction $u_{\wt W^\wedge,i}$ of $u_{\wt W^\wedge}$ to $N(x_i)$.  Let $\pi_j: \R^{2n}\to \R^2$ be the projection to the $q_j p_j$-plane and let $v_j:= \pi_j\circ u_{\wt W^\wedge,i}$.\footnote{The letter $j$ is used to denote a complex structure on $\dot{F}$ and as a subscript; we hope this will not create any confusion.} The choices in the previous paragraph were made so that $\pi_j$ is holomorphic and that $\pi_j(\wt a_i)$ and $\pi_j(\wt{h}(\wt a_i))$ are real lines spanned by $\bdry_{q_j}$ and $\bdry_{q_j}-\varepsilon\bdry_{p_j}$, respectively.

The holomorphic map $v_j$ can be viewed as a map $(-\infty,0]_\sigma\times[0,1]_\tau\to \C$ which is dominated by a term of the form $c_je^{(\theta_0+m_j\pi)(\sigma+i\tau)}$ as $\sigma\to -\infty$, where $c_j\in \R-\{0\}$, $0<\theta_0<\pi$ is the angle corresponding to the vector $-\bdry_{q_j}+\varepsilon \bdry_{p_j}$, and $m_j$ is a nonnegative integer.  In other words, if $v_j$ is nonconstant, then the image of $v_j$ sweeps out a large (= angle $> {\pi\over 2}$) sector. The asymptotic description of $v_j$ then implies that $\op{Im}(u_{\wt W^\wedge})\cap \{ \phi_1 >1\}\not=\varnothing$. Now, since $u_{\wt W^\wedge}$ maps $\partial \dot{F}$ to $\wt{\bs a} \cup \wt{h}(\wt{\bs a}) \subset \{ \phi_1 \leq 1\}$, the map $\phi_1 \circ u_{\wt W^\wedge}$ attains a maximum in the interior of $\dot{F}$, contradicting the maximum principle.

Now consider the case where $J^\Diamond=J=j_{\R\times[0,1]}\times J_{\wt W^\wedge}$ on the subset $\R\times[0,1]\times (\wt W^\wedge-W)$.  Observe the following:
\begin{enumerate}
\item there are no other intersection points of $\wt a_i\cap \wt{h}(\wt a_j)$ in $\wt W^\wedge-W$ besides points of $\mathbf{x}$; and
\item if $v$ is a component of $u\in \mathcal{M}_{J^\Diamond}({\bf y},{\bf y}')$ and one of the positive ends of $v$ limits to $x_i$, then $v$ is a trivial strip $\R\times[0,1]\times \{x_i\}$.
\end{enumerate}
Hence every component of $u\in \mathcal{M}_{J^\Diamond}({\bf y},{\bf y}')$ nontrivially intersects $\R\times[0,1]\times W$, provided the component is not a trivial strip over some $x_i$. This means that there exists an almost complex structure $J^\Diamond\in (\mathcal{J}^\Diamond)^{reg}$ such that $J^\Diamond=J$ on $\R\times[0,1]\times (\wt W^\wedge-W)$; in other words, $J^\Diamond$ just needs to be generic on $\R\times[0,1]\times W$ to attain regularity. The contact class ${\bf x}$ is a cycle for such a $J^\Diamond$, proving the lemma.
\end{proof}

{\em From now on we assume that $J^\Diamond\in (\mathcal{J}^\Diamond)^{reg}$ and $f$ are chosen so that Lemma~\ref{lemma: contact cycle} holds.}

\subsection{Connected sums}\label{subsection: connected sums}

Let $(W_0,\beta_0,\phi_0;h_0)$ and $(W_1,\beta_1,\phi_1;h_1)$ be supporting open book decompositions for the contact manifolds $(M_0,\xi_0)$ and $(M_1,\xi_1)$. We assume the Morse functions $\phi_i$, $i=0,1$, have been normalized so that $\partial W_i$ is the regular level set $\phi_i^{-1} (1)$. 

Let $(W,\beta,\phi)$ be the Weinstein domain obtained by attaching a Weinstein $1$-handle $H$ to $(W_0,\beta_0,\phi_0)\sqcup(W_1,\beta_1,\phi_1)$ along small balls that are disjoint from the boundaries of the bases $\bs a_0$ and $\bs a_1$ of Lagrangian disks for $(W_0,\beta_0,\phi_0)$ and $(W_1,\beta_1,\phi_1)$ and then a collar neighborhood $N$, such that $\phi$ is an extension of $\phi_0\cup \phi_1$ to $W$, $\phi(\bdry W)=3$, and $\phi^{-1} ([2,3])=H\cup N$. Also let $h$ be the extension of $h_0\cup h_1$ by the identity on $H\cup N$. 

The open book $(W,\beta,\phi;h)$ is an adapted open book for the connected sum $(M_0 \# M_1 ,\xi_0 \# \xi_1)$, where $\xi =\xi_0 \# \xi_1$ is obtained by gluing the complement of two  standard Darboux balls respectively in $(M_0,\xi_0)$ and $(M_1,\xi_1)$ along their boundaries.

Let $(\wt W,\wt \beta, \wt \phi; \wt h)$ be the capping of $(W,\beta,\phi;h)$ as usual. The capping $(\wt W,\wt \beta, \wt \phi)$ is trivial Weinstein homotopic to $(\wt W',\wt \beta', \wt \phi')$, which is obtained from 
$$(\wt W_0,\wt \beta_0,\wt \phi_0) \sqcup (\wt W_1,\wt \beta_1,\wt \phi_1)$$ 
by first capping off the Lagrangian bases $\bs a_0$ and $\bs a_1$ so that $(\wt\phi')^{-1} (2)=\bdry \wt W_0\cup \bdry \wt W_1$, then attaching a $1$-handle $H'$ connecting $\wt W_0$ and $\wt W_1$, and finally attaching a collar neighborhood $N'$ such that $(\wt \phi')^{-1} ([2,3])=H'\cup N'$. The symplectomorphism $\wt h$ is homotopic to $\wt h'$, which is the extension of $\wt h_0\cup \wt h_1$ to $\wt W'$ by the identity.

The Heegaard Floer homology groups are well-behaved under this connected sum operation:
	
\begin{lemma}\label{lemma: connected sum}
		With $(W,\beta,\phi;h)$ as above,
		$$\widehat{HF}(W,\beta,\phi;h) \simeq \widehat{HF}(W_0,\beta_0,\phi_0;h_0) \otimes \widehat{HF}(W_1,\beta_1,\phi_1;h_1).$$
		The contact class $c^0 (W,\beta,\phi;h)$ vanishes if and only if one of the contact classes $c^0 (W_0,\beta_0,\phi_0;h_0)$ or $c^0 (W_1,\beta_1,\phi_1;h_1)$ vanishes.
\end{lemma}

\begin{proof}
Let $\bs\alpha$, viewed as a Lagrangian in $\wt W'$, be the union of cappings $\wt{\bs a}_0$ and $\wt{\bs a}_1$ in $(\wt W_0,\wt \beta_0,\wt \phi_0)$ and $(\wt W_1,\wt \beta_1,\wt \phi_1)$ of the Lagrangian bases $\bs a_0$ and $\bs a_1$. Since $\wt{h}'|_{H'\cup N'}=\op{id}$, $\bs\alpha \cup \wt{h}'(\bs\alpha)\subset \{\wt{\phi}' \leq 2\}$ and is disjoint from $H'\cup N'$. No $J^\Diamond$-holomorphic curve $u$ that is counted in $\widehat{CF}(\wt W', \wt h'(\bs\alpha), \bs\alpha; J^\Diamond)$, for a generic $J^\Diamond$ close to a product $J$, has a projection $\pi_{\wt W'} \circ u$ to $\wt{W}$ that enters the handle $H'$ by the maximum principle and Gromov compactness. Hence
$$\widehat{CF}(\wt W', \wt h'(\bs\alpha), \bs\alpha)\simeq \widehat{CF}(W_0,\beta_0,\phi_0;h_0)\oplus \widehat{CF}(W_1,\beta_1,\phi_1;h_1).$$
Since $(\wt W,\wt \beta, \wt \phi)$ and $(\wt W',\wt \beta', \wt \phi')$ are Weinstein homotopic through a trivial Weinstein homotopy, Theorem~\ref{thm: invariance or lack thereof}(I1) implies the first statement of the lemma.  Since the maps in Theorem~\ref{thm: invariance or lack thereof} take the contact class to the contact class, the second statement follows.
\end{proof}

\begin{q}
If $(M_0,\xi_0)$ and $(M_1,\xi_1)$ have non-vanishing contact class for all supporting Weinstein open book decompositions, then does $(M_0 \# M_1 ,\xi_0 \# \xi_1)$ also have a non-vanishing contact class for all supporting Weinstein open book decompositions?  In particular is $(M_0 \# M_1 ,\xi_0 \# \xi_1)$ tight?
\end{q}

\subsection{Cobordisms} \label{subsection: cobordisms}

In this subsection we define several cobordisms and describe the effect of the induced Floer cohomology maps on the contact classes.

Let $(M=M_{(W,\beta;h)},\xi=\xi_{(W,\beta;h)})$ be the contact manifold supported by the open book decomposition $(W,\beta;h)$. As before, $(\wt W,\wt \beta)$ is the capping off of $W$ and $\wt W^\wedge$ is the completion of $\wt W$. Let $h^\star$ be the extension of $h$ to $\wt W$ or $\wt W^\wedge$ by the identity.

\s\n
{\em The cobordism $(X_+,\Omega_+,J_+)$.} The cobordism $(X_+,\Omega_+)$ interpolates from the symplectization of $\wt W^\wedge \times [0,1]$ at the positive end to the symplectization of $(M,\xi)$ with a Lagrangian boundary condition $L^+_{\wt {\bs a}} \subset \partial X_+$ at the negative end.  This is the higher-dimensional analog of the cobordism from \cite{CGH3} that was used to construct an isomorphism from the plus version of the Heegaard Floer homology of a $3$-manifold to its embedded contact homology. We only give a brief sketch since the constructions from \cite{CGH3} extend to higher dimensions with minimal change. 

Let $W_0=W$ and $W_{1/2} =\wt W^\wedge -W_0$.

First we construct fibrations $\pi_0: X^0_+\to B^0_+$ and $\pi_1: X^1_+\to D^2$ with fibers diffeomorphic to $\wt W^\wedge$ and $W_{1/2}$.  Here $B^0_+ =([0,\infty ) \times \R /2\Z)-B_+^c$ with coordinates $(s,t)$ and $B_+^c$ is $[2,\infty)\times[1,2]$ with the corners rounded. We then glue $X^0_+$ and $X^1_+$ and smooth a boundary component $M^\sqcap$ of $X^0_+\cup X^1_+$ to obtain $M$. Finally we attach the negative end $X_+^2=(-\infty,0]\times M$ to get $X_+$.

The fibration $\pi_0: X^0_+\to B^0_+$ is a subset of $[0,\infty)\times N_{(\wt W^\wedge,h^\star)}$, where
$$N_{(\wt W^\wedge, \wt{h})}=(\wt W^\wedge \times[0,2])/ (x,2)\sim (h^\star(x),0).$$
It has a symplectic form $\Omega^0_+ =\pi_0^* ds\wedge dt +d\beta$.  The fibration $\pi_1: X^1_+\to D^2$ is topologically trivial and its symplectic form $\Omega^1_+$ is the split form $d\beta\vert_{W_{1/2}} +\omega_{D^2}$.

We fiberwise glue $(X^0_+ ,\Omega^0_+)$ and $(X^1_+,\Omega^1_+)$ along $W_{1/2} \times \{0 \} \times \R /2\Z$ and $W_{1/2} \times \partial D^2$ with identity fiberwise gluing maps.
We then round corners to obtain the concave contact boundary $(M,\xi )$ with a compatible open book presentation and glue to it the negative symplectization $(X_+^2,\Omega_+^2)=(-\infty,0]\times M$.  This completes the construction of $(X_+,\Omega_+)$.

The Lagrangian $L^+_{\wt {\bs a}} \subset \partial X_+$ is obtained by placing a copy of $\wt{\bs a}$ over $(s,t)=(3,1)$ and parallel transporting it along $\bdry X_+$ using the symplectic connection $\Omega_+$. It agrees with $L_{1,\wt{\bs a}}$ on $\pi_0^{-1}([3,\infty)\times \{1\})$ and with $L_{0,h^\star (\wt{\bs a})}$ on $\pi_0^{-1}([3,\infty)\times \{0\})$.  (Strictly speaking, we want to use $L_{0,\wt h(\wt{\bs a})}$ where $\wt h$ is obtained from $h^\star$ by composing with a small Hamiltonian isotopy as in Section~\ref{subsection: extension}.  To do this, we need to slightly modify the symplectic connection $\Omega_+$ by adding a Hamiltonian term as in Equation~\eqref{eqn: defn of Omega prime}.)

The almost complex structure $J_+$ on $(X_+,\Omega_+)$ is
\be
\item[(i)] compatible with $(\wt W^\wedge, \wt\beta^\wedge)$ at the positive end;
\item[(ii)] of the form $J_+=j_{B_+^0\cup D^2}\times J_{W_{1/2}}$ on  $(B^0_+ \cup D^2)\times  W_{1/2}$; and
\item[(iii)] adapted to a contact form on the symplectization end $X^2_+$.
\ee
(ii) ensures that, as in Lemma~\ref{lemma: contact cycle}, the only holomorphic curves that limit to the contact class ${\bf x}$ at the positive end are constant horizontal sections over the once-punctured disk $B^0_+ \cup D^2$.

\s\n
{\em Lefschetz cobordism}.  A Lefschetz cobordism is a Lefschetz fibration over $\R_s\times[0,1]_t$ whose fiber is $(\wt W^\wedge, \wt\beta^\wedge)$ and which has $p$ singular fibers, located at $(s,t)=(1,1/2), \dots, (p,1/2)$ and corresponding to the vanishing cycles $S_1,\dots, S_p$ of $W$. We will not be explicit about the construction of the almost complex structure and the symplectic connection except to say that at the positive and negative ends they need to coincide with $J^\Diamond\in (\mathcal{J}^{\Diamond})^{reg}$ and $\widehat\Omega$.  The only additional ingredient is the construction near the singular fibers, which is carried out in Seidel~\cite{Se1}. Such a cobordism induces a map
$$\widehat{CF} (h({\bs a}),{\bs a}) \to \widehat{CF} (\tau_{S_p}^{-1} \circ \dots \circ \tau_{S_1}^{-1} (h({\bs a})),{\bs a}),$$
by a count of pseudo-holomorphic multisections. Here $\tau_{S}$ denotes the symplectic Dehn twist along the Lagrangian sphere $S$.

\s
The following two lemmas are proved in the same way as Lemma~\ref{lemma: contact cycle} and their proofs will be omitted.

\begin{lemma}\label{lemma: constant}
The only pseudoholomorphic curve in $(X_+ , \Omega_+ ,J_+ )$ that limits to the contact class ${\bf x}$ at the positive end is the union of $\kappa$ constant sections $x_i \times \overline{B}_+$.
\end{lemma}

\begin{lemma}\label{lemma: Lefschetz}
The induced map
$$\widehat{CF} (h({\bs a}), {\bs a}) \to \widehat{CF} (\tau_{S_p}^{-1} \circ \dots \circ \tau_{S_1}^{-1} (h({\bs a})), {\bs a})$$
of a Lefschetz cobordism with vanishing cycles $S_1,\dots,S_p$ maps the contact class to the contact class.
\end{lemma}

We now come to the proof of Theorem~\ref{thm: non-vanishing}:

\begin{proof}[Proof of Theorem~\ref{thm: non-vanishing}]
(1)	If $(M,\xi )$ is Liouville fillable, then we can attach the filling to the negative end of $(X_+ ,\Omega_+ )$ to obtain an exact symplectic cobordism $(\overline{X}_+ ,\overline{\Omega}_+)$ from $\wt W^\wedge \times [0,1]$ to $\varnothing$. If (*) $\sum_i c_i d {\bf y}^i={\bf x}$ with $c_i\in \Z$, then we can glue the holomorphic curves in $\R \times [0,1]\times \wt W^\wedge$ that are involved in (*) to the constant sections ${\bf x}\times \overline{B}_+$ in $(\overline{X}_+,\overline{\Omega}_+)$ to get an index one family of curves. Algebraically over $\Z$, there is no possible breaking for this family other than the one we started from, contradicting the compactness of moduli spaces in $(\overline{X}_+,\overline{\Omega}_+)$.

(2) We show that every Reeb vector field $R$ for $\xi$ admits a finite collection of periodic orbits whose sum is homologous to zero in $M$. Consider the cobordism $(X_+,\Omega_+)$ and argue as in (1) with the understanding that $J_+$ is adapted to the symplectization at the negative end. Again, we can glue the holomorphic curves in $\R \times [0,1]\times \wt W^\wedge$ that are involved in (*) to the constant sections ${\bf x}\times \overline{B}_+$ in $(X_+,\Omega_+)$ to get an index one family of curves. Algebraically over $\Z$,  there is no possible breaking for this family other than a breaking at $-\infty$, involving periodic orbits of the Reeb vector field $R$ as well as pseudoholomorphic curves in the symplectization of $(M,\xi)$ without negative ends. Those provide a bounding chain for the orbits of $R$ in $M$.
\end{proof}

\section{Examples of non-Liouville-fillable contact structures} \label{section: examples}

\subsection{Right-veering symplectomorphisms} \label{subsection: defn of right-veering}

In this subsection we give two higher-dimensional generalizations of a {\em right-veering} surface diffeomorphism and discuss the non-Liouville-fillability of contact manifolds which admit open book decompositions with non-right-veering monodromy.

Let $S$ be a {\em bordered surface}, i.e., a compact connected oriented surface with nonempty boundary, and let $\op{Diff}^+(S,\bdry S)$ be the set of orien\-tation-preserving diffeomorphisms of $S$ that restrict to the identity on $\bdry S$. According to  \cite{HKM}, an element $h\in\op{Diff}^+(S,\bdry S)$ is {\em right-veering} if $h$ takes every arc $a\subset S$ with boundary on $\bdry S$ to itself or ``to the right", after isotopy.

Let $(W,\beta)$ be a Liouville domain.  

\begin{defn}[Locally right-veering]
A symplectomorphism $h\in  \op{Symp}(W,\bdry W, d\beta)$ is {\em locally right-veering} if there is a collar neighborhood $[-\epsilon,0]_s\times \bdry W$ of $\bdry W=\{0\}\times \bdry W$ such that $h= \phi_{-s}$ on $\{s\}\times \bdry W$, $s\in[-\epsilon,0]$, where $\phi_s$ is the time-$s$ flow of the Reeb vector field of $\beta|_{\bdry W}$. 
\end{defn}

{\em In what follows we assume that $h\in  \op{Symp}(W,\bdry W, d\beta)$ is locally right-veering.}

Given an exact Lagrangian submanifold $a\subset W$ with Legendrian boundary $\bdry a\subset \bdry W$, we write $HF(W,h(a),a)$ for the Lagrangian Floer cohomology of the pair $(h(a),a)$ subject to a clean intersection condition along $\bdry a= \bdry h(a)$.  Let $x_a$ be the ``contact class", i.e., the top generator corresponding to the Morse-Bott family $\bdry a$.  

The following is a straightforward higher-dimensional generalization of a right-veering surface diffeomorphism:

\begin{defn}[Right-veering]\label{defn: right-veering}
Let $h\in  \op{Symp}(W,\bdry W, d\beta)$ be locally right-veering.
\be
\item $h$ is {\em strongly right-veering} if there exists no exact Lagrangian submanifold $a\subset W$ with Legendrian boundary $\bdry a\subset \bdry W$ such that $[x_a]=0\in HF(W,h(a),a)$ (i.e., $h$ {\em sends $a$ to the left}).
\item $h$ is {\em weakly right-veering}  if there exists no {\em regular} Lagrangian disk $a\subset W$ with Legendrian boundary $\bdry a\subset \bdry W$ such that $[x_a]=0\in HF(W,h(a),a)$.
\ee
If we want to specify the coefficient ring $R$ for the Floer homology groups, we say strongly or weakly {\em $R$-right-veering.}
\end{defn}

By a {\em regular Lagrangian disk} (cf.\ \cite{EGL}) we mean a Lagrangian disk which can be completed to a Lagrangian basis with respect to some $(W,\beta',\phi')$ that is Liouville homotopic to $(W,\beta)$.  Note that strongly right-veering implies weakly right-veering.

\begin{q}
Is there a difference between strongly right-veering and weakly right-veering?
\end{q}

\begin{rmk}
Allowing individual Lagrangians in the definition of $\widehat{HF}$, it is immediate that $\widehat{HF}(h(a),a)= HF(h(a),a)$, where the left-hand side is the Heegaard Floer group using the cylindrical reformulation and the right-hand side is the usual Floer cohomology group.
\end{rmk}

We first prove the following warm-up theorem:

\begin{thm}\label{thm: vanishing}
If  $(M,\xi)$ is overtwisted, then $\xi$ admits a supporting open book decomposition $(W,\beta,\phi ;h)$ for which $h$ is not weakly right-veering and $c^0 (W,\beta,\phi ;h) =0$ in $\widehat{HF} (W,\beta,\phi ;h)$.
\end{thm}

In this paper ``overtwisted" means overtwisted in the sense of Borman-Eliashberg-Murphy~\cite{BEM}.
Theorem~\ref{thm: vanishing} and Theorem~\ref{thm: non-vanishing} imply the well-known fact that overtwisted contact structures are not Liouville fillable and satisfy the Weinstein conjecture.

\begin{proof}
According to Casals-Murphy-Presas \cite{CMP}, an overtwisted contact structure is supported by an open book decomposition $(W',\beta',\phi' ;h')$ which is a negative stabilization of some Weinstein open book $(W,\beta,\phi;h)$.  Using the notation from Section~\ref{subsection: open books}, we have a basis ${\bs a}= a_1\cup\dots\cup a_\kappa$, a Lagrangian disk $c\subset W$ with Legendrian boundary which is disjoint from $\bdry {\bs a}$, a handle $W_0$ attached along $\bdry c$, and $\gamma= c\cup c'$, where $c'$ is the core of $W_0$. Then
$$h'=(h\cup \op{id}_{W_0} )\circ \tau_\gamma^{-1}.$$
We let $a_0$ be the cocore of $W_0$. Then $(W',\beta',\phi')$ has basis ${\bs a}'=a_0\cup\dots \cup a_\kappa$.

We now define a special Hamiltonian deformation of $\alpha_{0}=\wt a_0$ in a given Weinstein neighborhood $N(\alpha_{0})\simeq T^* \alpha_{0}$ equipped with the projection $\pi :T^* \alpha_{0} \to \alpha_{0}$. To that end, we consider a Morse function $f:\alpha_{0} \to \R$ with four critical points $x_{0}$, $x_{0}'$, $y_{0}$ and $z_{0}$, where $x_{0}$, $x_{0}'$ have index $0$, $y_{0}$ is an index $1$ saddle point between them, and $z_{0}$ has index $n$. We may assume that:
\be
\item $x_0$, $y_0$, and $z_0$ are contained in $\alpha_0-a_0$;
\item the trajectory of $\nabla f$ from $x_{0}$ to $y_{0}$ is entirely contained in $\alpha_{0} -a_{0}$;
\item $x_{0}'=\alpha_{0} \cap \gamma$ and the intersection $T^*\alpha_{0}\cap \gamma$ is the fiber of $T^* \alpha_{0}$ over $x_{0}'$.
\ee

We define $F : (\wt W')^\wedge\to \R$ to be $\pi^* f$ on $T^*\alpha_{0}$ and $0$ on a slightly bigger neighborhood and let $\phi_\epsilon$ be the time-$\epsilon$ flow of the Hamiltonian vector field $X_F$. The monodromy $\wt h'$ on $\wt W'$ sends $\alpha_{0}$ to $\tau_\gamma (\phi_{-\epsilon} (\alpha_{0} ))$. In particular the intersection $\wt h' (\alpha_{0} ) \cap \alpha_{0}$ consists of three points: $x_{0}$, $y_{0}$, and $z_{0}$, where the pair $x_{0}$ and $y_{0}$ can be eliminated by a Hamiltonian isotopy.

Let ${\bf y}=\{y_0,x_1,\dots,x_\kappa\}$ and ${\bf x}=\{x_0,\dots, x_\kappa\}$. We claim that $\pm d {\bf y}={\bf x}$ when $\epsilon>0$ is sufficiently small and $J^\Diamond$ is suitably chosen. Indeed, by Lemma \ref{lemma: constant}, for a well-chosen $J^\Diamond$ there is no nonconstant connected holomorphic curve that has a contact coordinate at the positive end. Hence any curve from ${\bf y}$ is a collection of $\kappa$ trivial strips from $x_1,\dots,x_\kappa$ to itself, together with an index $1$ curve from $y_{0}$ to a point in $\wt h' (\alpha_{0} ) \cap \alpha_{0}$, that must therefore be $x_{0}$.  We have now reduced the calculation to a standard Floer homology calculation of $d y_0$ in $CF(\wt h'(\alpha_0), \alpha_0)$.  Since $x_0$ and $y_0$ can be canceled by a Hamiltonian isotopy, $HF(\wt h'(\alpha_0),\alpha_0)$ is generated by $z_0$ and $\pm d y_0=x_0$. In other words, $h'$ sends $a_0$ to the left. This implies the claim and hence the theorem.
\end{proof}

\begin{proof}[Proof of Theorem~\ref{thm: not right-veering}]
	Follows immediately from the proofs of Theorems~\ref{thm: non-vanishing} and \ref{thm: vanishing}.  Note that in the proof of Theorem~\ref{thm: non-vanishing} it does not matter whether we take a full Lagrangian basis or just a single exact Lagrangian.
\end{proof}

The following theorem gives examples of symplectomorphisms that are not weakly right-veering. 

\begin{thm}\label{thm: negative twists}
If $h\in \op{Symp}(W,\bdry W,d\beta)$ is a product of negative symplectic Dehn twists, then $h$ is not weakly $\R$-right-veering.
\end{thm}

We remark that it is currently not known whether the corresponding contact manifold $(M,\xi)$ is overtwisted.

\begin{proof}
In this proof the Floer homology groups are over $\R$.

Since the proof of Lemma~\ref{lemma: Lefschetz} also applies to a single exact Lagrangian $a$, it suffices to show that if $h$ is a single negative symplectic Dehn twist $\tau_S^{-1}$ along a Lagrangian sphere $S$, then there exist a regular exact Lagrangian $a$ such that $[x_a]=0\in HF(h(a),a)$.  

We use the fact that, since the basis ${\bs a}=\{a_1,\dots ,a_\kappa\}$ generates the wrapped Fukaya category of $W$ by \cite{CDGG}, there exists $i\in \{1,\dots,\kappa\}$ for which the Lagrangian Floer homology $HF(S,a_i) \neq 0$. (Note that since $S$ is compact, there is no difference between the wrapped and unwrapped groups involving $S$.)  Up to reordering, we assume $a=a_1$ has this property.

Next, switching to the Heegaard decomposition perspective, we show that the element $[x_a]=[x_1] \in HF(\tau_S^{-1} (\wt a),\wt a)$ is zero: This is a consequence of Seidel's exact triangle \cite{Se1}:
$$\xymatrix{
HF(\wt a',\wt a) \ar[rr]^{L_*} & & HF(\tau_S^{-1} (\wt a'), \wt a) \ar[dl] \\
& HF(S,\wt a)\otimes  HF(\wt a', S) \ar[ul]^{\mu_2} &
}$$
where $\wt a'$ is the small Hamiltonian deformation of $\wt a$ as given in Section~\ref{subsection: extension}. Note that $HF(S,a)$ and $HF(S,\wt a)$ are canonically isomorphic. 
In this triangle, the map $L_*$ is the one given by the Lefschetz cobordism over the strip with one singular fiber and  whose vanishing cycle is $S$. Hence it maps the contact class in $HF(\wt a',\wt a)$ to the contact class in $HF(\tau_S^{-1} (\wt a'),\wt a)$ by the analog of Lemma~\ref{lemma: Lefschetz}. The latter is zero if the contact class in $HF(\wt a',\wt a)$ is in the image of $\mu_2$, as defined in Section~\ref{section: A infty operations}. 

From now on we make crucial use of $\R$-coefficients. For $\wt a'$ sufficiently close to $\wt a$, the intersection points $\wt a'\cap S$ and $\wt a\cap S$ are canonically identified and we refer to either by $\{y_1,\dots, y_s\}$. This gives us an identification $CF(\wt a',S)\simeq CF(S,\wt a)$ as vector spaces. 
We consider the standard inner products on $CF(\wt a',S)$ and $CF(S,\wt a)$ for which $\{y_1,\dots, y_s\}$ is an orthonormal basis and identify both with an inner product space $(V,\langle,\rangle)$.  The differentials on $CF(\wt a',S)$ and $CF(S,\wt a)$ can be written as adjoints $d$ and $d^*$ on $V$ satisfying 
\begin{equation} \label{eqn: adjoint}
\langle d^* y_i,y_j\rangle = \langle y_i,dy_j\rangle.
\end{equation}
Taking an orthogonal decomposition $\ker d=\op{Im} d\oplus (\op{Im} d)^\perp$, there exists a nonzero $y\in (\op{Im} d)^\perp\subset  V$ such that $dy=0$  since $HF(S,\wt a)\simeq HF(\wt a',S)\not=0$.  By Equation~\eqref{eqn: adjoint}, $y\in  (\op{Im} d)^\perp$ implies that $d^*y=0$.  


We now claim that $\mu_2 (y,y)$ is a positive multiple of $x$. Writing $y=\sum_{i=1}^s a_i y_i$, where $a_i\in \R$, we show that (i) $\mu_2(y_i,y_{i'})=0$ for $i\not=i'$, (ii)  $\mu_2(y_i,y_i)$ counts a single thin triangle $T_i$ which corresponds to a Morse gradient trajectory on $\wt a$ from $y_i$ to $x_a$, viewed as the top generator of $CF(\wt a',\wt a)$, and (iii) there exists an orientation system such that $\mu_2(y_i,y_i)=\pm 1$, where all the signs are same.  Here we are considering upward gradient trajectories with respect to a Morse function $f$ on $\wt a$ whose maximum is the generator $x_a$ and whose minimum $\not= y_1,\dots, y_s$ is in $\op{int} a$. 

(i) and (ii) are consequences of taking the limit $\wt a'\to \wt a$.  Given a sequence of curves $u_j$ that are counted in $\mu_2(y_i,y_{i'})$ (where $i,i'$ may be equal) with $\wt a'\to \wt a$, the limit $u_\infty$ consists of a (possibly broken) holomorphic strip $v$ with boundary on $\wt a$ and $S$ between $y_i$ and $y_{i'}$, together with a gradient trajectory $\gamma$ to the critical point $x_a$ of $f$. If $v$ is nontrivial, then $\op{ind}(v)\geq 1$ by regularity. 
This implies that $\op{ind}(u_j)\geq 1$, a contradiction.  If $v$ is trivial, then $u_\infty$ is a single gradient trajectory and $u_j$ are curves of the form (ii). 

(iii) Referring to Section~\ref{subsection: orientations} (and using the open book perspective), we can take neighborhoods $N_i\subset \wt W$ of the thin triangles $T_i$ and diffeomorphisms $\psi_{ij}:N_i\stackrel\sim \to N_j$ taking $T_i$ to $T_j$, $N_i\cap a$ to $N_j\cap a$, $N_i\cap a'$ to $N_j\cap a'$, and $N_i\cap S$ to $N_j\cap S$; and such that all the orientation data for (i.e., used in defining the sign/orientation of) $T_i$ is taken to the orientation data for $T_j$. This guarantees that all the curves counted in (ii) have the same sign. 




Therefore, working over $\R$-coefficients, $\mu_2(y,y)= \pm(\sum_i a_i^2) x_a$ and there exists an element $z \in CF(\tau_S^{-1} (\wt a), \wt a)$ such that $d z=x_a$. By Lemma \ref{lemma: constant},  
$$d\{z, x_2,\dots ,x_\kappa \}=c^0  (W,\beta,\phi;\tau^{-1}_S)\in \widehat{CF}(W,\beta,\phi;\tau^{-1}_S),$$
which implies the theorem.
\end{proof}

\section{Variant of symplectic Khovanov homology} \label{section: variant of symplectic Khovanov homology}

\subsection{Definitions} \label{subsection: symplectic Kh definitions}

Let $\wt W$ be the $4$-dimensional Milnor fiber of the $A_{2\kappa-1}$ singularity and let 
$$\wt p: \wt W\to \wt D\subset \C_{z}$$  
be the standard Lefschetz fibration over the square $\wt D=\{-2 \leq \op{Re}z, \op{Im} z\leq 2\}$ with regular fiber $A=S^1\times[-1,1]$ and $2\kappa$ critical values $\wt {\bf z}=\{z_1,\dots, z_{2\kappa}\}$, where $\op{Re} z_i = \op{Re} z_{i+\kappa}$, $\op{Im} z_i=-1$, and $\op{Im} z_{i+\kappa}=1$ for $i=1,\dots,\kappa$. Let 
$$p: W:=\wt p^{-1}(D)\to D$$ 
be its restriction to $D= \wt D\cap \{\op{Im} z\leq 0\}$.   Let $\wt {\bs a}=\{\wt a_1,\dots,\wt a_\kappa\}$ be the ``basis'' of $\kappa$ Lagrangian spheres over the $\kappa$ disjoint arcs $\{\wt \gamma_1,\dots,\wt \gamma_\kappa\}$, where $\wt \gamma_i$ is the straight line segment connecting $z_i$ to $z_{\kappa+i}$. The restrictions of $\wt {\bf z}$, $\wt a_i$ and $\wt \gamma_i$ to $p$ are denoted ${\bf z}$, $a_i$ and $\gamma_i$.

Let $h_\sigma\in \op{Symp}(W,\bdry W)$ be the monodromy on $W$ which descends to a braid $\sigma$, viewed as an element of $\op{Diff}^+(D,\bdry D,{\bf z})$ (an orientation-preserving diffeomorphism of $D$ which is the identity on $\bdry D$ and takes ${\bf z}$ to itself setwise), let $\widehat\sigma$ be the braid closure of $\sigma$, and let $\wt h_\sigma$ be the extension of $h_\sigma$ to $\wt W$ by the identity.

Let $h_{\wt {\bs a}, \wt h_\sigma(\wt {\bs a})}$ be given by \eqref{eqn: h}, let $\{T_{\bs \delta, {\bf y}}\}$ be a complete set of capping surfaces from Definition~\ref{defn: complete set}, and let 
\begin{equation}
\mathcal{A}=\mathcal{A}_{\wt {\bs a}, \wt h_\sigma(\wt {\bs a})}=H_2([0,1]\times \wt W, (\{1\}\times \wt{\bs a})\cup (\{0\}\times \wt h_\sigma (\wt {\bs a}));\Z)
\end{equation}
as in \eqref{eqn: mathcal A}.

Our variant $C\Kh(\widehat \sigma)$ of the symplectic Khovanov (co)chain complex is defined as the higher-dimensional Heegaard Floer (co)chain complex $\widehat{CF}(\wt W,\wt h_\sigma(\wt {\bs a}),\wt {\bs a})$ with coefficients $\F[\mathcal{A}]\llbracket \hbar, \hbar^{-1}]$ (this means power series in $\hbar$ and polynomial in $\hbar^{-1}$).  We write $\Kh(\widehat \sigma)$ for the homology $\widehat{HF}(\wt W,\wt h_\sigma(\wt{\bs a}),\wt{\bs a})$. Here $\F[\mathcal{A}]$ is the group ring over $\mathcal{A}$ which we determine to be isomorphic to $\Z^{r-1}$ in Lemma~\ref{lemma: meaning of A}, where $r$ is the number of connected components of $\widehat\sigma$.

After discussing coefficients in Section~\ref{subsection: coefficients}, we prove Theorem~\ref{thm: Kh is invariant of link} in Sections~\ref{subsection: arc slides} and \ref{subsection: stabilization invariance}.  This entails proving the invariance of $\widehat{HF}(\wt W,\wt h_\sigma(\wt{\bs a}),\wt{\bs a})$ under handleslides and Markov stabilizations.

\subsection{Coefficients} \label{subsection: coefficients}

We explain why we can use coefficients $\F[\mathcal{A}]\llbracket \hbar,\hbar^{-1}]$ instead of the Novikov ring and also describe $\F[\mathcal{A}]$.

Given $\kappa$-tuples ${\bf y}$ and ${\bf y}'$ of intersection points of $\wt {\bs a}$ and $\wt h_\sigma (\wt{\bs a})$, consider the moduli space $\mathcal{M}_J({\bf y},{\bf y}')$, where $J$ is the product $J_{\R\times[0,1]}\times J_{\wt W^\wedge}$ such that the projection $\wt p: \wt W^\wedge\to \C$ is $(J_{\wt W^\wedge},i)$-holomorphic. Let $J^\Diamond\in \mathcal{J^\Diamond}^{reg}$ be a perturbation of $J$.

\begin{claim}  \label{claim: images under projection same}
For all $u: \dot F\to \R\times[0,1]\times \wt W^\wedge$ in $\mathcal{M}_J({\bf y},{\bf y}')$, the images under the projection $\wt p \circ \pi_{\wt W^\wedge}\circ u$ to $\C$ are the same when viewed as weighted domains in the usual Heegaard Floer sense.
\end{claim}

\begin{proof}
Write 
$$\{y_{1\lambda(1)},\dots, y_{\kappa \lambda(\kappa)}\} \quad \mbox{and} \quad \{y'_{1 \mu(1)},\dots, y'_{\kappa \mu(\kappa)}\}$$ 
for the projections of ${\bf y}$ and ${\bf y}'$ to $\C$, where $\lambda$, $\mu$ are permutations of $(1,\dots,\kappa)$ and $y_{k\lambda(k)}\in \wt\gamma_k \cap \sigma(\wt\gamma_{\lambda(k)})$ and $y'_{k\mu(k)}\in \wt\gamma_k\cap \sigma(\wt\gamma_{\mu(k)})$.   Then $\wt p \circ \pi_{\wt W^\wedge}\circ u|_{\bdry F}$ is the union of:
\be
\item the unique path in $\wt \gamma_k$ from $y_{k\lambda(k)}$ to $y'_{k\mu(k)}$ for $k=1,\dots,\kappa$, and 
\item the unique path in $\sigma (\wt\gamma_{\mu(k)})$ from $y'_{k\mu(k)}$ to $y_{\lambda^{-1}(\mu(k))\mu(k)}$ for $k=1,\dots,\kappa$,
\ee 
and $\wt p \circ \pi_{\wt W^\wedge}\circ u|_{\bdry F}$ uniquely determines the weighted domain that it bounds.
\end{proof}

\begin{lemma} \label{lemma: finite for Khovanov}
	For fixed ${\bf y}$, ${\bf y}'$, and $\chi$, $\# \mathcal{M}_{J^\Diamond}^{\chi,\op{ind}=1}({\bf y},{\bf y}')/\R$ is finite.
\end{lemma}

\begin{proof}
Claim~\ref{claim: images under projection same} implies area bounds for $u\in \mathcal{M}_J({\bf y},{\bf y}')$ but not necessarily Euler characteristic bounds.  Hence, after perturbing $J$ to $J^\Diamond$, the desired curve count $\# \mathcal{M}_{J^\Diamond}^{\chi,\op{ind}=1}({\bf y},{\bf y}')/\R$ is finite. 
\end{proof}

Lemma~\ref{lemma: finite for Khovanov} justifies the use of the coefficient system $\F[\mathcal{A}]\llbracket \hbar,\hbar^{-1}]$.

\begin{lemma} \label{lemma: meaning of A}
$\mathcal{A}_{\wt {\bs a}, \wt h_\sigma(\wt {\bs a})}\simeq \Z^{r-1}$, where $r$ is the number of connected components of $\widehat\sigma$.
\end{lemma}

We also remark that $\mathcal{A}_{\wt {\bs a}, \wt h_\sigma(\wt {\bs a})}= \mathcal{A}'_{\wt {\bs a}, \wt h_\sigma(\wt {\bs a})}$ since there is no torsion.

\begin{proof}
Writing $X=[0,1]\times \wt W$, $Y=(\{1\}\times \wt{\bs a})\cup (\{0\}\times \wt h_\sigma (\wt {\bs a}))$, and $i:Y\to X$ for the inclusion, the relative homology sequence gives:
$$\mathcal{A}_{\wt {\bs a}, \wt h_\sigma(\wt {\bs a})}=H_2(X,Y)\simeq H_2(X)/i_*H_2(Y).$$
Now $H_2(X)\simeq \Z^{2\kappa-1}$ and is generated by a chain of $2\kappa-1$ spheres lying over a chain of arcs in $D'$ and one can compute the quotient $H_2(X)/ i_*H_2(Y)$ by iteratively collapsing each pair of points connected by an arc $\wt \gamma_i$ or $\sigma(\wt \gamma_i)$ to a point. 
\end{proof}

\subsection{Proof of invariance under arc slides} \label{subsection: arc slides}

Given $\{\wt\gamma_1,\dots,\wt\gamma_\kappa\}$, let $c\subset \wt D$ be a chord from $\wt\gamma_1$ to $\wt\gamma_2$ which does not intersect any other $\wt\gamma_i$.  Let $\wt\gamma_1'$ be obtained by arc sliding $\wt\gamma_1$ over $\wt\gamma_2$ along $c$ and let $\wt\gamma_i'$, $i=2,\dots,\kappa$, be a pushoff of $\wt\gamma_i$ which fixes $z_i, z_{i+\kappa}$; see Figure~\ref{fig: An-handleslide}.   Let $\wt{\bs a}'=\{\wt a'_1,\dots, \wt a'_\kappa\}$ be the Lagrangian basis corresponding to $\{\wt\gamma_1',\dots,\wt\gamma_\kappa'\}$.  

Let $\Theta_i$ (resp.\ $\Xi_i$), $i=1,\dots,\kappa$, be the unique intersection point of $\wt a_i$ and $\wt a_i'$ that lies over $z_i$ (resp.\ $z_{i+\kappa}$), and let $\bs\Theta=\{\Theta_1,\dots, \Theta_\kappa\}$ and $\bs \Xi= \{\Xi_1,\dots, \Xi_\kappa\}$. 

\begin{figure}[ht]
	\begin{overpic}[scale=1]{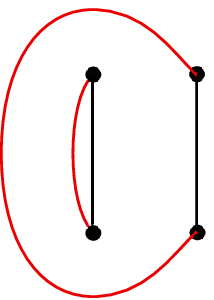}
		\put(35.2,25){\tiny $\Theta_2$}
		\put(71.8,25){\tiny $\Theta_1$}
		\put(35.3,77){\tiny $\Xi_2$}
		\put(71.8,77){\tiny $\Xi_1$}
		\put(14.55,51){\tiny $\wt \gamma_2'$}
		\put(-10,51){\tiny $\wt\gamma_1'$}
		\put(33.35,51){\tiny $\wt\gamma_2$}
		\put(69.7,51){\tiny $\wt\gamma_1$}
	\end{overpic}
	\caption{Arc sliding $\wt\gamma_1$ over $\wt\gamma_2$.}
	\label{fig: An-handleslide}
\end{figure}

\begin{thm} \label{thm: quasi-isomorphism}
	$\widehat{CF}(\wt W,\wt h_\sigma(\wt {\bs a}),\wt {\bs a})$ and $\widehat{CF}(\wt W,\wt h_\sigma(\wt {\bs a}'),\wt {\bs a}')$ are quasi-isomorphic.
\end{thm} 

This is expected by \cite[Lemma 49]{SS}, but there is some benefit in carrying out some holomorphic curve calculations in preparation for the next section.  Our proof closely follows the handleslide invariance proof of the original version of Heegaard Floer homology~\cite{OSz1}.

\begin{proof}
Omitting $\wt W$ from the notation, it suffices to prove that $\widehat{CF}(\wt h_\sigma(\wt {\bs a}),\wt {\bs a})$ and $\widehat{CF}(\wt h_\sigma(\wt {\bs a}),\wt{\bs a}')$ are quasi-iso\-morphic since the quasi-isomorphism of $\widehat{CF}(\wt h_\sigma(\wt{\bs a}),\wt{\bs a}')$ and $\widehat{CF}(\wt h_\sigma(\wt{\bs a}'),\wt{\bs a}')$ is proved analogously. 

We claim that $\bs\Theta$ is a cycle in $\widehat{CF}(\wt{\bs a}',\wt {\bs a})$ and $\bs\Xi$ is a cycle in $\widehat{CF}(\wt{\bs a},\wt{\bs a}')$. For simplicity let $\kappa=2$.  The Fredholm index differences are:
\begin{equation}\label{eqn: Fredholm index difference}
\op{ind}(\bs\Theta, \bs\Xi)=4, \quad \op{ind}(\bs\Theta, \{\Theta_1,\Xi_2\})=\op{ind}(\bs\Theta, \{\Theta_2,\Xi_1\})=2.
\end{equation}
We will treat the first case; the other cases are easier. A holomorphic map $u: \dot F\to \R\times [0,1]\times W$ that limits to $\bs\Theta$ and $\bs\Xi$ at the positive and negative ends satisfies $\op{ind}(\bs\Theta, \bs\Xi)=\mu(u)$ by Lemma~\ref{lemma: Fredholm index}. Let $A$ and $B$ be the outer and inner components of $\wt D$ that are bounded by $\wt \gamma_1,\wt\gamma_2, \wt \gamma_1',\wt \gamma_2'$.  Then the projection of $u$ to $\wt D$ has degree $1$ over $A$ and degree $2$ over $B$ and regardless of the number of branch points of this projection we have $\mu(u)= 4$.  Equation~\eqref{eqn: Fredholm index difference} implies that $\bs\Theta$ is cycle; the situation for $\bs\Xi$ is similar.

In view of the claim we can define the cochain map
\begin{gather} \label{eqn: chain map Phi}
\Phi: \widehat{CF}(\wt h_\sigma(\wt{\bs a}),\wt{\bs a})\to \widehat{CF}(\wt h_\sigma(\wt {\bs a}),\wt{\bs a}'),\\
\nonumber {\bf y}\mapsto \mu_2(\bs\Xi \otimes {\bf y}),
\end{gather}
where $\mu_2$ is the product map
$$\mu_2: \widehat{CF}(\wt{\bs a},\wt {\bs a}')\otimes \widehat{CF}(\wt h_\sigma(\wt {\bs a}), \wt {\bs a})\to \widehat{CF}(\wt h_\sigma(\wt{\bs a}),\wt {\bs a}').$$
Similarly we have the cochain map
$$\Psi: \widehat{CF}(\wt h_\sigma(\wt {\bs a}),\wt {\bs a}') \to \widehat{CF}(\wt h_\sigma(\wt{\bs a}),\wt{\bs a}),$$
$${\bf y}\mapsto \mu_2'({\bs \Theta}\otimes {\bf y}),$$
where $\mu_2'$ is the product map
$$\mu_2': \widehat{CF}(\wt{\bs a}', \wt{\bs a})\otimes \widehat{CF}(\wt h_\sigma(\wt {\bs a}), \wt{\bs a}')\to \widehat{CF}(\wt h_\sigma(\wt {\bs a}), \wt {\bs a}).$$

Composing the two maps corresponds to degenerating the base $D_4$ in one way. We now degenerate the base in the ``other way" as in Figure~\ref{fig: degenerating-D4}.

\begin{figure}[ht]
	\begin{overpic}[scale=0.8]{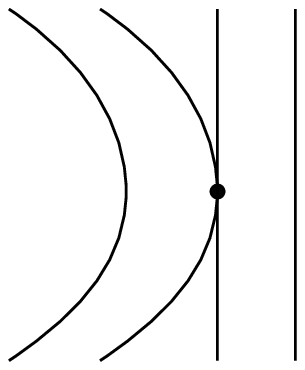}
		\put(83,78){\tiny $\wt h_\sigma(\wt{\bs a})$}
		\put(52,82){\tiny $\wt{\bs a}$}
		\put(52,15){\tiny $\wt{\bs a}$}
		\put(25,48){\tiny $\wt{\bs a}'$}
		\put(12,0){\tiny $\bs\Theta$}
		\put(13,98){\tiny $\bs \Xi$}
	\end{overpic}
	\caption{}
	\label{fig: degenerating-D4}
\end{figure}

Again assuming that $\kappa=2$ we consider the moduli space 
$$\mathcal{M}_{J}(\bs\Xi, \bs\Theta),$$ 
where the symplectic fibration and Lagrangian boundary conditions are:
$$\pi: \R\times[0,1]\times \wt W^\wedge \to \R\times[0,1], \quad L_1= \R\times \{1\}\times \wt{\bs a}', \quad L_0=\R\times\{0\}\times \wt{\bs a},$$
and we take $J=J_{\R\times[0,1]}\times J_{\wt W^\wedge}$. (To avoid cumbersome notation, we will not put a hat over $\wt p$.) We will use the following modifiers:
\begin{itemize}
	\item $\chi=j$, which means the domain $\dot F$ of the holomorphic map $u$ has Euler characteristic $j$;
	\item ${\bf w}=\{w_1,w_2\}$, which means that $u$ passes through the points $(0,0,w_1)$ and $(0,0,w_2)$, where $(0,0)\in \R\times[0,1]$ and $w_i\in \wt a_i$.
\end{itemize}

Since the projection of $u$ to $\C$ is a holomorphic map which is a degree $1$ map to region $A$ and a degree $2$ map to region $B$, where $A$ and $B$ are as before and $\op{ind}(\bs\Xi,\bs\Theta)= \mu(u)=4$, the only type of degeneration as in Figure~\ref{fig: degenerating-D4} is if the left-hand side had $\op{ind}=4$ and the right-hand side had $\op{ind}=0$.

By Theorem~\ref{thm: count} below, the count of $\chi=0$ curves from $\bs\Xi$ to $\bs \Theta$ (i.e., annuli) passing through a generic ${\bf w}=\{w_1,w_2\}$ (by this we mean passing through a pair of points $(0,0,w_1)$ and $(0,0,w_2)$ with $w_1,w_2$ generic) is $\pm 1$.  Although it would be interesting to do, we do not calculate the contributions from $\chi<0$. Thus $\Psi\circ \Phi$ is chain homotopic to $\hbar^2\cdot p(\hbar)$ times the identity, where $p(\hbar)$ is a power series in $\hbar$ whose constant term is $\pm 1$. (Note that the fact that $\Psi\circ \Phi$ is chain homotopic to $\hbar^2 \cdot p(\hbar)\op{id}$ forces us to work over $\F\llbracket \hbar, \hbar^{-1}]$.) The situation for $\Phi\circ \Psi$ is analogous. This completes the proof of Theorem~\ref{thm: quasi-isomorphism}, modulo a discussion of coefficients and the rather nontrivial calculations involved in the proof of Theorem~\ref{thm: count}.
\end{proof}

\subsubsection{Coefficients}

We relate $h_{\wt {\bs a}, \wt h_\sigma(\wt {\bs a})}$ and $\mathcal{A}_{\wt {\bs a}, \wt h_\sigma(\wt {\bs a})}$ for the pair $(\wt h_\sigma(\wt {\bs a}),\wt {\bs a})$ with those for $(\wt h_\sigma(\wt {\bs a}),\wt {\bs a}')$.

First consider the map 
\begin{gather*}
\bs\Xi_*: h_{\wt {\bs a}, \wt h_\sigma(\wt {\bs a})} \to h_{\wt {\bs a}', \wt h_\sigma(\wt {\bs a})},\quad [\bs\delta]\mapsto[\bs\Xi\bs\delta],
\end{gather*} 
where $\bs\Xi  \bs\delta$ is obtained from $\bs\Xi$, viewed as an $\kappa$-tuple of chords on $[\tfrac{1}{2},1]\times \wt W$, and $\bs\delta$, viewed as an $\kappa$-tuple of chords on $[0,\tfrac{1}{2}]\times \wt W$, by concatenating with an $\kappa$-tuple of connecting arcs on $\{\tfrac{1}{2}\}\times \wt {\bs a}$.   Since each component of $\wt {\bs a}$ is simply-connected, $[\bs\Xi\bs\delta]$ does not depend on the choice of connecting arcs. The map $\bs\Xi_*$ is a bijection since there is an analogously defined inverse 
$$\bs\Theta_*:  h_{\wt {\bs a}', \wt h_\sigma(\wt {\bs a})}\to h_{\wt {\bs a}, \wt h_\sigma(\wt {\bs a})},\quad [\bs\delta']\mapsto [\bs \Theta\bs \delta'].$$

Next we choose complete sets of capping surfaces $\{T_{\bs \delta, {\bf y}}\}$ for the pair $(\wt h_\sigma(\wt {\bs a}),\wt {\bs a})$, $\{T_{\bs\Xi_*\bs \delta, {\bf y}'}\}$ for the pair $(\wt h_\sigma(\wt {\bs a}),\wt {\bs a}')$, and $\{T_{\bs\delta, \bs\Xi, \bs\Xi_*\bs\delta}\}$ for the triple $(\wt h_\sigma(\wt {\bs a}),\wt {\bs a}, \wt {\bs a}')$. 

The map $\Phi$ given by Equation~\eqref{eqn: chain map Phi} can then be defined with coefficients 
$$\F[\mathcal{A}_{\wt {\bs a}, \wt h_\sigma(\wt {\bs a})}]\simeq \F[\mathcal{A}_{\wt {\bs a}', \wt h_\sigma(\wt {\bs a})}]$$ 
in view of the following lemma:

\begin{lemma}
$\mathcal{A}_{\wt {\bs a}, \wt h_\sigma(\wt {\bs a})}\simeq\mathcal{A}_{\wt {\bs a}', \wt h_\sigma(\wt {\bs a})}$.
\end{lemma}  

\begin{proof}
Let us write
$$X=[0,1]\times \wt W, ~~~ Y=(\{1\}\times \wt{\bs a})\cup (\{0\}\times \wt h_\sigma (\wt {\bs a})), ~~~ Y'=(\{1\}\times \wt{\bs a}')\cup (\{0\}\times \wt h_\sigma (\wt {\bs a})),$$
and $i:Y\to X$, $i': Y'\to X$ for the inclusions. We have $i_* H_2(Y)=i'_* H_2(Y')$ since $\wt \gamma_1'$ is obtained from $\wt \gamma_1$ by arc sliding over $\wt \gamma_2$. Hence
\begin{align*}
\mathcal{A}_{\wt {\bs a}, \wt h_\sigma(\wt {\bs a})}& =H_2(X,Y)\simeq H_2(X)/i_*H_2(Y)\\
&\simeq  H_2(X)/ i'_*H_2(Y')\simeq  H_2(X,Y') = \mathcal{A}_{\wt {\bs a}', \wt h_\sigma(\wt {\bs a})}.
\end{align*} \vskip-.23in
\end{proof}

\subsubsection{Holomorphic curves after Lagrangian surgery}

As preparation for Theorem~\ref{thm: count}, we state a useful fact which is a restatement of \cite[Theorem 55.11]{FOOO2}.  Let $b_0,\dots, b_m$ be transversely intersecting Lagrangians in $\widehat W$ that are either compact or with cylindrical Legendrian ends and project to arcs under the map $\widehat p: \widehat W\to \C$. 

Let $\mathcal{M}(y_0,\dots, y_m)$ be a moduli space of holomorphic maps 
$$v: D_m=D-\{p_0,\dots, p_m\} \to (\widehat W, J_{\widehat W}),$$ 
such that: 
\be
\item $\bdry_i D_m$, $i=0,\dots,m$, is mapped to $b_i$;
\item $p_i$, $i=0,\dots,m$, corresponds to the intersection point $y_i\in b_{i-1}\cap b_{i}$, where $i$ is viewed modulo $m+1$;
\item the intersection point $y_0$ lies above a singular value $z$ of $\widehat p$;
\item the $p_i$ may be allowed to vary or they may be fixed, and the maps $v$ may or may not have point constraints; 
\item $\mathcal{M}(y_0,\dots, y_m)$ is compact; and
\item all the curves of $\mathcal{M}(y_0,\dots, y_m)$ project to thin sectors between $b_m$ and $b_0$ near $p_0$ as shaded in Figure~\ref{fig: lagrangian-surgery}. 
\ee
Here we are using the notation from the beginning of Section~\ref{section: A infty operations}. 

Now there are two types of Lagrangian surgery operations at $p_0$; the results are called $b_+$ and $b_-$, corresponding to the right and middle pictures of Figure~\ref{fig: lagrangian-surgery}. The modified moduli spaces $\mathcal{M}^+(y_1,\dots, y_m)$ and $\mathcal{M}^-(y_1,\dots, y_m)$ are the analogously defined moduli spaces of maps 
$$v_\pm: D-\{p_1,\dots, p_m\}\to (\widehat W, J_{\widehat W})$$ 
such that (setting $b_\pm=b_0=b_m$):
\be
\item the counterclockwise arc component of $\bdry (D-\{p_1,\dots, p_m\})$ connecting $p_i$ and $p_{i+1}$ is mapped to $b_i$ for $i=1,\dots,m$;
\item $p_i$, $i=1,\dots,m$, corresponds to the intersection point $y_i\in b_{i-1}\cap b_{i}$;
\item if the points $p_0,\dots p_m$ were fixed for $\mathcal{M}(y_0,\dots,y_m)$, then $p_1,\dots, p_m$ are also for $\mathcal{M}^\pm(y_1,\dots, y_m)$; if there were point constraints for $\mathcal{M}(y_0,\dots,y_m)$, then the same point constraints are used for $\mathcal{M}^\pm(y_1,\dots, y_m$); and
\item all the curves of $\mathcal{M}^\pm (y_1,\dots, y_m)$ project to the shaded regions in Figure~\ref{fig: lagrangian-surgery}.
\ee

\begin{figure}[ht]
	\begin{overpic}[scale=1]{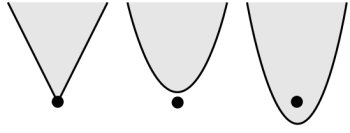}
	\end{overpic}
	\caption{The result of performing the two different kinds of Lagrangian surgery. The picture on the left shows the intersection of $b_m$ and $b_0$, and the pictures in the middle and on the right are $b_-$, $b_+$.}
	\label{fig: lagrangian-surgery}
\end{figure}

\begin{thm}[Fukaya-Oh-Ohta-Ono] \label{thm: Lagrangian surgery}
Suppose $J_{\widehat W}$ has a specific form on a small neighborhood of $y_0$ and $\mathcal{M}(y_0,\dots, y_m)$ is regular.  Then $\mathcal{M}^\pm(y_1,\dots,y_m)$ are regular and there are maps
$$\theta_\pm: \mathcal{M}^\pm(y_1,\dots, y_m)\to\mathcal{M}(y_0,\dots, y_m),$$
where $\theta_-$ is a $C^1$-diffeomorphism and $\theta_+$ is a $C^1$-double covering map.
\end{thm}

\begin{rmk} \label{rmk: general case}
	In general, if $\op{dim}W=2n$, then $\theta_-$ is still a $C^1$-diffeomorphism but $\theta_+$ is a fiber bundle with fiber $S^{n-2}$. 
\end{rmk}

\begin{rmk} \label{rmk: signs}
Theorem~\ref{thm: Lagrangian surgery} is essentially a consequence of the following fact:  Consider the Lefschetz fibration 
\begin{equation}\label{eqn: model for C}
\widehat p: \C^2_{z_1,z_2}\to \C_\zeta,\quad (z_1,z_2)\mapsto z_1z_2.
\end{equation}
Then the Clifford torus 
\begin{equation} \label{eqn: Clifford torus}
T=\{|z_1|=1\}\times \{|z_2|=1\}\subset \C^2
\end{equation}
is a Lagrangian that ``lies over'' $\{|\zeta|=1\}\subset \C$.  There exist two holomorphic disks $u$ in $\C^2$ with boundary on $T$ that pass through a generic point $(a_1,a_2)$ of $T$ and such that $\widehat p\circ u$ is a degree $1$ map to $\{|\zeta|< 1\}$.  They are of the form $z\mapsto (z,a_2)$ and $z\mapsto (a_1,z)$.
\end{rmk}

\subsubsection{Warm-up problem}

We will first treat a slightly easier warm-up problem.  Consider Figure~\ref{fig: An-arcslide}, which is half of Figure~\ref{fig: An-handleslide} and corresponds to the Lefschetz fibration $p: W\to D$. Strictly speaking, we are working in the cylindrical end situation of $\widehat p: \widehat W\to \C$ and $\gamma_i$, $\gamma_i'$, $a_i$, $a_i'$ have been completed by attaching cylindrical ends. In general, the completion of $*$ will be denoted by $\widehat{*}$. 
Let $c_{ij}$ be the shortest counterclockwise chord in $\bdry D$ that connects $\gamma_i$ to $\gamma_j'$. Corresponding to each chord $c_{ij}$ there exists an $S^1$-Morse-Bott family of chords in $\bdry W$; in the perturbed version we consider we consider chords $\check c_{ij}$ and $\hat c_{ij}$ which are the longer and shorter Reeb chords.\footnote{This appears to be the opposite of the previous usage.  Our rule is to put a check on the lower-degree generator in the appropriate Floer cohomology group, which in the current situation is $\widehat{CF}(\bs{ \widehat{a}}, \bs{\widehat{a}'})$.}  We write ${\bf c}= \{\check c_{12},\check c_{21}, \check c_{33},\dots, \check c_{\kappa \kappa}\}$.  
 
\begin{figure}[ht]
	\begin{overpic}[scale=1]{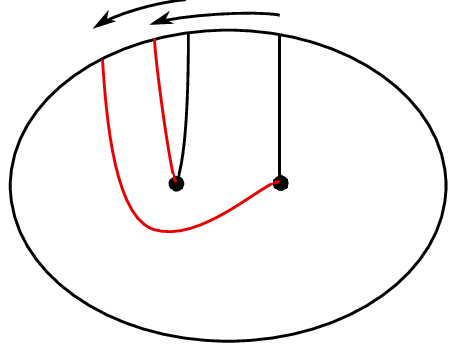}
		\put(41,35){\tiny $\Theta_2$}
		\put(65,35){\tiny $\Theta_1$}
		\put(28,51){\tiny $\gamma_2'$}
		\put(33,20){\tiny $\gamma_1'$}
		\put(42,51){\tiny $\gamma_2$}
		\put(63,51){\tiny $\gamma_1$}
		\put(50,77){\tiny $c_{12}$}
		\put(25,79){\tiny $c_{21}$}
	\end{overpic}
	\caption{}
	\label{fig: An-arcslide}
\end{figure}

Again assuming that $\kappa=2$, we consider the moduli space 
$$\mathcal{M}_{J^\Diamond}({\bf c}, \bs\Theta),$$ 
where the symplectic fibration and Lagrangian boundary conditions are:
$$\pi: \R\times[0,1]\times \widehat W \to \R\times[0,1], \quad L_1= \R\times \{1\}\times \widehat{\bs a}', \quad L_0=\R\times\{0\}\times \widehat {\bs a},$$
and we take $J$ to be the product $J_{\R\times[0,1]}\times J_{\widehat W}$ such that the projection $\widehat p: \widehat W \to \C$ is $(J_{\widehat W},i)$-holomorphic.

\begin{thm} \label{thm: count warm-up}
$\# \mathcal{M}^{\chi=1, {\bf w}}_{J^\Diamond}({\bf c}, \bs\Theta)=1$ mod $2$ for generic ${\bf w}$.
\end{thm}

\begin{proof}
We write $u$ for an element of $\mathcal{M}^{\chi=1, {\bf w}}_{J^\Diamond}({\bf c}, \bs\Theta)$ and $v$ for its projection $\pi_{\widehat W}\circ u: \dot F\to \widehat W$, where $F$ is a unit disk in $\C$.

When $\chi=1$, there are two types of maps $v$:
\begin{itemize}
	\item Type $int$ (for ``interior'') which has a branch point ${\frak b}$ that maps to the interior of $B$;
	\item Type $\bdry$ (for ``boundary'') which does not have a branch point that maps to the interior of $B$.
\end{itemize} 	

We claim that Type $\bdry$ falls into two subcases: $\bdry_1$ and $\bdry_2$, where for $v$ of Type $\bdry_1$ (resp.\ $\bdry_2$) there are two ``switch points" ${\frak b}_1, {\frak b}_2\in \bdry \dot F$ (they may be the same point) that both map to $\widehat \gamma_2'$ (resp.\ both map to $\widehat \gamma_2$) and where $\widehat p\circ v|_{\bdry \dot F}$ switches directions along $\widehat\gamma_2'$ (resp.\ $\widehat \gamma_2$).  This a consequence of the fact that there is only one component $\zeta$ of $\bdry \dot F$ that maps to $\widehat\gamma_2'$ (resp.\ $\widehat\gamma_2$) and $v(\zeta)$ must connect the terminal point of $c_{12}$ to $\widehat{p}(\Theta_2)$ (resp.\ $\widehat{p}(\Theta_2)$ to the initial point of $c_{21}$).

If ${\frak b}_1\not={\frak b}_2$, let ${\frak b}_1$ be the point that is closer to the puncture corresponding to $c_{12}$ (resp. $c_{21}$) along the arc of $\bdry \dot F$ that maps to $\widehat \gamma_2'$ (resp.\  $\widehat \gamma_2$). 

The calculation of $\# \mathcal{M}^{\chi=1, {\bf w}}_{J^\Diamond}({\bf c}, \bs\Theta)$ does not depend on the choice of generic ${\bf w}$ since ${\bf c}$ and $\bs\Theta$ are closed. We therefore choose ${\bf w}$ as follows, subject to genericity:
\begin{enumerate} 
	\item[(R1)] $w_1\in \widehat{a}_1$ so that its projection $\widehat p(w_1)$ to $\widehat\gamma_1$ is close to $z_1$;
	\item[(R2)] $w_2\in \widehat{a}_2$ so that its projection $\widehat p(w_2)$ to $\widehat\gamma_2$ satisfies $|\widehat p(w_2)|\gg 0$.
\end{enumerate}
We also assume that:
\be
\item[(R3)] the region in $\C$ between $\widehat \gamma_2'$ and $\widehat \gamma_2$ corresponding to $B$ is an arbitrarily thin strip whose width is $\frak w$. 
\ee 

\s\n
{\em Order of choice of $w_1, w_2, {\frak w}$.}  This will become important when taking limits using Gromov compactness:
\begin{enumerate}
	\item[(S1)] First choose $w_2$ satisfying (R2) and an interval $[0,1]$ so that 
	$$|\widehat p(w_1)-z_1|\in[0,1].$$
	\item[(S2)] Then choose ${\bf w}= \frak w(w_2)$ for (R3).
	\item[(S3)] Finally choose $w_1$ satisfying (R1).
\end{enumerate}

\s
We also remark that the proof of the regularity of any explicit holomorphic curves with constraints that we use is straightforward and is omitted.

\s\n
{\em Step 1.} We first apply Theorem~\ref{thm: Lagrangian surgery} to push the holomorphic curve $v$ off of the critical value $z_1$; see Figure~\ref{fig: lagrangian-surgery2}.  The mod $2$ curve counts remain the same as long as no point constraints are prescribed at the clean $S^1$-intersection on the right-hand side of the diagram.

\begin{figure}[ht]
	\begin{overpic}[scale=1]{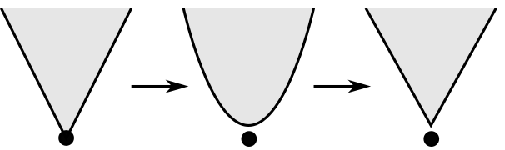}
	\end{overpic}
	\caption{}
	\label{fig: lagrangian-surgery2}
\end{figure}

\s\n
{\em Step 2.}
Fix $w_2$ with $\widehat p (w_2)$ large and $\frak w= \frak w(w_2)$ small. We write $\iota =\op{Im}\circ \widehat p$, where $\op{Im}$ refers to the imaginary part.
We claim that 
$$\iota \circ v ({\frak b}) \gg \iota(w_2) \quad \mbox{or} \quad \iota\circ v({\frak b}_2)  \gg \iota(w_2)$$ 
for $|\widehat p (w_1)- z_1|$ small. (Keep in mind the order (S1)--(S3).)  We will treat the ${\frak b}$ case; the ${\frak b}_2$ case is similar.  We use the notation $q(\Theta_1)$ to denote the point on $\bdry F$ that ``maps to'' $\Theta_1$, etc. See Figure~\ref{fig: involution-newer}.

\begin{figure}[ht]
	\begin{overpic}[scale=1]{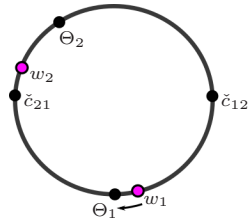}
		\put(39.3,-2){\tiny $\Theta_1$}
		\put(64.2,3.7){\tiny $w_1$}
		\put(101,50){\tiny $\check c_{12}$}
		\put(6.8,50){\tiny $\check c_{21}$}
		\put(9.8,63){\tiny $w_2$}
		\put(23.7,79.7){\tiny $\Theta_2$}
	\end{overpic}
	\caption{The location of the points $q(\Theta_1)$, etc.\ on $\bdry F$, assuming  $\iota\circ v({\frak b}), \iota\circ v ({\frak b}_2)\gg \iota(w_2)$ does not hold. Here $q$ is omitted from the notation.}
	\label{fig: involution-newer}
\end{figure}

Arguing by contradiction, if the claim does not hold, then there exists a sequence $\{v_i: \dot F_i\to \widehat W\}_{i=1}^\infty$ of $v$'s with respect to a sequence $w_1^{(i)}$ of $w_1$'s approaching $\Theta_1$ such that all the $\iota\circ v_i({\frak b}^{(i)})$ are bounded above,  where ${\frak b}^{(i)}$ is ${\frak b}$ for $v_i$.  By Gromov compactness, there exists a subsequence that converges to $v_\infty: \dot F_\infty\to \widehat W$ with a branch point ${\frak b}^\infty$ that is a subsequential limit of ${\frak b}^{(i)}$.  There also exist $\delta>0$ and a normalization of the points on $\bdry F_\infty$ such that 
\begin{itemize}
	\item[(i)] $q(\check c_{12})=1$, $q(\check c_{21})=-1$, $q(\Theta_1)=-i$;
	\item[(ii)] $|q(\Theta_2)-q(\check c_{12})|, |q(\Theta_2)-q(\check c_{21})| >\delta$;
	\item[(iii)] $|q(w_2)-q(\Theta_2)|>\delta$.
\end{itemize}

On the other hand, (i)--(iii) contradicts the requirement that there be an involution of $F_i$ for $i\gg 0$ taking
\begin{equation} \label{eqn: involution constraint}
q(\Theta_1)\mapsto q(\Theta_2), \quad q(\check c_{21})\mapsto q(\check c_{12}), \quad q(w_2)\mapsto q(w_1). 
\end{equation} 
This is because $|q(w_1^{(i)})-q(\Theta_1)|\to 0$ as $i\to\infty$.  
(Recall that such an involution exists if and only if there is a fractional linear transformation of the disk $F_i$ which fixes $q(\check c_{12})$ and $q(\check c_{21})$ and simultaneously moves $q(\Theta_1)$ and $q(\Theta_2)$ to antipodal points on $\bdry F_i$ and $q(w_1)$ and $q(w_2)$ to antipodal points on $\bdry F_i$.)

\s\n
{\em Step 3.}  We claim that $\# \mathcal{M}^{\chi=1, {\bf w},\sharp}_{J^\Diamond}({\bf c}, \bs\Theta)=0$ mod $2$, where $\sharp$ means we are counting curves that satisfy 
$$\iota\circ v({\frak b}) \gg \iota(w_2) \quad \mbox{or} \quad \iota\circ v({\frak b}_1)\geq \iota(w_2)-C,$$
for $C>0$ a fixed constant. 
Recall our convention is that $\iota\circ v({\frak b}_2) \geq \iota\circ v({\frak b}_1)$.  Also $\iota\circ v({\frak b}_2)\gg \iota(w_2)$  by Step 2.

We apply SFT-type stretching with $\iota(w_2)\to \infty$ and count $2$-level curves as given on the left-hand side of Figure~\ref{fig: stretching1}. The components are labeled $1,2,3$ and any $v\in \mathcal{M}^{\chi=1, {\bf w},\sharp}_{J^\Diamond}({\bf c}, \bs\Theta)$ is close to breaking into $v^{(1)}\cup v^{(2)}\cup v^{(3)}$. (The case drawn is for Type $int$, but $v^{(3)}$ may also have two switch points ${\frak b}_1$ and ${\frak b}_2$.) 

\begin{figure}[ht]
	\begin{overpic}[scale=1]{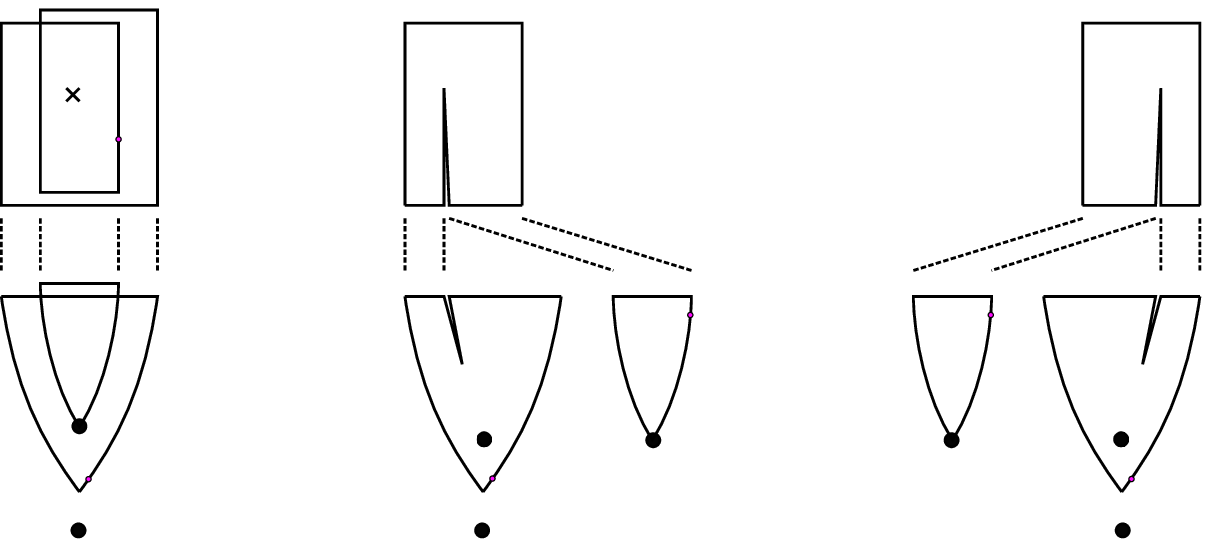}
		\put(43,-5){\tiny $\mbox{Type $\bdry_1$}$}
		\put(85,-5){\tiny $\mbox{Type $\bdry_2$}$}
		\put(2.1,15){\tiny $1$}
		\put(6,16){\tiny $2$}
		\put(5.7,32){\tiny $3$}
		\put(5.4,38.7){\tiny $\frak b$}
		\put(41,16){\tiny $1$}
		\put(53.5,16){\tiny $2$}
		\put(39.5,34){\tiny $3$}
		\put(91,16){\tiny $1$}
		\put(78.6,16){\tiny $2$}
		\put(92.7,34){\tiny $3$}
		\put(36.3,38.7){\tiny $\frak b_2$}
		\put(95.8,38.7){\tiny $\frak b_2$}
		\put(38,12){\tiny $\frak b_1$}
		\put(94,12){\tiny $\frak b_1$}
		\put(7,45){\tiny $\check c_{12}$}
		\put(41,21){\tiny $\check c_{12}$}
		\put(90,21){\tiny $\check c_{21}$}
		\put(94,44){\tiny $\check c_{12}$}
		\put(-1,44){\tiny $\check c_{21}$}
		\put(37,44){\tiny $\check c_{21}$}
		\put(5.5,22){\tiny $\check c_{22}$}
		\put(52.5,21){\tiny $\check c_{22}$}
		\put(78,21){\tiny $\check c_{22}$}
		\put(8.3,4.5){\tiny $w_1$}
		\put(41.9,4.5){\tiny $w_1$}
		\put(10.2,33){\tiny $w_2$}
		\put(58,17.3){\tiny $w_2$}
		\put(95,4.5){\tiny $w_1$}
		\put(83,17.3){\tiny $w_2$}
	\end{overpic}
	\vskip.2in
	\caption{}
	\label{fig: stretching1}
\end{figure}

The chords at the negative ends of $v^{(3)}$ must match the chords at the positive ends of $v^{(1)}$ and $v^{(2)}$. 
By Theorem~\ref{thm: Lagrangian surgery}, the algebraic count of curves $v^{(1)}$ satisfying the point constraint $w_1$ is $0$ mod $2$; in particular the set of chords at the positive end of $v^{(1)}$ is finite.
Since $\iota\circ v({\frak b}),\iota\circ v({\frak b}_2) \gg \iota(w_2)$, there is a portion of $v^{(3)}$ passing through $w_2$ that is close to a curve corresponding to a gradient trajectory through $w_2$; this effectively constrains the positive end of $v^{(2)}$ to a single chord.

We can therefore convert the count of $v^{(3)}$ to the count of $\nu\in \mathcal{M}^{\chi=1}_{J^\Diamond}({\bf c}, {\bf c}')$ that satisfy the following:  The symplectic fibration is
$$ \pi: \R\times[0,1]\times (\C\times T^*S^1)\to \R\times [0,1],$$ 
and the Lagrangian boundary conditions are
$$L_1=\R\times\{1\}\times {\bs a}'^\star,  \quad L_0=\R\times \{0\}\times {\bs a}^\star,$$
where 
\begin{gather*}
a_1^\star=\{x=1\}\times S^1, \quad a_2^\star= \{x=0\}\times S^1,\\
a_1'^\star= \{x=-1\}\times S^1, \quad a_2'^\star= \{x=-\epsilon\}\times S^1,
\end{gather*} 
are in $\C\times T^*S^1$, $\epsilon>0$ is small, and $S^1$ is the zero section of $T^*S^1$. Also the positive and negative ends of $\nu$ are
$${\bf c}=\{\check c_{12}, \check c_{21}\}, \quad  {\bf c}'=\{\check c_{11}, \check c_{22}\},$$
where $\check c_{ij}, \hat c_{ij}$ are the longer and shorter Reeb chords in the perturbed version of the $S^1$-Morse-Bott family of Reeb chords from $a_i^\star$ to $a_j'^\star$.  
\begin{figure}[ht]
	\begin{overpic}[scale=1.3]{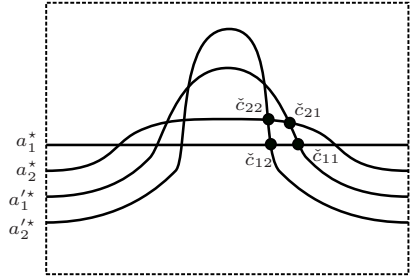}
		\put(-8,36){\tiny $a_1^\star$}
		\put(-8,28){\tiny $a_2^\star$}
		\put(-10,20){\tiny $a_1'^\star$}
		\put(-10,12){\tiny $a_2'^\star$}
		\put(54.6,31){\tiny $\check c_{12}$}
		\put(68.1,44.6){\tiny $\check c_{21}$}
		\put(73,32){\tiny $\check c_{11}$}
		\put(52,46){\tiny $\check c_{22}$}
	\end{overpic}
	\caption{The fiber $T^*S^1$ where the sides are identified.}
	\label{fig: TstarS1}
\end{figure}

On the other hand, generically $\mathcal{M}^{\chi=1}_{J^\Diamond}({\bf c}, {\bf c}')=\varnothing$.  We give two proofs: (P1)  Applying Hamiltonian perturbations to $S^1$ (on $T^*S^1$) we can push off the $a_i^\star$ and $a_i'^\star$ as in Figure~\ref{fig: TstarS1}. We then note that there is no domain in the picture with all positive weights which could be the projection of $\nu$ with positive corners $\check c_{12}, \check c_{21}$ and negative corners $\check c_{11}, \check c_{22}$.  (P2) Alternatively, we can verify that the Fredholm index $\op{ind}(\nu)=0$; this is consistent with $\op{ind}(v^{(1)})=2$ and $\op{ind}(v^{(2)})=2$ when we view $v^{(2)}$ as having $\check c_{22}$ at the positive end. 

The claim then follows.

\s\n
{\em Step 4.} We make one model calculation. {\em  Any notation introduced here is limited to this step.}

Consider the Lefschetz fibration $\widehat p$ given by Equation~\eqref{eqn: model for C}.  Let $T$ be the Clifford torus $\{|z_1|=1\}\times \{|z_2|=1\}$ over $|\zeta|=1$ and let $L$ be the Lagrangian thimble in $\C^2$ emanating from $\zeta=0$ along $$[-1,0]=\{-1\leq \op{Re}(\zeta)\leq 0, \op{Im}(\zeta)=0\}.$$ 
We note that for any $\zeta= re^{i\theta}\in \widehat p (T\cup L)$, $$\widehat p^{-1}(\zeta)\cap (T\cup L)=\{ \sqrt{r}( e^{i\phi},e^{i(-\phi+\theta)}),\phi\in [0,2\pi]\}.$$
Let $\mu$ be the $S^1$-family of intersections $T\cap L$.

Let $\mathcal{M}_J^{w_1}$ be the moduli space of holomorphic maps 
$$u=(u_1,u_2): \R_\sigma\times [0,1]_\tau\to \C^2$$ 
with respect to the standard complex structure $J$ satisfying the following:
\be
\item[(B1)] $u(\R\times\{0\})\subset T$ and $u(\R\times \{1\})\subset L$;
\item[(B2)] $u$ limits to points of $\mu$ as $\sigma\to \pm \infty$;
\item[(B3)] $\widehat p\circ u$ has degree $1$ over $\{|\zeta|\leq 1\}- \widehat p(T\cup L)$ and degree $0$ outside $\{|\zeta|\leq 1\}$;
\item[(B4)] $u(0,0)=w_1=(w_{11},w_{12})$.
\ee

In our current situation we take $w_1=(1,1)$.
Its boundary $\bdry \mathcal{M}^{w_1}_J$ consists of pairs $v\cup c$, where $c$ is a trivial strip that maps to a point in $\mu$ and $v: D^2\to \C^2$ is a disk bubble that passes through $(1,1)$ and such that $\widehat p\circ v$ has degree $1$ over $\{|\zeta|\leq 1\}$.  Recall that $\#\bdry \mathcal{M}^{w_1}_J=2$ by Remark~\ref{rmk: signs} and that the two maps $v$ are of the form $z\mapsto (z,1)$ and $z\mapsto (1,z)$.  See Figure~\ref{fig: modelcalc}, which gives a schematic description of $\overline{\mathcal{M}}^{w_1}_J$. The left-hand side of each row represents the image of $\widehat p\circ u$ or $\widehat p\circ v$ in $\C$ and the right-hand side of each row represents the image of $u$ or $v$ in $\C^2$.  The top and bottom maps are $v$ where $v\cup c\in \bdry \mathcal{M}^{w_1}_J$ and the middle map is of the form $(g(z),g(z))$ where $g$ is a conformal map to the half-disk.

\begin{figure}[ht]
	\begin{overpic}[scale=0.7]{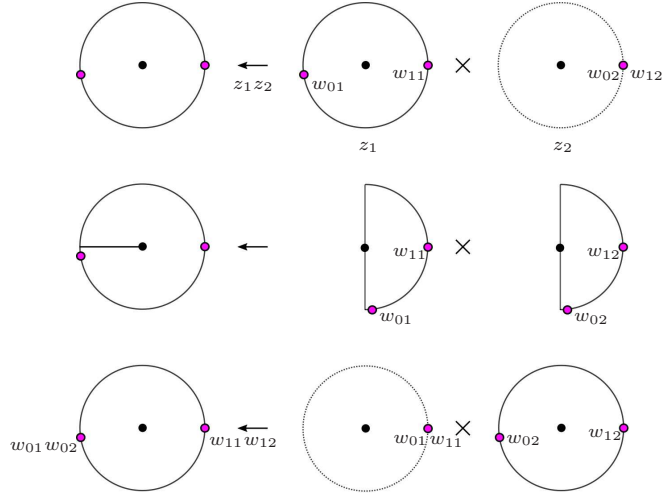}
		\put(24,8.6){\tiny $w_{11}w_{12}$}
		\put(-12,7.7){\tiny $w_{01}w_{02}$}
		\put(57.3,9){\tiny $w_{01}$}
		\put(63.8,8.7){\tiny $w_{11}$}
		\put(77.7,9){\tiny $w_{02}$}
		\put(92.5,10){\tiny $w_{12}$}
		\put(54.7,31){\tiny $w_{01}$}
		\put(57.2,42.7){\tiny $w_{11}$}
		\put(90,31){\tiny $w_{02}$}
		\put(42.5,73){\tiny $w_{01}$}
		\put(57.1,75){\tiny $w_{11}$}
		\put(92.2,75){\tiny $w_{02}$}
		\put(100,75){\tiny $w_{12}$}
		
		\put(92.2,42.7){\tiny $w_{12}$}
		\put(28.3,73){\tiny $z_1z_2$}
		\put(51,62){\tiny $z_1$}
		\put(86.1,62){\tiny $z_2$}
	\end{overpic}
	\caption{A schematic description of $\overline{\mathcal{M}}^{w_1}_J$.  The top and bottom rows represent curves $v$ where $v\cup c\in\bdry \mathcal{M}^{w_1}_J$ and the middle row represents a curve $u$ in the ``middle of" $\mathcal{M}^{w_1}_J$. All the circles are unit circles. A solid circle indicates the boundary of a holomorphic disk and a dotted circle with a red dot on it indicates a constant map to $1$.}
	\label{fig: modelcalc}
\end{figure}

Next pick a point $e^{i(\pi+\epsilon)}$ with $\epsilon>0$ small and consider the evaluation map
\begin{equation} \label{eqn: evaluation map}
ev_J:\mathcal{M}^{w_1}_J\to S^1_{|z_1|=1}, \quad u\mapsto w_{01},
\end{equation}
where $w_0=(w_{01},w_{02})$ is the unique point of intersection between $u(\R\times \{0\})$ and $\widehat{p}^{-1}(e^{i(\pi+\epsilon)})$.  While $J$ is not necessarily regular, there exists a small perturbation $J^\Diamond$ of $J$ with an analogously defined moduli space $\mathcal{M}^{w_1}_{J^\Diamond}$ and an evaluation map $ev_{J^\Diamond}: \mathcal{M}^{w_1}_{J^\Diamond}\to S^1$, such that $\mathcal{M}^{w_1}_{J^\Diamond}$ is transversely cut out and $\bdry \mathcal{M}^{w_1}_{J}=\bdry \mathcal{M}^{w_1}_{J^\Diamond}$. The condition $\bdry \mathcal{M}^{w_1}_{J}=\bdry \mathcal{M}^{w_1}_{J^\Diamond}$ is possible because the maps $v$ are transversely cut out by automatic transversality.

We claim that, if $z_1\in C=\{e^{i\theta_1}~|~\pi+\epsilon<\theta_1<2\pi\}$, then $\# ev_{J^\Diamond}^{-1}(z_1)=1$ modulo $2$. The evaluation map over $C$ is a degree $1$ map since:
\begin{itemize}
	\item The analogously defined $ev^{-1}_{J^\Diamond}$ values for $v\in \bdry \mathcal{M}^{w_1}_{J^\Diamond}$ are the endpoints $e^{i(\pi+\epsilon)}$ and $e^{2\pi i}$ of $C$.
	\item For $i=1,2$, consider the restriction $u_i(s,0)$ of $u_i$ to $\R_s\times\{0\}$, noting that $u_i(s,0)\in S^1_{|z_i|=1}$.  Since $u_1(s,0)$, $u_2(s,0)$, and $u_1(s,0)u_2(s,0)$ all monotonically rotate in the counterclockwise manner as $s\in \R$ increases and $u_i(0,0)=0$, $ev_J(\mathcal{M}^{w_1}_J)\subset C$.  By Gromov compactness, $ev_{J^\Diamond}(\mathcal{M}^{w_1}_{J^\Diamond})$ is contained in a small neighborhood of $C$.
\end{itemize}

\s\n
{\em Step 5.} We claim that $\# \mathcal{M}^{\chi=1, {\bf w}, \flat}_{J^\Diamond}({\bf c}, \bs\Theta)=1$ mod $2$, where we are counting curves of Type $\bdry$ that satisfy
\begin{equation} \label{eqn: iotas}
\iota\circ v({\frak b}_2)\gg \iota(w_2) \quad \mbox{but} \quad \iota\circ v({\frak b}_1)\leq \iota(w_2)-C,
\end{equation} 
for $C>0$ fixed.  We can apply SFT-type stretching and count $2$-level curves as given on the middle and right-hand side of Figure~\ref{fig: stretching1}. In view of \eqref{eqn: iotas}, we may take the long neck to be above $\iota(w_2)$.

We will treat the Type $\bdry_1$ case; the Type $\bdry_2$ case is complementary and analogous. Component $2$ corresponds to a gradient trajectory and there is a unique curve from $\check c_{22}$ to $\Theta_2$ passing through $w_2$. Component 3 has $\check c_{21}$ at the positive end and $\hat d_{21}$ and $\check c_{22}$ at the negative end, where $\hat d_{21}$ is the shorter Reeb chord from $a_2'$ to $a_1'$.  There is a single such curve (modulo $2$) by a calculation analogous to that of Figure~\ref{fig: TstarS1}: we perturb the $T^*S^1$-projections of the ends of $a_2,a_2',a_1'$ as in Figure~\ref{fig: TstarS1} and count a single holomorphic triangle with vertices $\check  c_{21}, \hat d_{21}, \check c_{22}$.   Next for Component 1 we view  $\hat d_{21}$ as a point constraint which corresponds to $w_0$ in Step 4 ($w_1$ in Step 4 directly corresponds to $w_1$ in our case). By Step 4, there is a single curve (modulo $2$) corresponding to Component $1$ which passes through $w_1$ and ``half" of the values of $\hat d_{21}$ (which in turn come from half of the values of $w_2$).  The reason the Type $\bdry_2$ case is complementary is that the ``other half" of the values of $w_2$ are taken care of by Type $\bdry_2$.

Finally, still assuming we are in the Type $\bdry_1$ case, we glue the $3$ components and impose the involution condition \eqref{eqn: involution constraint} on $F$; also refer to Figure~\ref{fig: involution-newer}. As we take $\iota\circ v({\frak b}_2)\to \infty$, $q(\Theta_2)$ approaches $q(\check c_{21})$ but $|q(w_2)-q(\Theta_2)|\ll |q(w_2)-q(\check c_{21})|$.  Hence, assuming $w_1$ is sufficiently close to $\Theta_1$, there is a single value of $\iota\circ v({\frak b}_2)$ for which there exists an involution of $F$ satisfying \eqref{eqn: involution constraint}.  This completes the proof of Theorem~\ref{thm: count}. 
\end{proof}

\subsubsection{Curve count}

\begin{thm} \label{thm: count}
$\# \mathcal{M}^{\chi=0, {\bf w}}_{J^\Diamond}(\bs\Xi, \bs\Theta)=1$ mod $2$ for generic ${\bf w}$.
\end{thm}

\begin{proof}
We write $u:\dot F\to \R\times[0,1]\times \wt W^\wedge$ for an element of $\mathcal{M}^{\chi=0, {\bf w}}_{J^\Diamond}(\bs\Xi, \bs\Theta)$ and $v=v(u)$ for its projection to $\wt W^\wedge$.
	
We stretch the base $D'$ in the $\op{Im} z$-direction, as given in Figure~\ref{fig: stretching2}.  This means that we are keeping $\op{Re} z_i=\op{Re} z_{i+\kappa}$ fixed and taking $\op{Im} z_i=-2K$, $\op{Im} z_{i+\kappa}=2K$, for $i=1,\dots,\kappa$ and $K\gg 0$.  The region $\mathcal{R}$ bounded by $\wt \gamma_1'$ and $\wt\gamma_1$ is divided into three regions $\mathcal{R}_1$, $\mathcal{R}_2$, and $\mathcal{R}_3$, which are intersections of $\mathcal{R}$ with $\{\op{Im} z\leq -K\}$, $\{-K\leq \op{Im} z\leq K\}$, and $\{K\leq \op{Im} z\}$. 
\begin{figure}[ht]
	\begin{overpic}[scale=1]{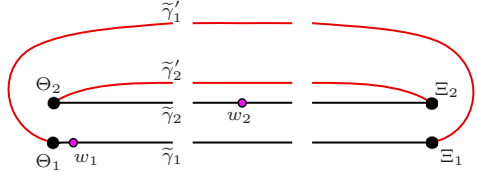}
		\put(6,-3.8){\tiny $\Theta_1$}
		\put(6,12.7){\tiny $\Theta_2$}
		\put(92.6,-2.8){\tiny $\Xi_1$}
		\put(91.5,11.7){\tiny $\Xi_2$}
		\put(14.2,-2.8){\tiny $w_1$}
		\put(47,6.2){\tiny $w_2$}
		\put(33,16){\tiny $\wt \gamma_2'$}
		\put(33,6.2){\tiny $\wt \gamma_2$}
		\put(33,-2.5){\tiny $\wt \gamma_1$}
		\put(33,28.5){\tiny $\wt \gamma_1'$}
	\end{overpic}
	\caption{The horizontal direction is the $\op{Im} z$-direction. The regions $\mathcal{R}_1$, $\mathcal{R}_2$, $\mathcal{R}_3$ are from left to right.}
	\label{fig: stretching2}
\end{figure}

Since the calculations do not depend on the choice of generic ${\bf w}$, we choose ${\bf w}$ to satisfy the following, subject to genericity:
\begin{enumerate} 
	\item[(R1$'$)] $w_1\in \wt{a}_1$ so that its projection $\wt p(w_1)$ to $\wt\gamma_1$ is close to $z_1$;
	\item[(R2$'$)] $w_2\in \wt{a}_2$ so that its projection $\wt p(w_2)$ to $\wt\gamma_2$ is close to $\op{Im} z=0$.
\end{enumerate}	
We also assume that:
\be
\item[(R3$'$)] the region in $\C$ between $\wt \gamma_2'$ and $\wt \gamma_2$ is an arbitrarily thin strip whose width is $\frak w$. 
\ee 
Additional conditions will be imposed later.

We classify the types of maps $v$ by their branch points and switch points.  The types are $2$, $1L$, $1R$, $0LL$, $0LR$, $0RR$, where the number is the number of interior branch points and each occurrence of $L$ (resp.\ $R$) indicates two switch points that map to $\wt \gamma_2'$ (resp.\ $\wt \gamma_2$).  If they exist, the interior branch points are enumerated by ${\frak b}, {\frak b}'\in \op{int}(\dot F)$ and the switch points by ${\frak b}_1,\dots, {\frak b}_4\in \bdry \dot F$. The switch points come in adjacent pairs $({\frak b}_1, {\frak b}_2)$ and $({\frak b}_3, {\frak b}_4)$ that map to the same $\wt\gamma_i'$ and we assume that $\iota({\frak b}_{j+1})\geq \iota({\frak b}_{j})$ for $j=1,3$; and if there is more than one branch point we assume that $\iota({\frak b}')\geq \iota({\frak b})$.  Here we are writing $\iota(x)= \op{Im}(\wt p \circ v(x))$ for $x\in \dot F$. 

\s\n
{\em Step 1.}  Suppose $K\gg 0$. We describe the limit of $v^{(i)}: F^{(i)}\to \wt W^\wedge$ as ${\frak w}_i \to 0$, where ${\frak w}_i$ is the width of the thin strip and $u^{(i)}\in \mathcal{M}^{\chi=0, {\bf w}}_{J^\Diamond}(\bs\Xi, \bs\Theta)$ and $v^{(i)}=v(u^{(i)})$ are with respect to ${\frak w}_i$. 

After possibly passing to a subsequence the limit of $v^{(i)}$ is $v_\infty\cup \delta_+\cup\delta_-$, where
\begin{itemize}
	\item $v_\infty: F_\infty\to \wt W^\wedge$ is a holomorphic annulus,
	\item $\delta_+$ is a gradient trajectory on $\wt a_2$ from $\Xi_2$ to a point $p_+$ on $\op{Im} v_\infty$,  and
	\item $\delta_-$ is a gradient trajectory on $\wt a_2$ from a point $p_-$ on $\op{Im} v_\infty$ to $\Theta_2$.
\end{itemize}     
Here $v_\infty$ could be degenerate in the sense that $F_\infty$ is a disk with an interior point $F_\infty$ that ``behaves like" a boundary component of $F_\infty$ and maps to $\wt a_2$.
We also use the following conventions:
\begin{itemize}
	\item If $w_2$ lies on the gradient trajectory $\delta_\pm$, then $q(w_2)= q(p_\pm)\in \bdry F_\infty$. 
	\item $q(\Theta_2)=q(p_-)$ and $q(\Xi_2)=q(p_+)$ on $\bdry F_\infty$.
\end{itemize}
In particular, in the limit it is possible for $q(w_2)= q(\Theta_2)$ or $q(\Xi_2)$. (In the ``close to breaking" picture for $v^{(i)}$ with $i\gg 0$, $q(w_2)$ is close to $q(\Theta_2)$ or $q(\Xi_2)$.)

The limit of Figure~\ref{fig: stretching2} is given by Figure~\ref{fig: stretching4}. 
\begin{figure}[ht]
	\begin{overpic}[scale=1]{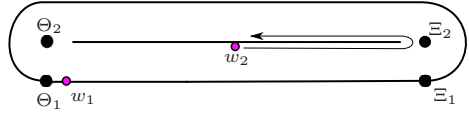}
		\put(6,-3.8){\tiny $\Theta_1$}
		\put(6,12.7){\tiny $\Theta_2$}
		\put(92.6,-2.8){\tiny $\Xi_1$}
		\put(91.5,11.7){\tiny $\Xi_2$}
		\put(13.5,-2.8){\tiny $w_1$}
		\put(47,5.8){\tiny $w_2$}
	\end{overpic}
	\caption{The limit of Figure~\ref{fig: stretching2} as we take ${\frak w}_i\to 0$.  In the picture, the red dot $w_2$ is placed just below the slit, indicating that $v_\infty$ passes through $w_2$ along the lower branch of the slit. The arrow indicates the possible locations of $\wt p \circ v_\infty(q(w_2))$. }
	\label{fig: stretching4}
\end{figure}

Let ${\frak b}^{(i)}, {\frak b}^{'(i)}, {\frak b}_j^{(i)}$ be ${\frak b}, {\frak b}', {\frak b}_j$ for ${\frak w}_i$ small and let ${\frak b}^\infty, {\frak b}^{'\infty}, {\frak b}_j^\infty$ be their limits after passing to a subsequence. {\em We will often omit the superscripts $(i)$ and $\infty$.}

\s\n
{\em Step 2.} We claim that, for $K\gg 0$ and ${\frak w}_i$ small, we can restrict to the case where:
\begin{gather}
\label{eqn: first} \op{max}\{\iota({\frak b'}), \iota({\frak b}_2),\iota({\frak b}_4)\}\geq K \quad \mbox{and}\\
\label{eqn: second} \op{min}\{\iota({\frak b}), \iota({\frak b}_1),\iota({\frak b}_3)\}\leq -K,
\end{gather}
and the left-hand sides of \eqref{eqn: first} and \eqref{eqn: second} correspond to the endpoints of the slit in the limit ${\frak w}_i\to 0$. 

We argue as in Step 3 of Theorem~\ref{thm: count warm-up}.  If \eqref{eqn: second} does not hold, then the restriction of $v_\infty$ to $\mathcal{R}_1$ (for $K\gg 0$) can be viewed as a curve of the type described in Remark~\ref{rmk: signs}.  Since such curves come in pairs, the mod $2$ count is unaffected by restricting to curves satisfying \eqref{eqn: second}. The argument for \eqref{eqn: first} is similar.

\s\n
{\em Step 3.}  We show that, for $K\gg 0$ and ${\frak w}_i$ small,  $\#\mathcal{M}^{\chi=0, {\bf w}, *}_{J^\Diamond}(\bs\Xi, \bs\Theta)=0$ mod $2$, where $*=2, 1L, 0LL$, or $0RR$.  By Step 2 we may restrict to the case where \eqref{eqn: first} and \eqref{eqn: second} hold. Hence in all the cases, $v_\infty$ has a long slit along $\wt \gamma_2'=\wt \gamma_2$ at least from $\op{Im} z=-K$ to $\op{Im}z=K$.

\s
{\em Type $2$.} Note that $v_\infty$ maps $q(\Theta_2)$ and $q(\Xi_2)$ to the endpoints of the slit.  Hence $q(\Theta_2)$ cannot be close to $q(w_2)$, a contradiction, and $v$ cannot be of Type $2$.

\s
{\em Type $1L$ or $1R$.} There are three cases.  First assume that $\iota({\frak b})\leq \iota({\frak b}_1)\leq \iota({\frak b}_2)$.  In this case $\iota({\frak b})$ and $\iota({\frak b}_2)$ are the endpoints of the slit, $q(\Theta_2)$ maps to $\iota({\frak b})$, and $q(\Xi_2)$ maps to $\iota({\frak b}_1)$.  If $\iota({\frak b}_1)\geq 0$, then the images of $q(\Theta_2)$ and $q(w_2)$ are too far, a contradiction. If $\iota({\frak b}_1)\leq 0$, then the images of $q(\Theta_2)$ and $q(\Xi_2)$ are too close, which is also a contradiction.  
 
Next assume that $\iota({\frak b}_1)\leq \iota({\frak b})\leq \iota({\frak b}_2)$. In this case $\iota({\frak b}_1)$ and $\iota({\frak b}_2)$ are the endpoints of the slit and $q(\Theta_2)$ and $q(\Xi_2)$ map to $\iota({\frak b})$. Hence the images of $q(\Theta_2)$ and $q(\Xi_2)$ are too close, a contradiction.

Finally assume that $\iota({\frak b}_1)\leq \iota({\frak b}_2)  \leq \iota({\frak b})$, where $q(\Theta_2)$ maps to $\iota({\frak b}_2)$ and $q(\Xi_2)$ maps to $\iota({\frak b})$.  If $\iota({\frak b}_2)\leq -1$, then the images of $q(\Theta_2)$ and $q(w_2)$ are too far, a contradiction. If $\iota({\frak b}_2)\geq -1$, then the images of $q(\Theta_2)$ and $q(\Xi_2)$ are too close, which is also a contradiction.

\s
{\em Type $0LL$ or $0RR$.} We treat Type $0LL$; Type $0RR$ is similar. We may assume that
$$\iota ({\frak b}_3) \leq  \iota({\frak b}_4) \leq \iota({\frak b}_1) \leq \iota({\frak b}_2).$$ 
This is because $\iota({\frak b}_1)< \iota({\frak b}_4)$ is incompatible with $v_\infty$ having degree $1$ over $\mathcal{R}$ and degree $2$ over the thin strip.
Then $\iota({\frak b}_3)$ and $\iota({\frak b}_2)$ are the endpoints of the slit, $q(\Theta_2)$ maps to $\iota({\frak b}_4)$, and $q(\Xi_2)$ maps to $\iota({\frak b}_1)$. 
If $\iota({\frak b}_4)\leq -1$, then the images of $q(\Theta_2)$ and $q(w_2)$ are too far, a contradiction.  If $\iota({\frak b}_4)\geq -1$, then the images of $q(\Theta_2)$ and $q(\Xi_2)$ are too close, which is also a contradiction.  

\s
It remains to consider Type $0LR$.

\s\n
{\em Step 4.} Let $K\gg 0$.  We describe the limit $v_\infty\cup \delta_+\cup \delta_-$ of $v^{(i)}$ as ${\frak w}_i\to 0$ (from Step 1), where $u^{(i)}\in\#\mathcal{M}^{\chi=0, {\bf w},0LR}_{J^\Diamond}(\bs\Xi, \bs\Theta)$.

Suppose for $v^{(i)}$ the switch points ${\frak b}_1^{(i)}, {\frak b}_2^{(i)}$ map to $\wt\gamma_2'$ and the switch points ${\frak b_3^{(i)}}, {\frak b}_4^{(i)}$ map to $\wt \gamma_2$. As before we are assuming 
$$\iota({\frak b}_{2}^{(i)})\geq \iota({\frak b}_{1}^{(i)}), \quad \iota({\frak b}_{4}^{(i)})\geq \iota({\frak b}_{3}^{(i)}).$$ 

By Step 2, we have:

\begin{claim} \label{claim A}
	$v_\infty$ has a long slit along $\wt \gamma_2'=\wt \gamma_2$ at least from $\op{Im} z=-K$ to $\op{Im}z=K$ with endpoints 
	$$\op{max}(\iota({\frak b}_2^\infty), \iota({\frak b}_4^\infty)),\quad \op{min}(\iota({\frak b}_1^\infty), \iota({\frak b}_3^\infty)).$$ 
\end{claim} 

See Figure~\ref{fig: slit} for an example. The restriction $C_+$ of $v^{(i)}$ over the shaded region converges to the trajectory $\delta_+$; there is an analogous curve $C_-$ converging to the trajectory $\delta_-$.  
\begin{figure}[ht]
	\begin{overpic}[scale=1]{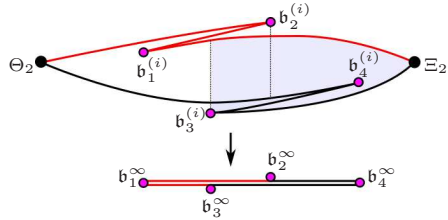}
		\put(-7,32.5){\tiny $\Theta_2$}
		\put(101,32.5){\tiny $\Xi_2$}
		\put(26,31){\tiny ${\frak b}_1^{(i)}$}
		\put(63,45){\tiny ${\frak b}_2^{(i)}$}
		\put(36,18){\tiny ${\frak b}_3^{(i)}$}
		\put(81,32.5){\tiny ${\frak b}_4^{(i)}$}
		\put(21.2,3.7){\tiny ${\frak b}_1^{\infty}$}
		\put(60,8){\tiny ${\frak b}_2^{\infty}$}
		\put(43,-4.2){\tiny ${\frak b}_3^{\infty}$}
		\put(86,3.7){\tiny ${\frak b}_4^{\infty}$}
	\end{overpic}
	\caption{A schematic picture of the switch points and the limit slit as $i\to \infty$. In the figure ${\frak b}_1^{(i)}$ is shorthand for $\wt p\circ v^{(i)}({\frak b}_1^{(i)})$, for example.}
	\label{fig: slit}
\end{figure}

We also have: 

\begin{claim} \label{claim B}
	\begin{align*}
		q(p_+)& = \left\{ \begin{array}{ll}  {\frak b}_3^\infty & \mbox{ if } \iota({\frak b}_3^\infty)\geq \iota({\frak b}_1^\infty),\\
			{\frak b}_1^\infty & \mbox{ if }  \iota({\frak b}_1^\infty)\geq \iota({\frak b}_3^\infty). \end{array} \right.\\
		q(p_-)& = \left\{  \begin{array}{ll}  {\frak b}_2^\infty & \mbox{ if } \iota({\frak b}_4^\infty)\geq \iota({\frak b}_2^\infty),\\
			{\frak b}_4^\infty & \mbox{ if }  \iota({\frak b}_2^\infty)\geq \iota({\frak b}_4^\infty). \end{array} \right.
	\end{align*}
\end{claim} 

The various possibilities for ${\frak b}_j^\infty$ around the slit are given in Figure~\ref{fig: slit2} and are denoted by (A)--(D).

\begin{figure}[ht]
	\begin{overpic}[scale=1.3]{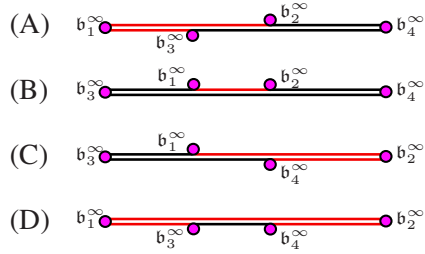}
		\put(-8,70){\tiny ${\frak b}_1^{\infty}$}
		\put(61,74.4){\tiny ${\frak b}_2^{\infty}$}
		\put(19,64){\tiny ${\frak b}_3^{\infty}$}
		\put(100,70){\tiny ${\frak b}_4^{\infty}$}
		
		\put(19.5,53){\tiny ${\frak b}_1^{\infty}$}
		\put(61,52.4){\tiny ${\frak b}_2^{\infty}$}
		\put(-8,48.3){\tiny ${\frak b}_3^{\infty}$}
		\put(100,48){\tiny ${\frak b}_4^{\infty}$}
		
		\put(19.5,31){\tiny ${\frak b}_1^{\infty}$}
		\put(100,26.8){\tiny ${\frak b}_2^{\infty}$}
		\put(-8,26.3){\tiny ${\frak b}_3^{\infty}$}
		\put(61,20){\tiny ${\frak b}_4^{\infty}$}
		
		\put(-8,5){\tiny ${\frak b}_1^{\infty}$}
		\put(100,4.7){\tiny ${\frak b}_2^{\infty}$}
		\put(19,-1.8){\tiny ${\frak b}_3^{\infty}$}
		\put(61,-2){\tiny ${\frak b}_4^{\infty}$}
		
		\put(-30,69){(A)}
		\put(-30,47){(B)}
		\put(-30,25){(C)}
		\put(-30,3){(D)}
	\end{overpic}
	\caption{The various possibilities for ${\frak b}_j^\infty$ around the slit.}
	\label{fig: slit2}
\end{figure}

\s\n
{\em Step 5.} We make one model calculation, which is a slightly more involved variant of Step 4 of Theorem~\ref{thm: count warm-up}.  
Let $\widehat p$, $T$, $L$, $\mathcal{M}^{w_1}_J$ be as in Step 4 of Theorem~\ref{thm: count warm-up}. 

We will define two moduli spaces $\mathcal{M}_1^\theta$ and $\mathcal{M}_2$ of holomorphic maps $\R\times[0,1]\to\C^2$, define the evaluation maps
$$ev_1, ev_2: \mathcal{M}_1^\theta, \mathcal{M}_2\to S^1\times S^1=\R/2\pi \Z \times \R/2\pi \Z,$$ 
and determine their images.   The fiber product of $ev_1$ and $ev_2$, i.e., the set of pairs $(v',v'')\in \mathcal{M}_1^\theta\times \mathcal{M}_2$ such that $ev_1(v')=ev_2(v'')$ with $\theta=-\tfrac{\pi}{2}+\tilde \epsilon$ (shown to be given by the intersection of the blue line and the pink region in Figure~\ref{fig: torus2})  will be the model for the gluing of $v'_\infty$ and $v''_\infty$ (cf.\ \eqref{eqn: gluing}) from Step 6.





\s
{\em Step 5A.} For $\theta\in [0,\tfrac{\pi}{2})$, let $\mathcal{M}_1^\theta= \mathcal{M}^{w_1}_J$, where $w_1=(e^{i\theta},e^{-i\theta})$.  
Let $w_\pm= e^{i(\pi\pm\epsilon)}\in \C_\zeta$, where $\epsilon>0$ is small, and let $ev_{1\pm}(v')$ be the $z_1$-coordinate of $(\widehat p \circ v')^{-1}(w_\pm)$, where $v'\in \mathcal{M}^\theta_1$.  We then define the map 
\begin{gather*}
	ev_1:\mathcal{M}^\theta_1\to S^1\times S^1,\\
	v'\mapsto (ev_{1+}(v'), ev_{1-}(v')).
\end{gather*}

The calculation of Step 4 of Theorem~\ref{thm: count warm-up} (i.e., examining the $1$-parameter family of maps in Figure~\ref{fig: modelcalc} that  depicts $\mathcal{M}^{\theta=0}_1$) implies that $ev_1(\overline{\mathcal{M}}^{0}_1)$ is homotopic rel endpoints to
\begin{gather} \label{eqn: ev1 part 0}
\{(1-t)(\pi+\epsilon, \pi-\epsilon)+ t(2\pi,0)~|~ t\in [0,1]\}.
\end{gather}
Note that there is a diffeomorphism 
$$\mathcal{M}_1^0\stackrel\sim\to\mathcal{M}_1^\theta, \quad v'\mapsto v'_\theta,$$
where $v'_\theta$ is obtained from $v'$ by rotating the first component by $\theta$ and the second component by $-\theta$. It follows that
\begin{gather} \label{eqn: ev1 part theta}
	ev_1(\mathcal{M}^{\theta}_1)=ev_1(\mathcal{M}^{0}_1)+(\theta,\theta).
\end{gather}
We will take $\theta=-\tfrac{\pi}{2}+\tilde\epsilon$ where $\tilde\epsilon>0$ is small. See Figure~\ref{fig: torus2}. 

\begin{figure}[ht]
	\begin{overpic}[scale=.8]{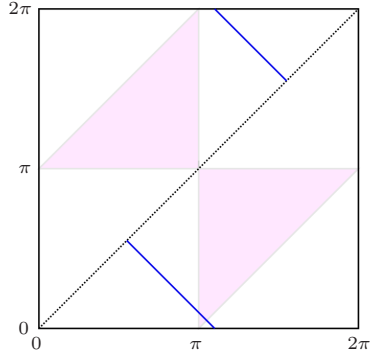}
		\put(-2,-6){\tiny $0$}
		\put(47,-6){\tiny $\pi$}
		\put(96,-6){\tiny $2\pi$}
		\put(-6,-1.3){\tiny $0$}
		\put(-6,48.5){\tiny $\pi$}
		\put(-9,98){\tiny $2\pi$}
	\end{overpic}
	\caption{The torus $S^1\times S^1$ with coordinates $(\theta_1,\theta_2)$.  The sides are identified and the top and the bottom are identified. The dotted line is the diagonal.  The blue line represents $ev_1(\mathcal{M}^\theta_1)$, where $\theta=-\tfrac{\pi}{2}+\tilde\epsilon$, and the pink region represents $ev_2(\mathcal{M}_2)$, as we take $\epsilon\to 0$.}
	\label{fig: torus2}
\end{figure}

\s
{\em Step 5B.}
The moduli space $\mathcal{M}_2$ is the set of maps $v'':\R\times[0,1]\to \C^2$ that satisfy (B1)--(B3) and
\be
\item[(B4$'$)] $v''$ nontrivially intersects $\ell:=\{(-r,r)~|~ r\in [0,1]\}\subset  L$. (Note that $\ell$ is obtained by parallel transporting a point on $L$.)
\ee
We will denote a point of intersection between $v''$ and $\ell$ by $w_0= w_0(v'')$. Let $ev_{2\pm}(v'')$ be the $z_1$-coordinate of $(\widehat p \circ v'')^{-1}(w_\pm)$ and define the map 
\begin{gather*}
	ev_2:\mathcal{M}_2\to S^1\times S^1,\\
	v''\mapsto (ev_{2-}(v''), ev_{2+}(v'')).
\end{gather*}
Note that we have switched the $+$ and $-$ compared to $ev_1$; this is because we want to identify $w_\pm$ for $v'$ with $w_\mp$ for $v''$.

We claim that, as we take $\epsilon\to 0$, $ev_2(\mathcal{M}_2)$ limits to the two pink triangles in Figure~\ref{fig: torus2}: $T_1$ with vertices $(0,\pi)$, $(\pi,\pi)$, $(\pi,2\pi)$ and $T_2$ with vertices $(\pi,0)$, $(\pi,\pi)$, $(2\pi,\pi)$. Let $\mathcal{G}_1$ (resp.\ $\mathcal{G}_2$) be the $2$-dimensional family of maps $v''$ satisfying (B1)--(B3) and (B4$'$) with image $T_1$ (resp.\ $T_2$). Figure~\ref{fig: modelcalc3} describes maps $v''=(v''_1,v''_2)$ that are on or close to $\bdry \mathcal{G}_1$;  the left-hand side depicts the image of $\widehat p\circ v''$ in $\C$ and the right-hand side depicts the images of the components $v''_1$ and $v''_2$. The red (resp.\ blue) dots on the left-hand side are $w_\pm$ (resp.\ $w_0$) and the red dots on the right-hand side with labels $w_\pm$ are the preimages of $w_\pm$ (resp.\ $w_0$). 
The maps are denoted $\phi_1,\dots, \phi_4$ from top to bottom.  The first map $\phi_1$ is $z\mapsto (z,1)$. The third map $\phi_3$ is close to $z\mapsto (g(z),g(z))$, where $g$ is a conformal map from $D^2$ to the half-disk.  Maps $\phi_2,\phi_3, \phi_4$ are of the form $(\phi_3)_\theta$, where the subscript $\theta$ indicates $\phi_3$ has been modified by rotating the first component by $\theta$ and the second component by $-\theta$.  There is a bijection
$$\mathcal{G}_1\stackrel\sim\to \mathcal{G}_2, \quad v''=(v''_1, v''_2)\mapsto (-v''_2,-v''_1).$$

Let $s(v'')$ be the length of the slit $\subset [-1,0]$ of the image $\widehat p\circ v''$.  We partition
$$\mathcal{M}_2=\sqcup_{s\in [0,1)}\mathcal{M}_2^s$$ 
by the slit length $s$ and note that $ev_2(\mathcal{M}_2^s)$, in the limit $\epsilon\to 0$, consists of two line segments of slope $1$: a segment from $\theta_2=\pi$ to $\theta_1=\pi$ in $T_1$ and a segment from $\theta_2=0$ to $\theta_1=2\pi$ in $T_2$.  Here the set of $(\phi_3)_\theta$ is half of $\mathcal{M}_2^s$ corresponding to $T_1$, for $s$ close to $1$. Since  $ev_2(\phi_1),ev_2(\phi_2),ev_2(\phi_4)$ are close to the vertices of $T_1$ and $s(\phi_1)\approx 0$ and $s(\phi_i)\approx 1$ for $i=2,3,4$, the claim follows.

\begin{figure}[ht]
	\begin{overpic}[scale=.85]{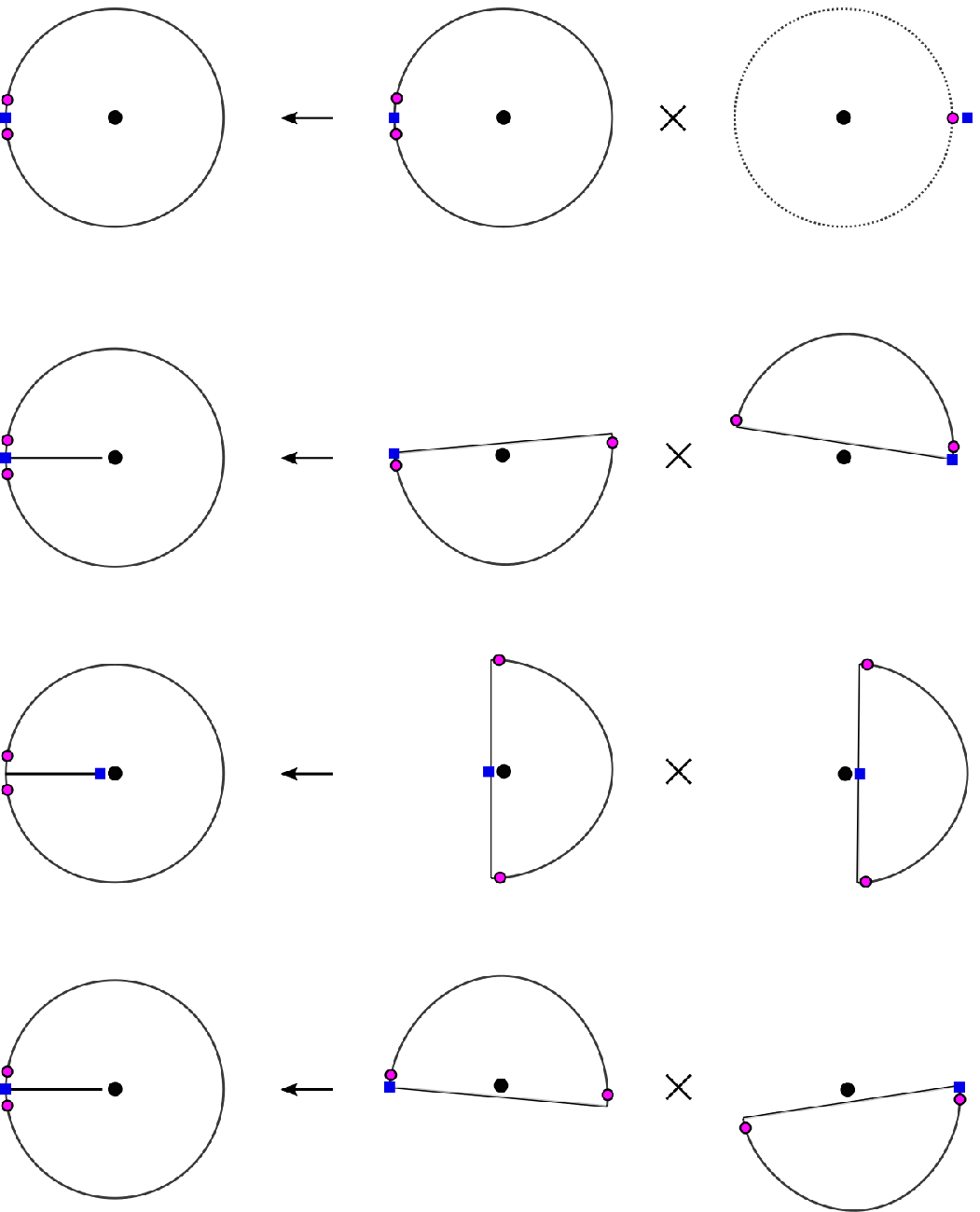}
		\put(-10,10){$\phi_4$}
		\put(-10,36){$\phi_3$}
		\put(-10,62.2){$\phi_2$}
		\put(-10,90.35){$\phi_1$}
		
		\put(23.2,87.7){\tiny $z_1z_2$}
		\put(41,79){\tiny $z_1$}
		\put(69.2,79){\tiny $z_2$}
		\put(-3,87.5){\tiny $w_+$}
		\put(-3,93){\tiny $w_-$}
		\put(-3.3,90.2){\tiny $w_0$}
		\put(-3.3,63){\tiny $w_0$}
		\put(-3.3,10){\tiny $w_0$}
		\put(6.7,37.5){\tiny $w_0$}
		\put(33.7,88.5){\tiny $w_+$}
		\put(33.7,92){\tiny $w_-$}
		\put(75,91){\tiny $w_+$}
		\put(80,91.5){\tiny $w_-$}
		
		\put(51.8,64.2){\tiny $w_-$}
		\put(62.5,66){\tiny $w_-$}
		\put(37.2,46){\tiny $w_-$}
		
		\put(34,61){\tiny $w_+$}
		\put(80,64){\tiny $w_+$}
		\put(37.2,28){\tiny $w_+$}
		\put(67.7,28){\tiny $w_+$}
		\put(67.7, 45.6){\tiny $w_-$}
		\put(33.5,11){\tiny $w_-$}
		\put(81,9){\tiny $w_-$}
		\put(46.2,10){\tiny $w_+$}
		\put(63,6){\tiny $w_+$}
	\end{overpic}
	\caption{}
	\label{fig: modelcalc3}
\end{figure}

\s\n
{\em Step 6.}
We finally show that $\#\mathcal{M}^{\chi=0, {\bf w}, \flat, 0LR}_{J^\Diamond}(\bs\Xi, \bs\Theta)=1$ mod $2$, where ${\bf w}$ satisfies (R1$'$) and (R2$'$) and ${\frak w}_i$ is close to $0$.

In view of Claim~\ref{claim A}, the holomorphic map $v_\infty$ from Step 4 can viewed as the gluing of 
\begin{equation}\label{eqn: gluing}
(v'_\infty:F'_\infty\to \wt W^\wedge) \in \mathcal{M}^\theta_1 \quad \mbox{ and } \quad (v''_\infty: F''_\infty \to \wt W^\wedge) \in \mathcal{M}_2,
\end{equation}
where $v'_\infty$ where corresponds to the left-hand side $\mathcal{R}_1$ and $v''_\infty$ corresponds to the right-hand side $\mathcal{R}_2\cup \mathcal{R}_3$.
\begin{figure}[ht]
	\begin{overpic}[scale=1]{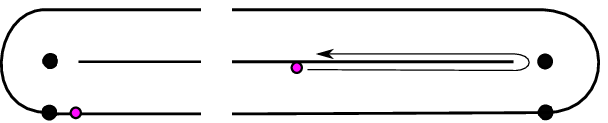}
		\put(6,-3.8){\tiny $\Theta_1$}
		\put(6,12.7){\tiny $\Theta_2$}
		\put(92.6,-2.8){\tiny $\Xi_1$}
		\put(91.5,11.7){\tiny $\Xi_2$}
		\put(13.5,-2.8){\tiny $w_1$}
		\put(47,5.8){\tiny $w_2$}
		\put(30,4.7){\tiny $c_{12}$}
		\put(30,13.5){\tiny $c_{21}$}
		\put(38,4.7){\tiny $c_{12}^*$}
		\put(38,13.5){\tiny $c_{21}^*$}
	\end{overpic}
	\caption{}
	\label{fig: stretching3}
\end{figure}
{\em We abuse notation and will not distinguish between $v'_\infty, v''_\infty$ and the restrictions of $v_\infty$ to $\mathcal{R}_1$ and $\mathcal{R}_2\cup\mathcal{R}_3$.} Under this identification,
\be
\item[(i)] the chords $c_{12}$ and $c_{21}$ (see Figure~\ref{fig: stretching3}) correspond to $w_-$ and $w_+$ for $\mathcal{M}_1^\theta$ and the chords $c_{12}^*$ and $c_{21}^*$ correspond to $w_+$ and $w_-$ for $\mathcal{M}_2$;
\item[(ii)] $v_\infty$ passing through $w_1$ corresponds to the constraint (B4) with $w_1=(e^{i\theta}, e^{-i\theta})$ for $v'_\infty$ and $v_\infty\cup \delta_+\cup \delta_-$ passing through $w_2$ corresponds to the constraint (B4$'$) for $v''_\infty$.
\ee
We further assume ${\bf w}$ has been chosen so that the constraint of passing through ${\bf w}$ is transformed to (B4) and (B4$'$) with $\theta=-\tfrac{\pi}{2} +\tilde \epsilon$. 

We now consider the matching condition $ev_1(v'_\infty)= ev_2(v''_\infty)$ for the gluing of $v'_\infty$ and $v''_\infty$. Let $\mathcal{I}$ be the intersection of the blue line and the pink region in Figure~\ref{fig: torus2}.  In view of the location of $\mathcal{I}$, we have the following:
\be
\item[(X)] The slits for $v'_\infty$ and $v''_\infty$ must be long, i.e., end near $\Theta_2$ and $\Xi_2$.  
\item[(Y)] The location of ${\frak b}_2^\infty$ is as depicted in Cases (A) and (B).
\ee
We will explain (Y): Since $q(w_1)$ must be close to $q(\Theta_1)$ and $F^{(i)}$ must have an involution taking $q(w_1)$ to $q(w_2)$ and $q(\Theta_1)$ to $q(\Theta_2)$, it follows that $q(w_2)$ is close to $q(\Theta_2)$ and in the limit $q(w_2)= q(p_-)=q(\Theta_2)$. The intersection $\mathcal{I}$ indicates that $v''_\infty$ intersects $w_2$ along the upper branch of the slit. Let $\mathcal{M}_{\mathcal{I}}$ be the set of curves $v_\infty$ that are obtained by gluing $(v'_\infty,v''_\infty)$.

At this point the possible gluings of $v'_\infty$ and $v''_\infty$ are parametrized by $\mathcal{I}$, but if we fix ${\frak b}_2^\infty$, corresponding to fixing a circle $\{\theta_1-\theta_2=\pi+\delta\}\subset S^1\times S^1$ for sufficiently small $\delta>0$, then the endpoints of the slit are uniquely determined.  Once ${\frak b}_2^\infty$ (i.e., where $q(\Theta_2)$ is mapped) and the slit are determined, the position of $q(\Xi_2)=q(p_+)$ (i.e., ${\frak b}_3^\infty$ in Case (A) or ${\frak b}_1^\infty$ in Case (B), as appropriate) is uniquely determined using Claim~\ref{claim B} and the involution of $F^\infty$.  In summary, fixing ${\frak b}_2^\infty$ uniquely determines all the other ${\frak b}_j^\infty$. 

Finally we consider $v^{(i)}$ for $i\gg 0$.  Taking $w_1$ to be sufficiently close to $\Theta_1$ forces $w_2$ to be on $C_-$ in Figure~\ref{fig: slit}.  The key observation is the following:
\be
\item[(*)]  the distance between $q(w_2)$ and $q(\Theta_2)$ decreases monotonically as ${\frak b}_2^{(i)}$ moves from left to right.
\ee  
The unique solution $v^{(i)}$ with an involution 
\begin{equation}\label{eqn: involution}
q(w_1)\mapsto q(w_2), \quad q(\Theta_1) \mapsto q(\Theta_2),  \quad q(\Xi_1)\mapsto q(\Xi_2)
\end{equation} 
is obtained by taking an open set of holomorphic curves $v^{(i)}$ that are close to $v_\infty \cup \delta_+\cup \delta_-$, where $v_\infty\in \mathcal{M}_{\mathcal{I}}$, and picking the unique one satisfying \eqref{eqn: involution} from this set using (*).

This completes the proof of Theorem~\ref{thm: count}.
\end{proof}

\subsection{Proof of invariance under Markov stabilizations} \label{subsection: stabilization invariance}

A Markov stabilization is given as follows:  Let $\sigma$ be an $\kappa$-strand braid which intersects $D$ along ${\bf z}=\{z_1,\dots,z_\kappa\}$. We view $\sigma$ as an element of $\op{Diff}^+(D,\bdry D,{\bf z})$ which additionally restricts to the identity on a neighborhood $N(\gamma_0)\subset D$ of a short arc $\gamma_0$ from a point $z_0$ to $\bdry D$.  Given an arc $c$ from $z_0$ to some $z_i$, $i>0$, which is disjoint from the other $z_j$, let $\sigma_c$ be the positive half twist along $c$.  Then a {\em positive} (resp.\ {\em negative}) {\em Markov stabilization} is the $(\kappa+1)$-strand braid given by $\sigma\circ \sigma_c$ (resp.\ $\sigma\circ \sigma_c^{-1}$). 

Let $\bs\gamma:=\{\gamma_1,\dots,\gamma_\kappa\}$ be a basis of half arcs, where $\gamma_i$ connects $z_i$ to $\bdry D$.  By the handleslide invariance, we may assume that $c$ connects from $z_0$ to $z_1$ and does not intersect any $\gamma_i$, $i>0$, in its interior. Then let $\bs\gamma':=\{\gamma_0,\dots,\gamma_\kappa\}$. Let 
$$p: W\to D-N(\gamma_0) \quad \mbox{and} \quad p': W'\to D$$  
be the standard Lefschetz fibrations with critical values ${\bf z}=\{z_1,\dots, z_\kappa\}$ and ${\bf z}'=\{z_0,\dots,z_\kappa\}$ and regular fiber $A=S^1\times[-1,1]$, and such that $p'|_{W}=p$. Let ${\bs a}=\{a_1,\dots,a_\kappa\}$ be the Lagrangian thimbles over $\bs\gamma$ and let ${\bs a}'=\{a_0,\dots, a_\kappa\}$ be the Lagrangian thimbles over $\bs\gamma'$.  If $h_\sigma\in \op{Symp}(W,\bdry W)$ descends to $\sigma$, then let $h'_\sigma\in \op{Symp}(W',\bdry W')$ be its extension to $W'$ by the identity. Finally let $\tau_c\in \op{Symp}(W',\bdry W')$ be the Dehn twist along the Lagrangian sphere over $c$. 

\subsubsection{Model calculation} \label{subsubsection: model calculation}

{\em Any notation introduced here is limited to this subsection.}
Consider the product fibration $\widehat p:\C\times T^*S^1\to \C$ and Lagrangians $a_i= \{x=i\}\times S^1$, $i=1,2$, and $b_j=\{y=j\}\times S^1$, $j=1,2$, where $S^1$ is the zero section of $T^*S^1$. We write ${\bs a}=\{a_1,a_2\}$ and ${\bs b}=\{b_1,b_2\}$. Let $x_{ij}=(i,j)\in \C$ and let $\check x_{ij}, \hat x_{ij}\in CF(b_j,a_i)$ be the bottom and top generators of the clean intersection $x_{ij} \times S^1$.  We can alternatively take ${\bs a}'=\{a'_1,a'_2\}$ and ${\bs b}'=\{b'_1,b'_2\}$, where $a_i'$ is $\{x=i\}$ (resp.\ $b_j'$ is $\{y=j\}$) times a Hamiltonian perturbation of $S^1$, where all the perturbations intersect transversely; see the right-hand side of Figure~\ref{fig: modelcalc2}.

\begin{figure}[ht]
	\begin{overpic}[scale=1]{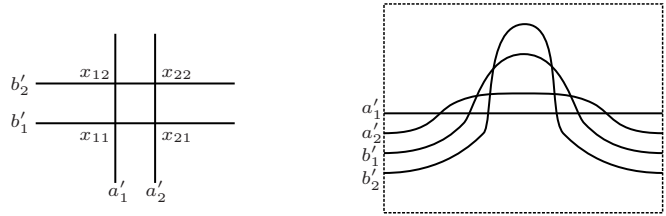}
		\put(11.5,3){\tiny $a'_1$}
		\put(17.5,3){\tiny $a'_2$}
		\put(-4,14){\tiny $b'_1$}
		\put(-4,20){\tiny $b'_2$}
		\put(7,22){\tiny $x_{12}$}
		\put(20,22){\tiny $x_{22}$}
		\put(7,12){\tiny $x_{11}$}
		\put(20,12){\tiny $x_{21}$}
		
		\put(51.8,16.2){\tiny $a_1'$}
		\put(51.8,12.3){\tiny $a_2'$}
		\put(51.8,8.65){\tiny $b_1'$}
		\put(51.8,4.95){\tiny $b_2'$}
	\end{overpic}
	\caption{The base $\C$ on the left and the fiber $T^*S^1$ on the right (the sides are identified).}
	\label{fig: modelcalc2}
\end{figure}

\begin{lemma}\label{lemma: model calc}
	$\widehat{CF}({\bs b}', {\bs a}')$ is generated by $8$ generators $\{x_{12}^\dagger,x_{21}^\dagger\}$ and $\{x_{11}^\dagger,x_{22}^\dagger\}$, where $\dagger$ may be a hat or check, and the differentials are given by $\hbar$ times the arrows given in Figure~\ref{fig: differential}.
	
\end{lemma}

\begin{figure}[ht]
	\begin{overpic}[scale=1.4]{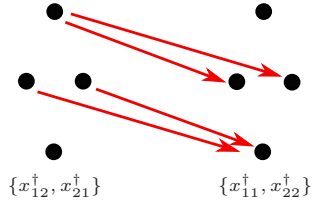}
	\put(-3.5,-10){\tiny $\{x_{12}^\dagger,x_{21}^\dagger\}$}
	\put(70,-10){\tiny $\{x_{11}^\dagger,x_{22}^\dagger\}$}
	\end{overpic}
\vskip.15in
	\caption{Description of the differentials of $\widehat{CF}({\bs b}', {\bs a}')$. The generators in the top row have $2$ checks, those in the middle row have $1$ check, and those in the bottom row have no checks.}
	\label{fig: differential}
\end{figure}

\begin{proof}
Let $u: \dot F\to \R\times [0,1]\times (\C\times T^*S^1)$ be a holomorphic map with $\{\check x_{12},\check x_{21}\}$ at the positive end and $\{\hat x_{11},\check x_{22}\}$ at the negative end. Its projection to $\C$ is a degree $1$ map over $[1,2]\times[1,2]$; this determines the complex structure of the $4$-punctured disk $\dot F=\dot F_0$. 

Next we consider holomorphic maps $w:\dot F\to T^*S^1$ which could be projections of $u$ to $T^*S^1$.  The two possible domains $A,B$ that $w$ can map to are shaded in Figure~\ref{fig: curves}.  The domain $A$ on the left (resp.\ $B$ on the right) has a $\tfrac{3\pi}{2}$ corner at $\hat x_{11}$ (resp.\ $\check x_{22}$) and there may be slits to the right or pointing up (resp.\ to the right or pointing up).  We will refer to the four types of slits by $AR$, $AU$, $BR$, $BU$. As the slit $AR$ goes all the way to the right (i.e., until it hits $a_2'$), the punctures $q(\hat x_{11})$ and $q(\check x_{21})$ approach one another; when $AU$ goes all the way up, we can view it as continuing to $BR$ going all the way to the right; finally, as $BU$ goes all the way up, $q(\check x_{21})$ and $q(\check x_{22})$ approach one another.  Hence the algebraic count of $w$ with the given domain $\dot F=\dot F_0$ is one.

\begin{figure}[ht]
	\begin{overpic}[scale=1.35]{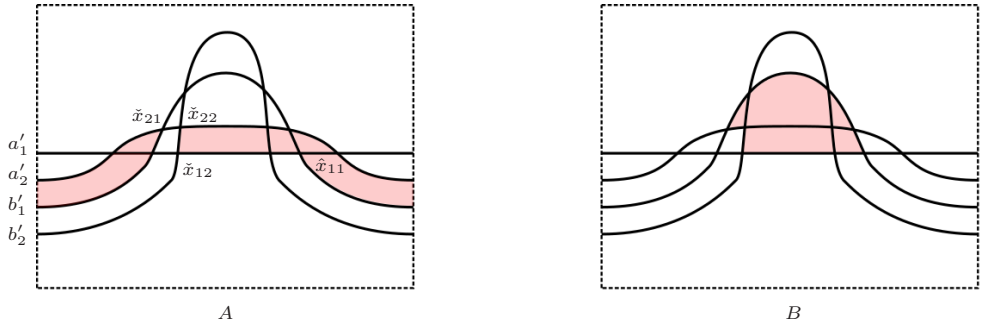}
		\put(15.5,12.3){\tiny $\check x_{12}$}
		\put(16,18.2){\tiny $\check x_{22}$}
		\put(10.15,18){\tiny $\check x_{21}$}
		\put(29.5,12.7){\tiny $\hat x_{11}$}
		\put(19.2,-3){\tiny $A$}
		\put(79.2,-3){\tiny $B$}
		\put(-3,15){\tiny $a_1'$}
		\put(-3,11.8){\tiny $a_2'$}
		\put(-3,8.6){\tiny $b_1'$}
		\put(-3,5.4){\tiny $b_2'$}
	\end{overpic}
\s
	\caption{The two domains $A$ and $B$ representing the possible closures of $w(\dot F)$.}
	\label{fig: curves}
\end{figure}

The determination of the other arrows of Figure~\ref{fig: differential} is analogous.
\end{proof}

\subsubsection{Proof of stabilization invariance}

In this subsection we prove the invariance under Markov stabilization.  We will use the open book interpretation of the Heegaard Floer groups.

\begin{thm} \label{thm: stabilization invariance}
$\widehat{CF}(W,h_\sigma({\bs a}),{\bs a}))$ and $\widehat{CF}(W',h'_\sigma \circ \tau_c({\bs a}'), {\bs a}')$ are isomorphic cochain complexes for specific choices of almost complex structures and $h_\sigma({\bs a})$ and $h'_\sigma \circ \tau_c({\bs a}')$ after a Hamiltonian isotopy.
\end{thm}

\begin{proof}
We will treat the positive stabilization case; the negative stabilization case is analogous.

The generators of $\widehat{CF}(W,h_\sigma({\bs a}),{\bs a}))$ come in two types of $\kappa$-tuples: $\{x_1\}\cup \bf y'$ and $\bf y$, where $\{x_1,\dots,x_\kappa\}$ is the contact class and $\bf y$ does not contain $x_1$.  The generators of $\widehat{CF}(W',h'_\sigma \circ \tau_c({\bs a}'), {\bs a}')$ are in bijection with the generators of $\widehat{CF}(W,h_\sigma({\bs a}),{\bs a}))$ and have the form $\{x_0,x_1\}\cup \bf y'$ and $\{\Theta_{01}\}\cup \bf y$, where $\{x_0,\dots, x_\kappa\}$ is the contact class and $\Theta_{01}$ is the unique intersection point between $a_0$ and $h'_\sigma\circ \tau_c(a_1)=\tau_c(a_1)$.   Refer to Figure~\ref{fig: stabilization}.  

Let $\gamma_1=\{\op{Re} z_1\}\times[-1,0]\subset D$. Choose $\epsilon>0$ small.  We normalize $h_\sigma (a_1)$ and $h_\sigma'\circ \tau_c(a_0)$ within their Hamiltonian isotopy classes (rel boundary) such that their projections $\sigma (\gamma_1)$ and $\sigma\circ \sigma_c(\gamma_0)$ can be written as concatenations $\zeta_1\cup \zeta_2\cup \zeta_3$ and $\zeta_1'\cup \zeta_2\cup \zeta_3$ of arcs, where:
\begin{itemize}
	\item $\zeta_1$ starts at $p'(x_1)$ and ends at $(\op{Re} z_1-\epsilon, -1/3)$; 
	\item $\op{Re} z_1-\epsilon < \op{Re}(\op{int}\zeta_1)< \op{Re} z_1$;
	\item $\zeta'_1$ starts at $p'(x_0)$ and ends at $(\op{Re} z_1-\epsilon, -1/3)$; and
	\item $\zeta_2=\{\op{Re} z_1-\epsilon\}\times[-2/3,-1/3]$.
\end{itemize}

We claim that the linear isomorphism 
\begin{gather} \label{eqn: another Phi}
	\Phi_s: \widehat{CF}(W,h_\sigma({\bs a}),{\bs a}))\to \widehat{CF}(W',h'_\sigma \circ \tau_c({\bs a}'), {\bs a}'),\\
\nonumber \{x_1\}\cup {\bf y}'\mapsto \{x_0,x_1\}\cup {\bf y'}, \quad {\bf y}\mapsto \{\Theta_{01}\}\cup \bf y,
\end{gather}
commutes with the differentials for $\epsilon>0$ small. Here we are assuming that $D$ has the standard complex structure. 

We will use the notation $u':\dot F'\to \R\times[0,1]\times \widehat W'$ for a holomorphic map that is counted in the differential of $\widehat{CF}(W',h'_\sigma \circ \tau_c({\bs a}'), {\bs a}')$ and $u$  for a holomorphic map $\dot F\to \R\times[0,1]\times \widehat W$ that is counted in the differential of $\widehat{CF}(W,h_\sigma({\bs a}),{\bs a}))$.

If a curve $u'$ goes from $\{x_0,x_1\}\cup \bf y_1'$ to $\{x_0,x_1\}\cup \bf y_2'$, then, apart from the trivial strip from $x_0$ to itself, it projects to the region $D-N(\gamma_0)$ and hence is in bijection with a curve $u$ that goes from $\{x_1\}\cup \bf y_1'$ to $\{x_1\}\cup \bf y_2'$. Similarly, a curve $u'$ from $\{\Theta_{01}\}\cup \bf y_1$ to $\{\Theta_{01}\}\cup \bf y_2$ is in bijection with a curve $u$ from $\bf y_1$ to $\bf y_2$. There are no curves from $\{x_0,x_1\}\cup \bf y_1'$ to $\{\Theta_{01}\}\cup \bf y_2$ and likewise no curves from $\{x_1\}\cup \bf y_1'$ to $\bf y_2$. 

It remains to identify curves $u'$ from $\{\Theta_{01}\}\cup \bf y_1$ to $\{x_0,x_1\}\cup \bf y_2'$ with curves $u$ from $\bf y_1$ to $\{x_0\}\cup \bf y_2'$. 
By our normalization of $\sigma (\gamma_1)$ and $\sigma\circ \sigma_c(\gamma_0)$ with $\epsilon>0$ small, if $\op{ind}(u)=1$, then there exists $u'$ with $\op{ind}=1$ which is obtained by gluing $u$ (where the negative end $x_1$ is now viewed as the thin neck between $h_\sigma'\circ \tau_c(a_0)$ and $a_1$ that projects to $[\op{Re} z_1-\epsilon,\op{Re}z_1]\times[-2/3,-1/3]$) and a holomorphic $4$-punctured disk $u'': \dot F''\to \R\times\widehat W'$ that projects to the shaded region in Figure~\ref{fig: stabilization}, i.e, a curve of the type calculated in Lemma~\ref{lemma: model calc}.  The algebraic count of such a curve $u''$ is $1$.  Conversely a curve $u'$ with $\op{ind}(u')=1$ is obtained by gluing $u$ and $u''$.  
\end{proof}

\begin{figure}[ht]
	\begin{overpic}[scale=1]{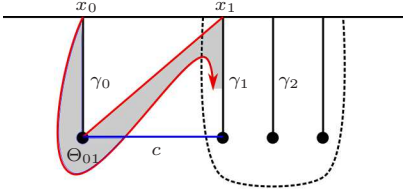}
		\put(21.7,26){\tiny $\gamma_0$}
		\put(56.5,26){\tiny $\gamma_1$}
		\put(69,26){\tiny $\gamma_2$}
		\put(18,45){\tiny $x_0$}
		\put(53,45){\tiny $x_1$}
		\put(15.6,7){\tiny $\Theta_{01}$}
		\put(37,8.5){\tiny $c$}
	\end{overpic}
	\caption{The region cut off by the dotted line and containing $\gamma_1,\dots\gamma_\kappa$ is $D-N(\gamma_0)$. The red half-arcs are $\sigma\circ \sigma_c(\gamma_0)$ and $\sigma\circ \sigma_c(\gamma_1)$.}
	\label{fig: stabilization}
\end{figure}

\subsection{Contact class}

In this section we prove Theorem~\ref{thm: invariance of contact class in Kh}.  

\begin{proof}[Proof of Theorem~\ref{thm: invariance of contact class in Kh}]
We first prove the invariance under handleslides.  Consider the cochain map $\Phi$ given by Equation~\eqref{eqn: chain map Phi}. 
\begin{figure}[ht]
	\begin{overpic}[scale=1]{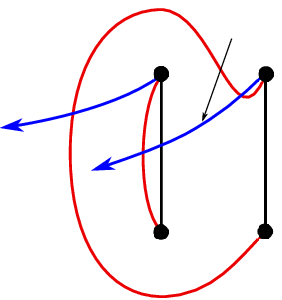}
		\put(92.6,50){\tiny $\wt \gamma_1$}
		\put(57.3,38){\tiny $\wt \gamma_2$}
		\put(39.7,32){\tiny $\wt \gamma_2'$}
		\put(16.7,25){\tiny $\wt \gamma_1'$}
		\put(80,91){\tiny $\wt \delta_1$}
		\put(-7,64.8){\tiny $\wt \delta_2$}
	\end{overpic}
	\caption{}
	\label{fig: An-handleslide-contact}
\end{figure}
We take $\wt h_\sigma(\wt a_i)$ to be the Lagrangian over the blue arc $\wt \delta_i$ in Figure~\ref{fig: An-handleslide-contact}. The key is to position $\wt \delta_i$ as shown so it is ``locally to the right" of $\wt \gamma_i$ and $\wt \gamma_i'$ as viewed from $z_{\kappa+i}$: by this we mean that on a small neighborhood $N(z_{\kappa+i})$ of $z_{\kappa+i}$,
\be
\item the arcs $\wt \delta_i$, $\wt\gamma_i'$, $\wt\gamma_i$ restrict to rays emanating from $z_{\kappa+i}$;
\item they are contained in a thin circular sector ${\frak S}_{\kappa+i}$ with center $z_{\kappa+i}$;
\item inside ${\frak S}_{\kappa+i}$ the arcs are $\wt\delta_i$, $\wt\gamma_i'$, $\wt\gamma_i$ in counterclockwise order.
\ee 
Strictly speaking, we then further apply a $C^\infty$-small perturbation to $\wt h_\sigma(\wt a_i)$ so that it is slightly pushed off of the intersection point $\Xi_i$ between $\wt a_i$ and $\wt a_i'$; this can be done in the local model $\widehat p: \C^2_{z_1,z_2}\to \C_\zeta$, $(z_1,z_2)\mapsto z_1z_2$, from Equation~\eqref{eqn: model for C}. The contact class ${\bf x}=\{x_1,\dots, x_\kappa\}\in \widehat{CF}(\wt h_\sigma(\wt {\bs a}), \wt {\bs a})$ is the $\kappa$-tuple of points where $x_i$ is over $z_{\kappa+i}$.  Let ${\bf x}'=\{x_1',\dots, x_\kappa'\}\in \widehat{CF}(\wt h_\sigma(\wt {\bs a}), \wt {\bs a}')$ be the pushoff of the $\kappa$-tuple of points over $z_{\kappa+1},\dots, z_{2\kappa}$. 

We claim that $\Phi({\bf x})={\bf x}'$:  Using the local model one can verify that every curve $u$ that is counted in $\mu_2(\bs\Xi,{\bf x})$ must have a component which is a small triangle $\Xi_i,x_i, x'_i$ for $i\not=2$.  Once $\wt a_i$, $\wt a_i'$, and $\wt h_\sigma (\wt a_i)$ are used up for $i\not= 2$, the only possible holomorphic triangle involving $\Xi_2$ and $x_2$ is the small triangle $\Xi_2,x_2,x_2'$.  Hence $\Phi({\bf x})={\bf x}'$. 

The analogous map 
$$\Phi': \widehat {CF}(\wt h_\sigma (\wt {\bs a}), \wt{\bs a}')\to \widehat{CF}(\wt h_\sigma (\wt {\bs a}'), \wt {\bs a}')$$
similarly takes ${\bf x}'$ to the contact class.  This proves the invariance under handleslides.

The invariance under positive Markov stabilizations is immediate from the definition of the cochain map $\Phi_s$ given by Equation~\eqref{eqn: another Phi}.
\end{proof}

\subsection{Relationship with symplectic Khovanov homology} \label{subsection: spectral sequence}

We will work over the ring $\F[\mathcal{A}]\llbracket\hbar,\hbar^{-1}]\llbracket U^{-1}\rrbracket$. Consider the following filtration for $CKh^\sharp (\widehat \sigma)\otimes \llbracket U^{-1}\rrbracket$:
$$\mathcal{F}:= (\dots\subset  \mathcal{F}_{-2}\subset \mathcal{F}_{-1}\subset \mathcal{F}_0), \quad \mathcal{F}_i=U^i\cdot CKh^\sharp (\widehat \sigma)\otimes \llbracket U^{-1}\rrbracket,$$
where the tensor product is over $\F$.

Given $u: \dot F\to \R\times[0,1]\times \wt W$ in $\mathcal{M}_{J^\Diamond}^{\op{ind}=1,A,\chi}({\bf y},{\bf y}')$ and its projections
$v_1$ to $\R\times[0,1]$ and $v_2$ to the base $\wt D$ of the Lefschetz fibration $\wt p: \wt W \to \wt D\subset \C$, we consider the map
$$(v_1,v_2): \dot F\to \R\times[0,1]\times \wt D.$$
Note that the target is $4$-dimensional. Although we will not go into detail here, it is possible to choose a regular almost complex structure $J^\Diamond$ on $\R\times[0,1]\times \wt W$, a regular almost complex structure $J_{\R\times[0,1]\times \wt D}$ on $\R\times[0,1]\times \wt D$ for curves of the form $(v_1,v_2)$, and the standard complex structure $J_{\R\times[0,1]}$ on $\R\times[0,1]$ such that the projections
$$(\R\times[0,1]\times \wt W,J^\Diamond) \to (\R\times[0,1]\times \wt D, J_{\R\times[0,1]\times \wt D})\to (\R\times[0,1],J_{\R\times[0,1]})$$
are holomorphic and the fibers of the projections are holomorphic.

Let $W(u)$ be the total weight of the singularities of $(v_1,v_2)$, including self-intersections. By the positivity of intersections for pseudoholomorphic curves in dimension $4$, this is a nonnegative number which 
is equal to the difference between the Heegaard Floer index of \cite{CGH1} and the Fredholm index in the usual Heegaard Floer homology. Writing $d_U$ for the differential of $CKh^\sharp(\widehat \sigma)\otimes \llbracket U^{-1}\rrbracket$, the contribution of $u$ to $d_U(U^a{\bf y})$ is $\hbar^{\kappa-\chi}e^A U^{a-W(u)}{\bf y}'$ (modulo $2$) and 
$$(CKh^\sharp (\widehat \sigma)\otimes \llbracket U^{-1}\rrbracket,d_U,\mathcal{F})$$ 
is a filtered cochain complex.  Reinterpreting Manolescu~\cite{Ma}, the symplectic Khovanov condition is that the holomorphic map $(v_1,v_2)$ be embedded, i.e., that we only count curves with $W(u)=0$.

The following lemma is immediate from the above discussion.

\begin{lemma} \label{lemma: spectral sequence}
	With $\F[\mathcal{A}]\llbracket\hbar,\hbar^{-1}]\llbracket U^{-1}\rrbracket$-coefficients, the $E_1$ term of the spectral sequence of $(CKh^\sharp (\widehat \sigma)\otimes \llbracket U^{-1}\rrbracket,d_U,\mathcal{F})$ is the symplectic Khovanov homology $Kh_{\op{symp}}(\widehat\sigma)$ tensored with $\llbracket U^{-1}\rrbracket$.
\end{lemma}

The main issue that we face is the convergence of the spectral sequence.  The following is the key question:

\begin{q}
Is there a bound on $W(u)$ for $u\in \mathcal{M}_{J^\Diamond}^{\op{ind}=1,A,\chi}({\bf y},{\bf y}')$?  This is closely related to whether there is a bound on $\chi$. 
\end{q}

What we can say for the moment is the following:

\begin{thm}
There is a filtered cochain map 
$$\widehat{CF}(\wt W,\wt h_\sigma(\wt {\bs a}),\wt {\bs a})\otimes \llbracket U^{-1}\rrbracket \to \widehat{CF}(\wt W,\wt h_\sigma(\wt {\bs a}'),\wt {\bs a}')\otimes  \llbracket U^{-1}\rrbracket$$
for a handleslide (cf.\ Theorem~\ref{thm: quasi-isomorphism}) which is a quasi-isomorphism of cochain complexes. There is also a filtered cochain map for a Markov stabilization (cf.\ Theorem~\ref{thm: stabilization invariance}) which is a quasi-isomorphism of cochain complexes. Hence the $E_k$ term of the spectral sequence of $(CKh^\sharp (\widehat \sigma)\otimes \llbracket U^{-1}\rrbracket,d_U,\mathcal{F})$ is a link invariant.
\end{thm}

\begin{proof}
The cochain map is an enhancement of 
$$\widehat{CF}(\wt W,\wt h_\sigma(\wt {\bs a}),\wt {\bs a})\to  \widehat{CF}(\wt W,\wt h_\sigma(\wt {\bs a}'),\wt {\bs a}')$$
from the proof of Theorem~\ref{thm: quasi-isomorphism} which keeps track of the powers of $U^{-1}$. Since the total singularity weight $W(u)$ of the appropriate projection of any curve $u$ counted in the maps $\Phi$ and $\Psi$ is always nonnegative, the cochain map is filtration-preserving.  The quasi-isomorphism as cochain complexes follows from observing that in the proof of Theorem~\ref{thm: count} the only nonzero (mod $2$) curve count is that of $\mathcal{M}^{\chi=0, {\bf w}, \flat, 0LR}_{J^\Diamond}(\bs\Xi, \bs\Theta)$ in Step 6, whose elements satisfy $W=0$. 

The case of Markov stabilizations is straightforward and is left to the reader.
\end{proof}

\section{The Baldwin-Plamenevskaya examples} \label{section: BP}

The goal of this section is to prove Theorem~\ref{thm: BP}.

\begin{proof}
	In view of Lemma~\ref{lemma: Lefschetz} it suffices to show that $\psi^\sharp(\widehat\sigma_{BP,3})=0$. We claim that
\begin{gather*}
	d (\{\hat y_1,\hat y_2,x_3\} + \{\check y_1, z_2, x_3\})= \hbar \{x_1,x_2,x_3\} \mbox{ mod } 2
\end{gather*}
in Figure~\ref{fig: sigma-BP}. Referring to Figure~\ref{fig: sigma-BP}, there is a single intersection point of ${\bs a}\cap h_{BP,3}(\bs a)$ over each $z_i$, denoted by $\check z_i$ or simply $z_i$; two intersection points (here we are doing Morse-Bott theory) over each of $r_i,s_i,t_i,y_i$, given by $\check r_i$ of lower degree, $\hat r_i$ of higher degree, etc.; and one intersection point over $x_i$, denoted by $\hat x_i$ or simply $x_i$.  The point $\hat x_i$ is a component of the contact class.
\begin{figure}[ht]
	\begin{overpic}[scale=1.2]{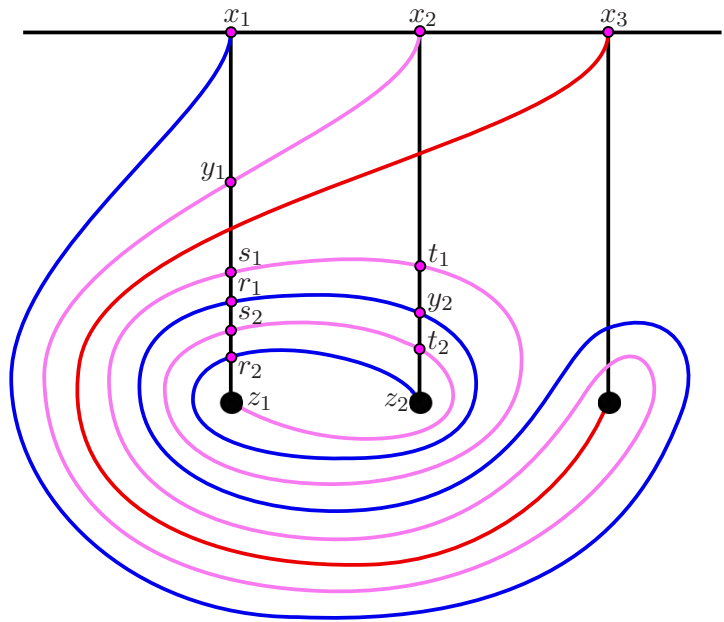}
		\put(30,84){$x_1$}
		\put(56,84){$x_2$}
		\put(82.8,84){$x_3$}
		\put(26.8,62.5){$y_1$}
		\put(58.55,44.1){$y_2$}
		\put(32,46.2){$r_1$}
		\put(32,50.4){$s_1$}
		\put(32,34.6){$r_2$}
		\put(32,42.25){$s_2$}
		\put(58.6,38){$t_2$}
		\put(58.7,50.3){$t_1$}
		\put(33.3,30.2){$z_1$}
		\put(52.5,30){$z_2$}
	\end{overpic}
	\caption{The braid $\sigma_{BP,3}$ acting on $\gamma_1,\gamma_2,\gamma_3$ on $D$.}
	\label{fig: sigma-BP}
\end{figure}
	
Since the only holomorphic curve that has $x_i$ at the positive end is a trivial strip, we may erase the arcs $\gamma_3$ and $\sigma_{BP,3}(\gamma_3)$ and omit $x_3$ from the notation; see Figure~\ref{fig: sigma-BP-alt}.
\begin{figure}[ht]
	\begin{overpic}[scale=0.5]{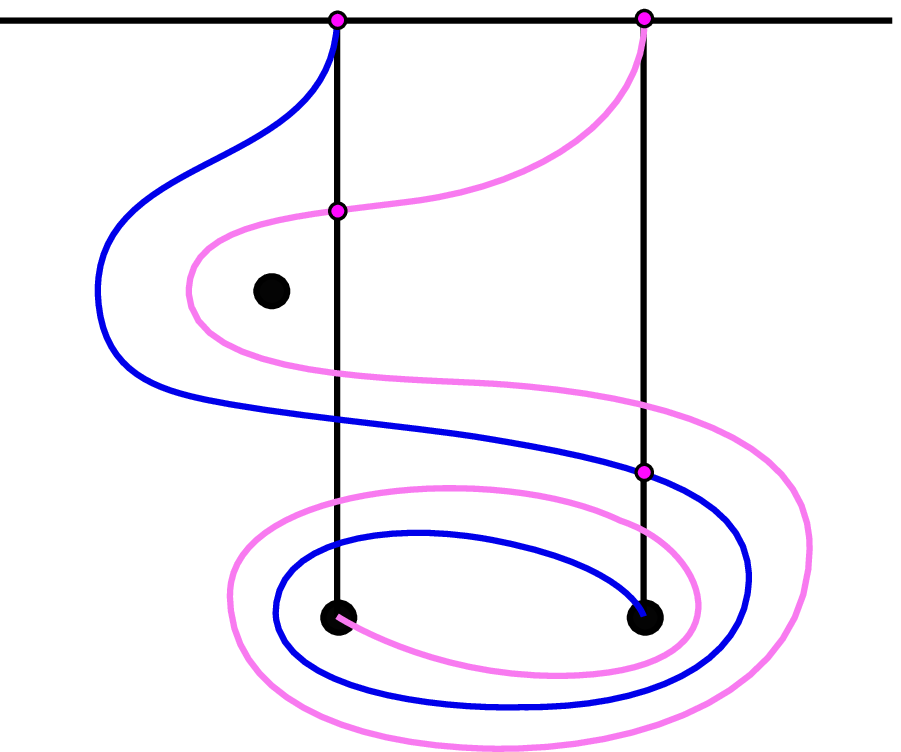}		
	\end{overpic}
	\caption{}
	\label{fig: sigma-BP-alt}
\end{figure}
Excluding trivial strips, all the curves $u:\dot F\to \R\times[0,1]\times \widehat W$ that we end up counting have a domain $\dot F$ which is a disk with $4$ boundary punctures and contributes the coefficient $\hbar$. 

\begin{rmk}
In fact, the proof of Theorem~\ref{thm: BP} shows that $\psi^\sharp (\widehat \sigma)=0$ for any $\kappa$-braid $\sigma$ for which there exist arcs $\gamma_1$ and $\gamma_2$ such that $\sigma(\gamma_1)$ and $\sigma(\gamma_2)$ are as given in Figure~\ref{fig: sigma-BP-alt} and there are no other critical points in regions bounded by subarcs of $\gamma_1,\gamma_2,\sigma(\gamma_1),\sigma(\gamma_2)$. 
\end{rmk}

\s\n
{\em First Calculation.}  The next several pages are devoted to showing that 
$$d \{\hat y_1,\hat y_2\} = \hbar (\{x_1,x_2\}+ \{\check r_2, x_2\}+ \{\check t_2,x_1\}) \mbox{ mod } 2.$$

All the loops in $D$ that are concatenations of an arc each of 
$$\gamma_1,\gamma_2,  \sigma_{BP,3}(\gamma_1), \sigma_{BP,3}(\gamma_2)$$ 
and that switch at $y_1, y_2$ are given by:
$$y_2\to x_2,t_1,t_2\to y_1\to r_1,r_2,x_1 \to y_2,$$
where $x_2,t_1,t_2$ means any one of them can be used.  The only loops that can possibly bound domains with nonnegative weights (denoted $D_1,D_1',D_2, D_3, D_4$) are:
\be
\item[(1)] $y_2\to x_2\to y_1\to r_1\to y_2$,
\item[(1')]  $y_2\to t_1\to y_1\to x_1\to y_2$,
\item[(2)] $y_2\to x_2\to y_1\to x_1\to y_2$,
\item[(3)] $y_2\to x_2\to y_1\to r_2\to y_2$,
\item[(4)] $y_2\to t_2\to y_1\to x_1\to y_2$.
\ee

\s\n
{\em Cases (1) and (1').} $D_1$ is a quadrilateral and an easy Maslov index calculation gives:
\be
\item[(i)] $\op{ind}(u)=0$ if $u$ has $\{\hat y_1,\hat y_2\}$ at the positive end and $\{\hat r_1,x_2\}$ at the negative end;
\item[(ii)] $\op{ind}(u)=-1$ if $u$ has $\{\hat y_1,\hat y_2\}$ at the positive end and $\{\check r_1,x_2\}$ at the negative end.
\ee	
Hence curves $u$ that project to $D_1$ do not have the right indices and are not counted. The $D_1'$ case is analogous.

\begin{figure}[ht]
	\begin{overpic}[scale=1]{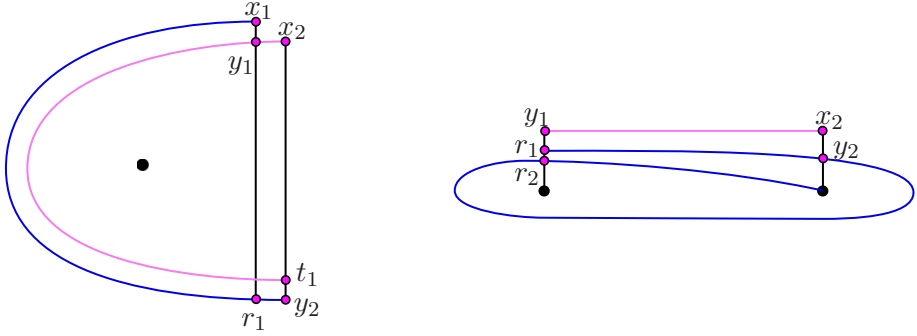}
		\put(26.3,32){$x_1$}
		\put(30,29.8){$x_2$}
		\put(24.5,26){$y_1$}
		\put(31.75,-.6){$y_2$}
		\put(32,2.7){$t_1$}
		\put(26,-2.2){$r_1$}
		\put(57,20.3){$y_1$}
		\put(89,20){$x_2$}
		\put(91,16.7){$y_2$}
		\put(56,16.8){$r_1$}
		\put(56,13.75){$r_2$}
	\end{overpic}
	\caption{The domains $D_2$ (left) and $D_3$ (right).}
	\label{fig: BC}
\end{figure}

\begin{figure}[ht]
	\begin{overpic}[scale=1.25]{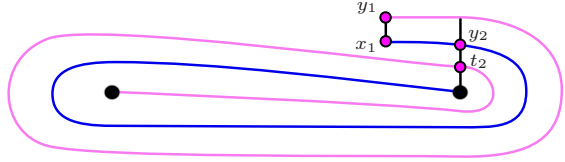}
		\put(62.4,20){\tiny $x_1$}
		\put(62.7,27){\tiny $y_1$}
		\put(83,17){\tiny $t_2$}
		\put(82.7,22){\tiny $y_2$}
	\end{overpic}
	\caption{The domain $D_4$.}
	\label{fig: D}
\end{figure}

\s\n
{\em Case (2).} Denoting the closure of $D_2$ by $\overline{D}_2$, we choose $D_2$ as in Figure~\ref{fig: BC} such that the intersections $\overline{D}_2\cap \gamma_1$ and $\overline{D}_2\cap\gamma_2$ are close and that $\overline{D}_2\cap \sigma_{BP,3}(\gamma_1)$ and $\overline{D}_2\cap \sigma_{BP,3}(\gamma_2)$ are close. 

\begin{claim}\label{claim: count}
The count of $u:\dot F\to \R\times[0,1]\times \widehat W$ from $\{\hat y_1,\hat y_2\}$ to $\{x_1,x_2\}$ modulo $\R$-translation is $1$ mod $2$ if its projection to $\C$ is $D_2$.
\end{claim} 

Note that such $u$ satisfies $\op{ind}(u)=1$ and that the index difference between $\{y_1^\dagger,y_2^\dagger\}$ and $\{x_1,x_2\}$ is $1$ plus the number of $\dagger=\check{\mbox{}}$.  The intersection points $\hat y_1, \hat y_2$ should be viewed as point constraints and $x_1,x_2$ as imposing no point constraints (and can be disregarded).

\begin{proof}[Proof of Claim~\ref{claim: count}]
There are two types of slits in $D_2$ that start at $y_1$: Type $D$ that goes straight down along $\gamma_1$ and Type $L$ that initially goes to the left along $\sigma_{BP,3}(\gamma_2)$.  As the slit $D$ goes all the way down so that it hits $r_1$, the punctures $q(y_2), q(x_2), q(y_1)$ can be viewed as approaching one another on the domain $\dot F$; as the slit $L$ goes all the way to $\gamma_2$, the punctures $q(y_1), q(x_1), q(y_2)$ approach one another on $\dot F$. Moreover, for Type $D$, $q(y_1)$ comes right after $q(y_2)$, traveling counterclockwise, whereas $q(y_1)$ comes right before $q(y_2)$ for Type $L$. These two degenerations can be viewed as the ends of a $1$-dimensional family of maps $\widehat p\circ v: \dot F\to \C$. 

Recall the Lefschetz fibration 
$$\widehat p: \C^2_{z_1,z_2}\to \C_\zeta,\quad (z_1,z_2)\mapsto z_1z_2,$$
and the Clifford torus $T=\{|z_1|=1\}\times \{|z_2|=1\}$ over $|\zeta|=1$. 
\begin{figure}[ht]
	\begin{overpic}[scale=1]{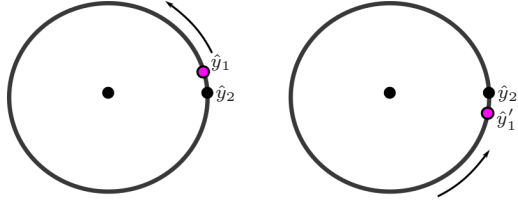}
		\put(42.5,20.3){\tiny $\hat y_2$}
		\put(41.5,26.7){\tiny $\hat y_1$}
		\put(100.3,20.3){\tiny $\hat y_2$}
		\put(100,15){\tiny $\hat y_1'$}
	\end{overpic}
	\caption{The base of the Lefschetz fibration $\widehat p$. The circles are $\widehat p(T)=\{|\zeta|=1\}$. The path $\widehat p \circ \eta$ goes counterclockwise from $e^{i\epsilon}$ to $e^{-i\epsilon}$, given by red dots.}
	\label{fig: moving-points}
\end{figure}
Let $\epsilon>0$ be small. 
Referring to Figure~\ref{fig: moving-points}, we view $\hat y_1, \hat y_2$ as points on $T$ sitting over $e^{i\epsilon},1$ and take a path $\eta$ on $T$ which starts at $\hat y_1$, is obtained by parallel transport around $\bdry D$ in the counterclockwise direction via a symplectic connection, and ends at $\hat y_1'$ sitting over $e^{-i\epsilon}$. The desired curve count is equivalent to the curve count in Lemma~\ref{lemma: count} below and implies Claim~\ref{claim: count}.
\end{proof}

\begin{lemma} \label{lemma: count}
The count of disks $w: F\to \C^2$ modulo reparametrizations satisfying (i) and (ii) below is $1$ modulo $2$.
\be
\item[(i)] $w(\bdry F)\subset T$ and $w(F)=D$ with degree $1$.
\item[(ii)] $w$ intersects $\hat y_2$ and $\eta$.
\ee
\end{lemma}

\begin{proof}[Proof of Lemma~\ref{lemma: count}]
Refer to Figure~\ref{fig: torus}, which depicts the Lagrangian torus $T$, where the vertical direction represents the fibers of $\widehat p: T^2\to \{|\zeta|=1\}$. The blue curves are boundaries of the $2$ curves $w_1$ and $w_2$ that pass through $\hat y_2$ by Remark~\ref{rmk: signs}, and the curve $\eta$ has slope $\tfrac{1}{2}$ and connects from the fiber $T^2_{e^{i\epsilon}}$ over $e^{i\epsilon}$ to the fiber $T^2_{e^{-i\epsilon}}$. As long as $\epsilon$ is small and $\hat y_1$ does not go to $\hat y_2$ as $\epsilon\to 0$, the intersection number between $\eta$ and $w_1\cup w_2$ is $1$ modulo $2$, which implies the lemma.
\end{proof}

\begin{figure}[ht]
	\begin{overpic}[scale=.6]{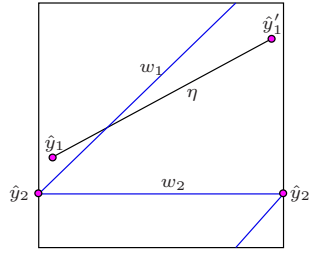}
		\put(-10,20.3){\tiny $\hat y_2$}
		\put(4,40){\tiny $\hat y_1$}
		\put(101,20.3){\tiny $\hat y_2$}
		\put(90,87){\tiny $\hat y_1'$}
		\put(60,60){\tiny $\eta$}
		\put(41.5,70){\tiny $w_1$}
		\put(50,24.5){\tiny $w_2$}
	\end{overpic}
	\caption{The sides of the torus $T$ are identified and the top and the bottom are also identified.}
	\label{fig: torus}
\end{figure}

We now claim that the count of $u$ from $\{\hat y_1,\hat y_2\}$ to $\{x_1,x_2\}$ mod $2$ does not change under a (generic) isotopy of $D_2$:  As we isotop the shape of $D_2$, the $1$-dimensional family of curves from $\{\hat y_1,\hat y_2\}$ to $\{x_1,x_2\}$ can degenerate to an $\op{ind}=0$ curve from $\{\hat y_1,\hat y_2\}$ to $\{\hat r_1,x_2\}$ and an $\op{ind}=1$ curve from $\{\hat r_1\}$ to $\{x_1\}$ (or an $\op{ind}=0$ curve from $\{\hat y_1,\hat y_2\}$ to $\{\hat t_1,x_1\}$ and an $\op{ind}=1$ curve from $\{\hat t_1\}$ to $\{x_2\}$).  Since the $\op{ind}=1$ curve types come in pairs by Remark~\ref{rmk: signs}, the mod $2$ curve count remains invariant as we pass through this bifurcation.

\s\n
{\em Case (3).} We choose $D_3$ as in Figure~\ref{fig: BC} such that the two critical values $z_1$ and $z_2$ are far apart and that $D_3$ has a long neck in the middle. A Maslov index calculation implies that, in Case (3), $\op{ind}(u)=1$ if and only if $u$ is from $\{\hat y_1,\hat y_2\}$ to $\{\check r_2,x_2\}$.

We claim that the count of $u: \dot F\to \R\times[0,1]\times \widehat W$ from $\{\hat y_1,\hat y_2\}$ to $\{\check r_2,x_2\}$ is $1$ mod $2$.  Observe that $\hat y_1,\hat y_2,\check r_2$ should be viewed as point constraints and $x_2$ as imposing no point constraints.  There are slits going to the left or down at $y_2$ (labeled $yL$ or $yD$) and slits going to the right or down at $r_2$ (labeled $rR$ or $rD$).  The slits $yL$ and $rR$ cannot be too long: if $rR$ is long, then we are effectively gluing a disk $u_1$ about $z_1$ with no constraint and a disk $u_2$ about $z_2$ with three constraints, which is a contradiction; the situation for $yL$ long is similar.  

On the other hand, if $yL$ and $rR$ are not too long, then we are effectively gluing a disk $u_1$ about $z_1$ with $2$ constraints and a disk $u_2$ about $z_2$ with $1$ constraint. In the next paragraph we sketch a model calculation for $u_1$ similar to Step 4 of Theorem~\ref{thm: count warm-up}, which implies that the count of $u_1$ is $1$ mod $2$.  The gluing of $u_1$ to $u_2$ imposes another constraint on $u_2$ and the count of $u_2$ with $2$ constraints is $1$ mod $2$ by a similar calculation. The claim then follows.

\s\n
{\em Model calculation.}  We use the notation from Step 4 of Theorem~\ref{thm: count warm-up} and consider $\mathcal{M}_J^{w_1}$, where $w_1=(w_{11},w_{12})=(e^{i(\pi-\epsilon)},1)$ with $\epsilon>0$ small. Figure~\ref{fig: modelcalc4} gives a schematic description of $\overline{\mathcal{M}}_J^{w_1}$ analogous to Figure~\ref{fig: modelcalc}.  The evaluation map $ev_J$ from Equation~\eqref{eqn: evaluation map} therefore maps $\mathcal{M}^{w_1}_J$ almost completely once around $S^1_{|z_1|=1}$. As long as $w_{01}$ and $w_{11}$ are not too close on $S^1_{|z_1|=1}$, the degree calculation implies that $\# ev_{J^\Diamond}^{-1} (w_{01})=1$ modulo $2$. 

We also note that $w_{01}$ and $w_{11}$ not being close translates to the points corresponding to $\hat y_1$ and $\check r_2$ not being close.  This implies that we cannot obtain a disk $u_1$ that passes through $\hat y_1$ and $\check r_2$ by adjusting the slit length of $rR$ and that we must use $rD$. Our model calculation corresponds to using the slit $rD$.

\begin{figure}[ht]
	\begin{overpic}[scale=0.7]{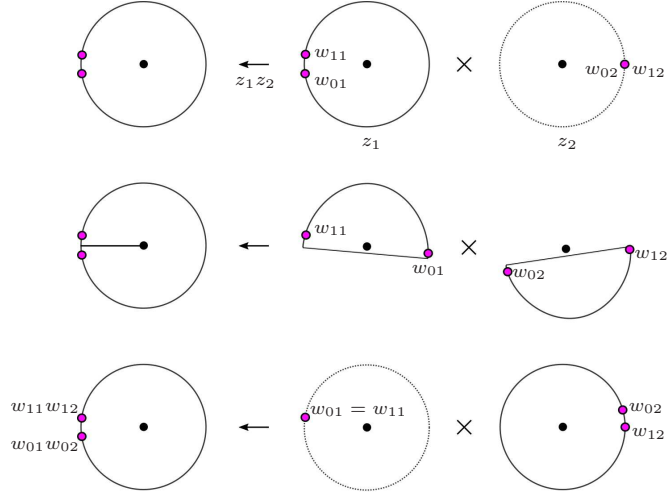}
		\put(-12,14.5){\tiny $w_{11}w_{12}$}
		\put(-12,7.7){\tiny $w_{01}w_{02}$}
		\put(41.7,14){\tiny $w_{01}=w_{11}$}
		\put(99,14.5){\tiny $w_{02}$}
		\put(99.5,10){\tiny $w_{12}$}
		\put(60,39.3){\tiny $w_{01}$}
		\put(42.5,46.5){\tiny $w_{11}$}
		\put(78,38.6){\tiny $w_{02}$}
		\put(100,42){\tiny $w_{12}$}
		
		\put(42.5,73){\tiny $w_{01}$}
		\put(42.5,78){\tiny $w_{11}$}
		\put(91.1,75){\tiny $w_{02}$}
		\put(99.5,75){\tiny $w_{12}$}
		
		\put(28.3,73){\tiny $z_1z_2$}
		\put(51,62){\tiny $z_1$}
		\put(86.1,62){\tiny $z_2$}
	\end{overpic}
	\caption{A schematic description of $\overline{\mathcal{M}}_J^{w_1}$ analogous to Figure~\ref{fig: modelcalc} but with a different $w_1$.}
	\label{fig: modelcalc4}
\end{figure}

\s\n
{\em Case (4).} Consider $D_4$ as in Figure~\ref{fig: D}. In Case (4), $\op{ind}(u)=1$ if and only if $u$ is from $\{\hat y_1,\hat y_2\}$ to $\{\check t_2,x_1\}$.

We claim that the count of $u: \dot F\to \R\times[0,1]\times \widehat W$ from $\{\hat y_1,\hat y_2\}$ to $\{\check t_2,x_1\}$ is $1$ mod $2$.  Refer to Figure~\ref{fig: D} for $D_4$. The Fredholm index calculations are analogous to those of Case (3).  Consider the slits $BR, PR$ of $D_4$ that start at $t_2, y_2$, respectively, and initially go to the right.   If both slits $BR, PR$ are too short, then we are effectively gluing a disk $u_1$ about $z_1$ with no constraint and a disk $u_2$ about $z_2$ with three constraints, which is a contradiction.  If $PR$ is too long, then $u_2$ about $z_2$ will have no constraints, which is also a contradiction. Hence $PR$ is short and $BR$ is long.  Also note that $BR$ must be long but cannot be so long that its other endpoint gets too close to $z_2$, in which case $u_1$ about $z_1$ will also have no constraints. This means that we are effectively gluing $u_1$ about $z_1$ with $1$ constraint corresponding to $\hat y_1$ and $u_2$ about $z_2$ with $2$ constraints corresponding to $\hat y_2$ and $\check t_2$.  Calculations similar to those of Case (3) imply the claim.

\s\n
{\em Second Calculation.} We show that 
$$d \{\check y_1, z_2\} = \hbar (\{\check r_2, x_2\}+ \{\check t_2,x_1\}) \mbox{ mod } 2.$$
As before, the loops in $D$ that are concatenations of an arc each of 
$$\gamma_1,\gamma_2,  \sigma_{BP,3}(\gamma_1), \sigma_{BP,3}(\gamma_2)$$
and that switch at $y_1, z_2$ are given by:
$$y_1\to r_1,r_2,x_1 \to z_2\to t_1,t_2,x_2\to y_1,$$
but the only loops that can possibly bound domains with nonnegative weights are:
\be
\item $y_1\to r_2\to z_2\to x_2\to y_1$,
\item $y_1\to x_1\to z_2\to t_2\to y_1$.
\ee
They both bound quadrilaterals and by Lemma~\ref{lemma: model calc} (1) contributes $\{\check r_2, x_2\}$ and (2) contributes $\{\check t_2,x_1\}$ mod $2$ to the differential. This completes the proof of the second calculation and the proof of Theorem~\ref{thm: BP}.
\end{proof}

\end{document}